\documentclass[12pt, a4paper]{amsart}
\usepackage{amsmath}
\usepackage{bm}
\usepackage{geometry,amsthm,graphics,tabularx,amssymb,shapepar,color}
\usepackage{amscd}
\usepackage[all]{xypic}
\allowdisplaybreaks

\newcommand{\bfk}{{\bf  k}}
\newcommand{\bfd}{{\bf d}}

\newcommand{\ft}{\mathfrak t}
\newcommand{\ff}{\mathfrak f}

\newcommand{\fg}{\mathfrak g}
\newcommand{\fh}{\mathfrak h}

\newcommand{\al}{\alpha}

\newcommand{\wt}{\widetilde}
\newcommand{\wh}{\widehat}

\newcommand{\Der}{\mathrm{Der}}

\newcommand{\ot}{\otimes}
\newcommand{\dwz}{\delta\(\frac{w}{z}\)}
\newcommand{\wpw}{w\frac{\partial}{\partial w}}
\newcommand{\pw}{\frac{\partial}{\partial w}}

\newcommand{\CA}{\mathcal{A}}

\newcommand{\CL}{\mathcal{L}}

\newcommand{\CE}{\mathcal{E}}
\newcommand{\CY}{\mathcal{Y}}

\newcommand{\C}{\mathbb{C}}

\newcommand{\N}{\mathbb N}
\newcommand{\Z}{\mathbb Z}
\newcommand{\rK}{\mathrm{K}}
\newcommand{\rD}{\mathrm{D}}
\newcommand{\rd}{\mathrm{d}}
\newcommand{\rk}{\mathrm{k}}

\newcommand{\ba}{\begin {eqnarray}}
\newcommand{\ea}{\end {eqnarray}}
\newcommand{\baa}{\begin {eqnarray*}}
\newcommand{\eaa}{\end {eqnarray*}}
\newcommand{\be}{\begin {equation}}
\newcommand{\ee}{\end {equation}}
\newcommand{\bee}{\begin {equation*}}
\newcommand{\eee}{\end {equation*}}

\newcommand{\U}{\mathcal{U}}

\newcommand{\te}[1]{\textnormal{{#1}}}

\theoremstyle{Theorem}

\theoremstyle{Theorem}

\newtheorem{thm}{Theorem}[section]
\newtheorem{cort}[thm]{Corollary}
\newtheorem{lemt}[thm]{Lemma}
\newtheorem{prpt}[thm]{Proposition}

\newtheorem{remt}[thm]{Remark}

\theoremstyle{Theorem}

\theoremstyle{Theorem}

\theoremstyle{Plain}

\theoremstyle{Definition}

\newtheorem{dfnt}[thm]{Definition}

\def\({\left(}

\def\){\right)}

\def \<{{\langle}}
\def \>{{\rangle}}

\numberwithin{equation}{section}
\title[Toroidal EALA and VA]{Toroidal extended affine Lie algebras and vertex algebras}

\author{Fulin Chen}
\address{School of Mathematical Sciences, Xiamen University,
 Xiamen, 361005, China} \email{chenf@xmu.edu.cn }
  \thanks{$^1$Partially supported by China NSF grant (No.11971397) and the Fundamental
Research Funds for the Central Universities (Nos.20720190069, 20720200067).}

\author{Haisheng Li}\address{Department of Mathematical Sciences,
Rutgers University, Camden, NJ 08102, USA} \email{hli@camden.rutgers.edu}

\author{Shaobin Tan}
 \address{School of Mathematical Sciences, Xiamen University,
 Xiamen, 361005, China} \email{tans@xmu.edu.cn}
 \thanks{$^2$Partially supported by China NSF grants (No.11971397)}


\begin{document}

\begin{abstract}
In this paper, we study nullity-$2$ toroidal extended affine Lie algebras in the context of
 vertex algebras and their $\phi$-coordinated modules. Among the main results, we introduce a variant of
 toroidal extended affine Lie algebras, associate vertex algebras to the variant Lie algebras, and
 establish a canonical connection between modules for  toroidal extended affine Lie algebras and
 $\phi$-coordinated modules for these vertex algebras. Furthermore, by employing some results of Billig,
 we obtain an explicit realization of irreducible modules for the variant Lie algebras.
\end{abstract}
\maketitle

\section{Introduction}
Extended affine Lie algebras (of positive nullity) are a family of infinite-dimensional Lie algebras,
which are natural generalizations of affine Kac-Moody Lie algebras.
Of great importance in the study of these Lie algebras are those called toroidal extended affine Lie algebras.
Let $\fg$ be a finite-dimensional simple Lie algebra (over $\C$).
Denote by ${\mathcal{R}}_N$ the Laurent polynomial ring $\C[t_0^{\pm 1},t_1^{\pm},\dots,t_N^{\pm}]$
(with $N$ a positive integer).
By definition,  the toroidal Lie algebra, denoted by $\ft({\fg})$ in this paper, is  a universal central extension of
the multi-loop Lie algebra ${\mathcal{R}}_N\ot \fg$, and furthermore the full toroidal Lie algebra is
the semi-product extension of the toroidal Lie algebra
by the (full) derivation Lie algebra ${\mathcal{D}}$ of the torus ${\mathcal{R}}_N$.
While it is ``perfect" (for various reasons), the full toroidal Lie algebra is {\em not} an extended affine Lie algebra
in the precise sense of \cite{BGK}.
Replacing ${\mathcal{D}}$ with its subalgebra ${\mathcal{D}}_{\rm div}$, consisting of
what were called divergence-zero derivations, one gets a Lie algebra,
which in the precise sense is an extended affine Lie algebra,
commonly called the toroidal extended affine Lie algebra.

In the study of this family of Lie algebras, one of the main focuses has been
 the classification of irreducible modules and their (explicit) vertex operator realizations.
The study on vertex operator realization of toroidal Lie algebras was initiated by Rao and Moody
(see \cite{EM}, \cite{MRY}),
and  since then extended affine Lie algebras and their relations with vertex algebras
have been extensively studied in literature.
The first precise connection with vertex algebras was obtained by Berman, Billig and  Szmigielski
(see \cite{BB}, \cite{BBS}).
 Since then Billig has intensively studied representations of toroidal extended affine Lie algebras
in the context of vertex algebras and modules.  In particular, Billig proved that
 full toroidal Lie algebras are vertex Lie algebras in the sense of \cite{DLM1}
and then associated vertex algebras to these Lie algebras. He also gave a realization of irreducible modules
for full toroidal Lie algebras in terms of modules for certain concrete vertex operator algebras.
Furthermore, Billig (see \cite{B2}) studied modules for toroidal extended affine Lie algebras
by employing the natural embedding.
A new phenomenon discovered therein is that toroidal extended affine Lie algebras
are {\em not} vertex Lie algebras in the precise sense (at least in the obvious way).
 Nevertheless, by making use of the affinization of vertex algebras,
Billig successfully obtained an explicit realization of irreducible modules
for toroidal extended affine Lie algebras in terms of modules for certain vertex operator algebras.
In the nullity-$2$ case, the construction of Billig was later generalized in \cite{CLT1}, which leads to a classification of
irreducible integrable modules for toroidal extended affine Lie algebras (see \cite{CLT2}).

In the present paper, on the basis of Billig's work we explore natural connections between
toroidal extended affine Lie algebras and vertex algebras with different perspectives.
We here concentrate ourselves to the nullity-$2$ case (with $N=1$).  An important gadget in this study
is the theory of $\phi$-coordinated modules for vertex algebras, which was developed in \cite{L}.
As a key ingredient, we introduce a variation of the toroidal extended affine Lie algebras, which are shown to be
vertex Lie algebras. Then we associate vertex algebras  to the newly introduced Lie algebras naturally.
Among the main results, we establish a canonical connection between restricted modules for the nullity-$2$
toroidal extended affine Lie algebras and $\phi$-coordinated modules for the vertex algebras associated to the
 variants  of the toroidal extended affine Lie algebras.  By using Billig's results
we obtain an explicit realization of irreducible modules for the variant Lie algebras.
Furthermore, by using a result of Zhu (and Huang) and a result of Lepowsky we obtain an explicit realization of irreducible modules
for the  toroidal extended affine Lie algebras,
which recovers a result of \cite{B2}.

To better explain the motivation and idea we continue this introduction with some technical details.
We start with a basic fact for vertex algebra modules.
Let $V$ be a vertex algebra and let  $(W,Y_W)$ be a $V$-module.
For $u,v\in V$, we have
\begin{align}\label{module-comm-B}
[Y_W(u,z_1),Y_W(v,z_2)]
=\sum_{j\ge 0}Y_W(u_jv,z_2)\frac{1}{j!}\left(\frac{\partial}{\partial z_2}\right)^jz_1^{-1}\delta\left(\frac{z_2}{z_1}\right)
\end{align}
(the {\em Borcherds commutator formula}).
For the full toroidal Lie algebras (like other vertex Lie algebras such as untwisted affine Lie algebras
and the Virasoro algebra), the commutator of two (suitably formulated) generating functions
can be expressed in the form:
$$\sum_{r=0}^{k}A_r(z_2)\frac{1}{r!}\left(\frac{\partial}{\partial z_2}\right)^rz_1^{-1}\delta\left(\frac{z_2}{z_1}\right).$$
It follows (from a conceptual result in \cite{Li-local}) that vertex algebras (and modules)
can be associated to the full toroidal Lie algebras.

For the toroidal extended affine Lie algebras, as it was noticed by Billig there is a discrepancy,
which causes the nonexistence of a direct correspondence between toroidal extended affine Lie algebras
and vertex algebras or their modules.
In this paper, as the first step we suitably modify the generating functions,
so that the commutator of two generating functions has the form:
$$\sum_{r=0}^kB_r(z_2)\frac{1}{r!}\left(z_2\frac{\partial}{\partial z_2}\right)^r\delta\left(\frac{z_2}{z_1}\right).$$
This still does not match the commutator formula for vertex algebras and modules.
However,  such commutator relations are in a full agreement
with the commutator formula in the theory of what were called $\phi$-coordinated modules
in \cite{L} for vertex algebras as we describe next.

Note that a module $(W,Y_W)$ for a vertex algebra $V$ can be defined by
using the weak commutativity: For $u,v\in V$, there exists a nonnegative integer $k$ such that
\begin{align}\label{locality-module-intro}
(z_1-z_2)^k[Y_W(u,z_1),Y_W(v,z_2)]=0,
\end{align}
a consequence of Borcherds' commutator formula (\ref{module-comm-B}), together with
the following weak associativity
\begin{align}
z^kY_W(Y(u,z)v,z_2)=\left((x_1-x)^kY_W(u,x_1)Y(v,x)\right)|_{x_1=x+z}.
\end{align}

In the theory of $\phi$-coordinated modules for vertex algebras,
$\phi$ is what it was called an associate
of the $1$-dimensional additive formal group (law) $F(x,y)=x+y$,
which by definition is a formal series $\phi(x,z)\in \C((x))[[z]]$, satisfying the condition
$$\phi(x,0)=x,\quad  \phi(\phi(x,z_1),z_2)=\phi(x,z_1+z_2).$$
It was proved (see \cite{L}) that for every $p(x)\in \C((x))$, $\phi(x,z)=e^{zp(x)\frac{d}{dx}}(x)$ is an associate
of $F(x,y)$ and every associate is of this form.
As two particular cases,  we have $\phi(x,z)=e^{z\frac{d}{dx}}(x)=x+z$ (the formal group itself) with $p(x)=1$, and
$\phi(x,z)=e^{zx\frac{d}{dx}}(x)=xe^z$ with $p(x)=x$.

For a vertex algebra $V$, a $\phi$-coordinated $V$-module $(W,Y_W)$
can be defined by using the weak commutativity
(for $u,v\in V$, there exists a nonnegative integer $k$ such that (\ref{locality-module-intro}) holds)
together with the following weak $\phi$-associativity
\begin{align}
(\phi(x,z)-x)^kY_W(Y(u,z)v,x)=\left((x_1-x)^kY_W(u,x_1)Y(v,x)\right)|_{x_1=\phi(x,z)}.
\end{align}
Let  $\phi(x,z)=xe^z$. If $(W,Y_W^{\phi})$ is a $\phi$-coordinated module for a vertex algebra $V$, then
for $u,v\in V$, we have
\begin{align}
[Y_W^{\phi}(u,z_1),Y_W^{\phi}(v,z_2)]
=\sum_{j\ge 0}Y_W^{\phi}(u_jv,z_2)\frac{1}{j!}\left(z_2\frac{\partial}{\partial z_2}\right)^j\delta\left(\frac{z_2}{z_1}\right).
\end{align}
This indicates that modules for the toroidal extended affine Lie algebra, denoted by $\wh{\ft}(\fg,\mu)$ in this paper,
may be $\phi$-coordinated modules for some vertex algebra.
It is this fact that motivated us to introduce the variant Lie algebra $\wh{\ft}(\fg,\mu)^{o}$.

Assume that $V$ is a vertex operator algebra in the sense of \cite{FLM} and \cite{FHL}.
By using a result of Zhu (and Huang) and a result of Lepowsky we show (cf. \cite{li-F})
that there is a canonical category isomorphism between the module category and the $\phi$-coordinated category.
Then using this isomorphism we obtain a realization of irreducible $\wh{\ft}(\fg,\mu)$-modules
of nonzero levels in terms of  modules for certain concrete vertex algebras,  recovering a result of Billig.

This paper is organized as follows: In Section 2, we review the full toroidal Lie algebra
and the toroidal extended affine Lie algebra. In Section 3, we introduce a variation $\wh{\ft}(\fg,\mu)^{o}$ of
the toroidal extended affine Lie algebra $\wh{\ft}(\fg,\mu)$.
In Section 4, we study vertex algebras $V_{\wh{\ft}(\fg,\mu)^{o}}(\ell)$,
$V_{\wh{\ft}(\fg,\mu)^{o}}^{\rm int}(\ell)$, and their $\phi$-coordinated modules.
In Section 5, we give a realization of irreducible $\wh{\ft}(\fg,\mu)^{o}$-modules of nonzero levels
using irreducible modules for certain concrete vertex algebras and
determine all bounded irreducible $\wh{\ft}(\fg,\mu)^{o}$-modules.
 In Section 6, using certain results in vertex operator algebra theory
 we give a realization of irreducible $\wh{\ft}(\fg,\mu)$-modules of nonzero levels
 in terms of  modules for certain concrete vertex algebras.
\\

The following is a list of notations we frequently use in this paper:

\begin{align*}
\N: \quad \quad & \te{The set of nonnegative integers}\\
\Z_{+}: \quad \quad& \te{The set of positive integers}\\
\Z^{\times}: \quad \quad& \te{The set of nonzero integers}\\
\fg: \quad \quad& \te{A finite-dimensional simple Lie algebra}\\
\fh:  \quad \quad& \te{A Cartan subalgebra of }\fg\\
\Delta: \quad \quad& \te{The root system of }(\fg,\fh)\\
\mathcal{R}: \quad \quad &\te{The Laurent polynomial ring }\C[t_0^{\pm 1},t_1^{\pm 1}]\\
\mathcal{K}: \quad \quad &\te{A quotient space of K\"{a}hler differentials of }\C[t_0^{\pm 1},t_1^{\pm 1}]\\
\mathcal{D}: \quad \quad &\te{The derivation Lie algebra of }\C[t_0^{\pm 1},t_1^{\pm 1}]\\
\ft(\fg): \quad \quad &\te{The toroidal Lie algebra }{\mathcal{R}}\ot \fg+{\mathcal{K}}\\
{\mathcal{T}}(\fg,\mu): \quad \quad &\te{The full (rank 2) toroidal Lie algebra}\\
\wt{\ft}(\fg,\mu): \quad \quad &\te{The (rank 2) toroidal EALA (Lie algebra)}\\
\wh{\ft}(\fg,\mu): \quad \quad &\te{A subalgebra of }\wt{\ft}(\fg,\mu)\\
\mathcal{D}_{\rm div}: \quad \quad &\te{The Lie algebra of divergence-zero derivations of }\C[t_0^{\pm 1},t_1^{\pm 1}]\\
\mathcal{D}_{\rm div'}: \quad \quad &\te{A Lie subalgebra of } {\mathcal{D}}\\
\wt{\ft}(\fg,\mu)^{o}: \quad \quad &\te{A subalgebra of the full toroidal Lie algebra }{\mathcal{T}}(\fg,\mu) \\
\wh{\ft}(\fg,\mu)^{o}: \quad \quad &\te{A Lie subalgebra of }\wt{\ft}(\fg,\mu)^{o}.
\end{align*}

\section{Toroidal extended affine  Lie algebras}
In this section, we briefly review the nullity-$2$ full toroidal Lie algebras and
 toroidal extended affine Lie algebras.

\subsection{Full toroidal Lie algebras}
We here follow \cite{B1} and \cite{BB} to present the nullity-$2$ full toroidal Lie algebras.
Set
\begin{align}
\mathcal R=\C[t_0,t_0^{-1}, t_1,t_1^{-1}],
\end{align}
the ring of Laurent polynomials in (commuting) variables $t_0$ and $t_1$. Let
$$\Omega_\mathcal R^1=\mathcal R dt_0\oplus \mathcal R dt_1$$
be the space of $1$ forms on $\mathcal R$, which is a free $\mathcal R$-module. Set
\begin{align}
\rk_0=t_0^{-1}\rd t_0, \   \  \  \   \rk_1=t_1^{-1}\rd t_1,
\end{align}
which also form a basis of $\Omega_\mathcal R^1$ over $\mathcal R$.
For $f\in {\mathcal R}$, define
$$d(f)=\frac{\partial f}{\partial t_0}dt_0+\frac{\partial f}{\partial t_1}dt_1=
t_0\frac{\partial f}{\partial t_0}\rk_0+t_1\frac{\partial f}{\partial t_1}\rk_1 \in \Omega_{\mathcal R}^1.$$
Set $d({\mathcal R})=\{ d(f)\, |\, f\in \mathcal R\}$ and form the vector space
 \begin{align}
 \mathcal K=\Omega_{\mathcal R}^1/ d({\mathcal R}).
 \end{align}
 The following relations hold in $ \mathcal K$ for $m_0, m_1\in \Z$:
\begin{align}\label{relainK}
m_0t_0^{m_0}t_1^{m_1}\rk_0+m_1t_0^{m_0}t_1^{m_1}\rk_1=0.
\end{align}

For  $m_0, m_1\in \Z$, set
\begin{equation}
\rk_{m_0,m_1}=\begin{cases}\frac{1}{m_1} t_0^{m_0}t_1^{m_1}\rk_0\ &\te{if}\ m_1\ne 0,\\
 -\frac{1}{m_0} t_0^{m_0}\rk_1\ &\te{if}\ m_1=0,\ m_0\ne 0,\\
0\ &\te{if}\ m_0=m_1=0.\end{cases}
\end{equation}
Then the set
\begin{align}
\mathbb B_{\mathcal{K}}:=\{\rk_0, \rk_1\}\cup \{\rk_{m_0, m_1}\mid m_0,m_1\in \Z\ \te{with}\ (m_0,m_1)\ne (0,0)\}
\end{align}
is a basis of $\mathcal K$.
We can also present $\mathbb B_{\mathcal{K}}$ slightly differently as
\begin{align}\label{basisK}
\mathbb B_{\mathcal{K}}=\{\rk_0\}\cup \{t_0^{m_0}\rk_1,\ \rk_{m_0,m_1}\mid m_0\in \Z,\  m_1\in \Z^{\times}\}.
\end{align}

We have the following straightforward fact:

\begin{lemt}\label{relation-K}
The following relations hold in $\mathcal K$:
\begin{align}\label{relationK}
at_0^{m}t_1^{n}\rk_0+bt_0^{m}t_1^{n}\rk_1
=(an-bm)\rk_{m,n}+\delta_{m,0}\delta_{n,0}(a\rk_0+b\rk_1)
\end{align}
for $a,b\in \C,\ m,n\in \Z$. In particular, we have
\begin{align*}
 t_0^{m}\rk_1=-m\rk_{m,0}+\delta_{m,0}\rk_1\quad\text{for }m\in \Z.
\end{align*}
\end{lemt}

On the other hand, set
\begin{align}
\mathcal{D}=\Der (\mathcal R),
\end{align}
the derivation Lie algebra of  $\mathcal R$. Especially, set
\begin{align}
\rd_0=t_0\frac{\partial}{\partial t_0},\   \   \   \   \rd_1=t_1\frac{\partial}{\partial t_1}.
\end{align}
Then $\mathcal{D}$ is a free $\mathcal R$-module with basis $\{\rd_0, \rd_1\}$.
The space $\mathcal K$ is a $\mathcal D$-module with
\begin{align}
\psi(f\rd g)=\psi(f)\rd g+f\rd\psi(g)\   \   \   \  \mbox{ for }\psi\in \mathcal{D},\  f,g\in \mathcal R.
\end{align}

Let $\fg$ be any Lie algebra equipped with a non-degenerate symmetric invariant bilinear form $\<\cdot,\cdot\>$.
 Form a central extension of the double loop algebra $\mathcal R\ot \fg$
 \begin{align*}
 \ft(\fg)=\(\mathcal R\ot \fg\)\oplus \mathcal K,
 \end{align*}
called the {\em toroidal Lie algebra}, where  $\mathcal K$ is central and
\begin{equation}\begin{split}\label{fre1}
 &[t_0^{m_0}t_1^{m_1}\ot u, t_0^{n_0}t_1^{n_1}\ot v]\\
 =\ &t_0^{m_0+n_0}t_1^{m_1+n_1}\ot [u,v]
 +\< u,v\>
 (m_0t_0^{m_0+n_0}t_1^{m_1+n_1}\rk_0+m_1t_0^{m_0+n_0}t_1^{m_1+n_1}\rk_1)
 \end{split}
 \end{equation}
for $u,v\in \fg$, $m_0,n_0, m_1,n_1\in \Z$.
When $\fg$ is a finite-dimensional simple Lie algebra, the toroidal Lie algebra $\ft(\fg)$
is a universal central extension  (see \cite{MRY}),
which is commonly called the {\em nullity-$2$ toroidal Lie algebra}.

Lie algebra $\mathcal D$ acts on the Lie algebra $\mathcal R\ot \fg$
as a derivation Lie algebra by
$$\psi (a\otimes u)=\psi(a)\otimes u\   \   \   \mbox{ for }\psi\in \mathcal{D},\  a\in \mathcal R,\ u\in \fg.$$
Furthermore, $\mathcal D$ acts on the toroidal Lie algebra $\ft(\fg)$ as a derivation Lie algebra.
In particular, (considering the semi-product Lie algebra $\mathcal R\ot \fg\rtimes \mathcal{D}$) we have
\begin{align}\label{fre2}
[t_0^{m_0}t_1^{m_1} \rd_i, t_0^{n_0}t_1^{n_1}\ot u]=n_i( t_0^{m_0+m_1}t_1^{n_0+n_1}\ot u)
\end{align}
for $u\in \fg$, $m_0,n_0,m_1,n_1\in \Z,\ i\in \{0,1\}$, and
\begin{align}\label{fre3}
[t_0^{m_0}t_1^{m_1} \rd_i, t_0^{n_0}t_1^{n_1}\rk_j]
=n_i t_0^{m_0+n_0}t_1^{m_1+n_1}\rk_j+\delta_{i,j}\sum_{r=0, 1} m_r t_0^{m_0+n_0}t_1^{m_1+n_1}\rk_r
\end{align}
in $\mathcal K\rtimes \mathcal D$ for $m_0,n_0,m_1,n_1\in \Z,\ i,j\in \{0,1\}$.

The following notion can be found in  \cite{B2}:

\begin{dfnt}
{\em Let $\mu$ be a complex number, which is fixed throughout this paper.
The nullity-$2$ {\em full toroidal Lie algebra} is the Lie algebra
\begin{align}
\mathcal{T}(\fg,\mu):=\mathfrak t(\fg)\oplus \mathcal D=\(\mathcal R\ot \fg\)\oplus \mathcal K\oplus \mathcal D,
\end{align}
where in addition to the relations \eqref{fre1}, \eqref{fre2} and \eqref{fre3}, the following relations are postulated:
\begin{equation}\begin{split}\label{fre4}
&[t_0^{m_0}t_1^{m_1}\rd_i , t_0^{n_0}t_1^{n_1}\rd_j ]=\ n_it_0^{m_0+n_0}t_1^{m_1+n_1}\rd_j
-m_jt_0^{m_0+n_0}t_1^{m_1+n_1}\rd_i\\
&\qquad -\mu m_jn_i(m_0t_0^{m_0+n_0}t_1^{m_1+n_1}\rk_0+m_1t_0^{m_0+n_0}t_1^{m_1+n_1}\rk_1)
\end{split}\end{equation}
for $m_0,n_0,m_1,n_1\in \Z,\ i,j\in \{0,1\}$. }
\end{dfnt}

\begin{remt}\label{full-Lie-subalgebras}
{\em  Note that with \eqref{fre1}-\eqref{fre4} it can be readily seen that for any Lie subalgebra $\mathbb{D}$
of $\mathcal{D}$, the subspace $\mathcal{R}\ot \fg+\mathcal{K}+\mathbb{D}$
  of $\mathcal{T}(\fg,\mu)$ is a Lie subalgebra.   }
\end{remt}

\begin{dfnt}\label{de:zgrading}
{\em We fix the particular $\Z$-grading on $\mathcal{T}(\fg,\mu)$ given by the adjoint action of $-\rd_0$:
\begin{align}
\mathcal{T}(\fg,\mu)=\bigoplus_{m\in \Z}\mathcal{T}(\fg,\mu)_{(m)},
\end{align}
where
$\mathcal{T}(\fg,\mu)_{(m)}=\{x\in \mathcal{T}(\fg,\mu)\mid [\rd_0,x]=-m x\}.$}
\end{dfnt}

Set
\begin{align}\label{defgi}
&\widetilde\fg_0=\C[t_0,t_0^{-1}]\ot \fg+ \C\rk_0+ \C\rd_0,\\
&\widetilde\fg_1=\C[t_1,t_1^{-1}]\ot \fg+ \C\rk_1+ \C\rd_1,
\end{align}
both of which are subalgebras of $\mathcal{T}(\fg,\mu)$,  canonically isomorphic to
the affine Lie algebra $\widetilde{\fg}=\C[t,t^{-1}]\ot \fg+\C \rk+\C \rd$.
As a convention, for $X\in  \widetilde\fg_1$ (resp. $\widetilde\fg_0$),  we write
$t_0^m X$ (resp. $t_1^m X$) for the corresponding elements of $\mathcal{T}(\fg,\mu)$.

For convenience, we formulate the following two technical lemmas,
 which follow straightforwardly from (\ref{fre3})
and (\ref{fre4}) together with Lemma \ref{relation-K}:

\begin{lemt}\label{abdk}
The following relations hold in  $\mathcal{T}(\fg,\mu)$:
\begin{align}
&[at_0^{m_0}t_1^{m_1}\rd_0+bt_0^{m_0}t_1^{m_1}\rd_1, \rk_{n_0,n_1}]\\
=\ &  (a(m_0+n_0)+b(m_1+n_1))\rk_{m_0+n_0,m_1+n_1}+\delta_{m_0+n_0,0}\delta_{m_1+n_1,0}(b\rk_0-a\rk_1)\nonumber
\end{align}
for $a,b\in \C,\ m_0,m_1,n_0,n_1\in \Z$. 
\end{lemt}

\begin{lemt}\label{abdd}
The following relations hold in  $\mathcal{T}(\fg,\mu)$:
\begin{align}
&[a_0t_0^{m_0}t_1^{m_1}\rd_0+a_1t_0^{m_0}t_1^{m_1}\rd_1,b_0t_0^{n_0}t_1^{n_1}\rd_0+b_1t_0^{n_0}t_1^{n_1}\rd_1]\\
=\ &  (b_0(a_0n_0+a_1n_1)-a_0(b_0m_0+b_1m_1))t_0^{m_0+n_0}t_1^{m_1+n_1}\rd_0\nonumber\\
& +(b_1(a_0n_0+a_1n_1)-a_1(b_0m_0+b_1m_1))t_0^{m_0+n_0}t_1^{m_1+n_1}\rd_1\nonumber\\
&-\mu (a_0n_0+a_1n_1)(b_0m_0+b_1m_1)\cdot (m_0n_1-m_1n_0)\rk_{m_0+n_0,m_1+n_1}\nonumber\\
&-\mu (a_0n_0+a_1n_1)(b_0m_0+b_1m_1)\cdot\delta_{m_0+n_0,0}\delta_{m_1+n_1,0}(m_0\rk_0+m_1\rk_1)\nonumber
\end{align}
for $a_0,a_1,b_0,b_1\in \C,\ m_0,m_1,n_0,n_1\in \Z$.
\end{lemt}

\subsection{Toroidal extended affine Lie algebra $\wt{\ft}(\fg,\mu)$}
Here, we follow  \cite{BGK,B2} to present the toroidal extended affine Lie algebra.
Set
\begin{align}
\mathcal{D}_{\rm div}=\{f_0\rd_0+f_1\rd_1\mid f_0,f_1\in \mathcal R,\ \rd_0(f_0)+\rd_1(f_1)=0\}\subset \mathcal D,
\end{align}
the Lie algebra of {\em divergence-zero  derivations,} or namely {\em skew derivations} on $\mathcal R$.
For $m_0,m_1\in \Z$, set
\begin{equation}
\tilde\rd_{m_0,m_1}=
m_0t_0^{m_0}t_1^{m_1}\rd_1-m_1t_0^{m_0}t_1^{m_1}\rd_0,
\end{equation}
which is an element of $\mathcal{D}_{\rm div}$.
A simple fact is that the set
\begin{align}
\mathbb B:=\{\rd_0,\rd_1\}\cup \{ \tilde\rd_{m_0,m_1}\mid (m_0,m_1)\in (\Z\times \Z)\backslash \{ (0,0)\}\}
\end{align}
is a $\C$-basis  of $\mathcal{D}_{\rm div}$.
With the relation
\begin{align}\label{d1-0-relations}
\tilde\rd_{m,0}=mt_0^m\rd_1\ \te{ for }m\in \Z,
\end{align}
we can also present $\mathbb B$  as (cf. (\ref{basisK}))
\begin{align}\label{basisDiv}
\mathbb B=\{\rd_0\}\cup\{ t_0^{m_0}\rd_1,\  \tilde\rd_{m_0,m_1}\mid m_0\in \Z,\ m_1\in \Z^{\times}\}.
\end{align}

Record the following relations in $\mathcal{D}_{\rm div}$:
\begin{align}
&[\rd_i,\tilde{\rd}_{m_0,m_1}]=m_i\tilde{\rd}_{m_0,m_1},\label{relations-divergence1}\\
&[\tilde{\rd}_{m_0,m_1},\tilde{\rd}_{n_0,n_1}]=(m_0n_1-m_1n_0)\tilde{\rd}_{m_0+n_0,m_1+n_1}\label{relations-divergence2}
\end{align}
for $i\in \{0,1\},\ m_0, m_1, n_0, n_1\in \Z$.

\begin{remt}
{\em It follows from the Lie bracket relations (\ref{relations-divergence1}) and (\ref{relations-divergence2}) that
$\mathcal{D}_{\rm div}$ is indeed a subalgebra of $\mathcal{D}$. Furthermore, we have
\begin{align}
\mathcal{D}_{\rm div}^{(1)}:=[\mathcal{D}_{\rm div},\mathcal{D}_{\rm div}]=\te{Span}\{ \tilde{\rd}_{m,n}\mid m,n\in \Z\}
\end{align}
and $\mathcal{D}_{\rm div}$ contains the following subalgebras:
\begin{align}
\mathcal{D}_{\rm div}^{(1)},\quad
\mathcal{D}_{\rm div}^{(1)}\rtimes \C \rd_0, \quad \mathcal{D}_{\rm div}^{(1)}\rtimes \C \rd_1, \quad
\mathcal{D}_{\rm div}^{(1)}\rtimes (\C \rd_0+\C\rd_1)=\mathcal{D}_{\rm div}.
\end{align}}
\end{remt}

\begin{dfnt}
{\em Set
\begin{align}
&\wt\ft(\fg,\mu)=(\mathcal R\ot \fg)+ \mathcal K+ \mathcal{D}_{\rm div},\\
&\wh\ft(\fg,\mu)=(\mathcal R\ot \fg)+ \mathcal K+ \mathcal{D}_{\rm div}^{(1)}\rtimes \C\rd_1,
\end{align}
both of which are $\Z$-graded subalgebras of $\mathcal{T}(\fg,\mu)$ (recall Remark \ref{full-Lie-subalgebras}).}
\end{dfnt}

\begin{remt}
{\em Assume that $\fg$ is a finite-dimensional simple Lie algebra.
A special feature of  $\wt\ft(\fg,\mu)$ is that it admits
a suitable non-degenerate symmetric invariant bilinear form, whereas
 $\mathcal{T}(\fg,\mu)$ does not.
This makes $\wt{\mathfrak{t}}(\fg,\mu)$ an extended affine Lie algebra in the sense of \cite{BGK},
which is commonly known as the {\em toroidal extended affine Lie algebra of nullity $2$.}}
\end{remt}

Using Lemmas \ref{relation-K}, \ref{abdk} and \ref{abdd} we straightforwardly have:
\begin{align}
\label{ere1}&[t_0^{m_0}t_1^{m_1}\ot u, t_0^{n_0}t_1^{n_1}\ot v]=t_0^{m_0+n_0}t_1^{m_1+n_1}\ot [u,v]\\
&+\< u,v\>\(
 (m_0n_1-m_1n_0)\rk_{m_0+n_0,m_1+n_1}+\delta_{m_0+n_0,0}\delta_{m_1+n_1,0}(m_0\rk_0+m_1\rk_1)\),\notag
 \end{align}
 \begin{align}
\label{ere2}&[\tilde\rd_{i,m}, t_0^jt_1^n\ot u]=(ni-mj) t_0^{i+j}t_1^{m+n}\ot u,\quad [\tilde\rd_{i,m}, t_0^j\rk_1]=m j^2\rk_{i+j,m},\\
\label{ere3}& [t_0^i\rd_1,t_0^jt_1^n\ot u]
=n t_0^{i+j}t_1^n\ot u,\quad [t_0^i\rd_1,\rk_{j,n}]=n\rk_{i+j,n},\\
\label{ere4}&[\tilde\rd_{i,m},\rk_{j,n}]=(in-mj)\rk_{i+j,m+n}+\delta_{m+n,0}\delta_{i+j,0}(i\rk_0+m\rk_1),\\
\label{ere5}&[t_0^i\rd_1,t_0^j\rk_1]=i\delta_{i+j,0}\rk_0,\   \   \   [t_0^i\rd_1,t_0^j\rd_1]=0,\\
\label{ere6}&[\tilde\rd_{i,m},\tilde\rd_{j,n}]=(in-mj)\tilde\rd_{i+j,m+n}+\mu(in-mj)^3\rk_{i+j,m+n},\\
\label{ere7}&[t_0^i\rd_1,\tilde\rd_{j,n}]=n \tilde\rd_{i+j,n}
+\mu n^3 i^2\rk_{i+j,n}
\end{align}
for $u,v\in \fg$, $i,j,m,n\in \Z$.

Introduce a vector space $\mathcal{KD}$ with a designated basis $\{ \rK_n,\, \rD_n\mid n\in \Z^{\times}\}$.
Then form a vector space
\begin{align}\label{defag}
\mathcal A_\fg=\widetilde\fg_1\oplus \mathcal{KD}
=\C[t_1^{\pm 1}]\otimes \fg\oplus \C \rk_1\oplus \C \rd_1\oplus \sum_{n\in \Z^{\times}}(\C \rK_n\oplus \C \rD_n).
\end{align}
Form generating functions in $\wh{\mathfrak{t}}(\fg,\mu)[[z,z^{-1}]]$:
\begin{align}
u[z]&=\sum_{n\in \Z}(t_0^nu) z^{-n},\quad
\rK_r[z]=\sum_{n\in \Z}\rk_{n,r}z^{-n},\quad  \rD_m[z]=\sum_{n\in \Z} \tilde\rd_{n,m} z^{-n},
\end{align}
where $u\in \widetilde\fg_1,\   r\in \Z^\times, \ m\in \Z$.
Note that the coefficients of $a[z]$ for $a\in \mathcal A_\fg$
linearly span $\wh{\mathfrak{t}}(\fg,\mu)$.
In addition, set
\begin{align}
\rK_0[z]=\sum_{n\in \Z}\rk_{n,0}z^{-n}+\rk_1(\log z).
\end{align}
(Recall that $\rk_{0,0}=0$.)
Note that
$$\rd_1[z]=\sum_{n\in \Z}(t_0^n\rd_1)z^{-n},\quad \rk_1[z]=\sum_{n\in \Z}(t_0^{n}\rk_1)z^{-n}.$$
Then
\begin{align}
z\frac{d}{dz}\rK_0[z]=\rk_1[z],\quad z\frac{d}{dz} \rd_1[z]=-\rD_0[z].
\end{align}

Now, we write the relations \eqref{ere1}-\eqref{ere7} in terms of the generating
functions.

\begin{prpt}\label{prop:hgcomm}
For $u,v\in \fg$, $m,n\in \Z$, we have
\begin{align*}
&\ \te{(1)}\quad [(t_1^m\ot u)[z],(t_1^n\ot v)[w]]\\
&\quad  \  \ =(t_1^{m+n}\ot [u,v])[w]\dwz+\<u,v\>m\(\wpw\rK_{m+n}[w]\)\dwz\\
&\quad\quad  \ +\<u,v\>(m+n)\rK_{m+n}[w]\wpw\dwz+\delta_{m+n,0}\<u,v\>\wpw\dwz\rk_0,\\
&\ \te{(2)}\quad [\rK_m[z],(t_1^n\ot u)[w]]=0,\\
&\ \te{(3)}\quad[\rD_m[z],(t_1^n\ot u)[w]]\\
&\quad\quad = m\(\wpw (t_1^{m+n}\ot u)[w]\)\dwz+(m+n) (t_1^{m+n}\ot u)[w]\(\wpw\)\dwz,\\
&\ \te{(4)}\quad [\rd_1[z],(t_1^n\ot u)[w]]=n(t_1^n\ot u)[w]\dwz,\\
&\ \te{(5)}\quad [\rK_m[z],\rK_n[w]]=0,\\
&\ \te{(6)}\quad [\rD_m[z],\rK_n[w]]
=m\(\wpw \rK_{m+n}[w]\)\dwz\\
&\quad \hspace{3cm} +(m+n) \rK_{m+n}[w]\wpw\dwz+\delta_{m+n,0}\wpw\dwz\rk_0,\\
&\ \te{(8)}\quad [\rd_1[z],\rK_{n}[w]]=n\rK_{n}[w]\dwz+\delta_{n,0}\dwz \rk_0,\\
&\ \te{(9)}\quad [\rd_1[z],\rk_1[w]]=\wpw \dwz\rk_0,\quad [\rd_1[z],\rd_1[w]]=0,\\
&\ \te{(10)}\quad [\rD_{m}[z],\rD_{n}[w]]\\
&\quad \quad =m\(\wpw \rD_{m+n}[w]\)\dwz+(m+n) \rD_{m+n}[w]\wpw\dwz\\
&\quad\quad \  \  +\mu\sum_{r=0}^3{3\choose r}\(\(m\wpw\)^r \rK_{m+n}[w]\)\((m+n)\wpw\)^{3-r}\dwz,\\
&\ \te{(11)}\quad [\rd_1[z],\rD_{n}[w]]=n\rD_{n}[w]\dwz+\mu n^3\rK_{n}[w]\(\wpw\)^2\dwz.
\end{align*}
\end{prpt}

\begin{proof}  We here prove relation (10) while the others  easier can be proved similarly.
Using \eqref{ere6} and the identity $in-mj=i(m+n)+m(-i-j)$, we get
\begin{align*}
 &[\rD_m[z],\rD_n[w]]=\sum_{i,j\in \Z} [\tilde\rd_{i,m} z^{-i},\tilde\rd_{j,n} w^{-j}]\\
=\ &\sum_{i,j\in \Z}\( (m+n)\tilde{\rd}_{i+j,m+n}w^{-i-j}(iw^iz^{-i})
+m\tilde{\rd}_{i+j,m+n}(-i-j)w^{-i-j}(w^iz^{-i})\)\\
&+\mu \sum_{i,j\in \Z} \sum_{r=0}^3{3\choose r}m^r(-i-j)^r(m+n)^{3-r}i^{3-r} \rk_{i+j,m+n}w^{-i-j}(w^iz^{-i})\\
=\ &m\(\wpw \rD_{m+n}[w]\)\dwz+(m+n) \rD_{m+n}[w]\(\wpw\dwz\)\\
& \  +\mu\sum_{r=0}^3{3\choose r}\(\(m\wpw\)^r \rK_{m+n}[w]\)\((m+n)\wpw\)^{3-r}\dwz,
\end{align*}
proving (10).
\end{proof}

\section{Lie algebra $\wh{\mathfrak{t}}(\fg,\mu)^{o}$ and vertex algebras $V_{\wh{\mathfrak{t}}(\fg,\mu)^{o}}(\ell)$}
In this section, as the key ingredients we introduce a particular subalgebra
$\wh{\mathfrak{t}}(\fg,\mu)^{o}$ of the full toroidal Lie algebra $\mathcal{T}(\fg,\mu)$ and
associate vertex algebras $V_{\wh{\mathfrak{t}}(\fg,\mu)^{o}}(\ell)$ to this very Lie algebra
$\wh{\mathfrak{t}}(\fg,\mu)^{o}$.

\subsection{Lie algebras $\wh{\mathfrak{t}}(\fg,\mu)^{o}$ and $\wt{\mathfrak{t}}(\fg,\mu)^{o}$}
First, consider a variant of $\mathcal{D}_{\rm div}$.
Instead of using the free basis $\{ \rd_0,\rd_1\}$ of the $\mathcal{R}$-module $\mathcal{D}$,
we now use free basis $\{ t_0^{-1}\rd_0,\rd_1\}$.
 Set
\begin{align}
\mathcal{D}_{{\rm div}'}=\left\{f_0\frac{\partial}{\partial t_0}+f_1\rd_1\mid f_0,f_1\in \mathcal R,\
\frac{\partial}{\partial t_0}(f_0)+\rd_1(f_1)=0\right\}.
\end{align}
It is straightforward to show that $\mathcal{D}_{{\rm div}'}$ is also a Lie subalgebra of $\mathcal D$.
Note that
$$t_0^{-1}\rd_0,\quad t_0^{-1}\rd_1,\quad
\rd_1\in \mathcal{D}_{{\rm div}'},\quad  \rd_0\notin \mathcal{D}_{{\rm div}'}.$$

For $m,n\in \Z$, set
\begin{align}
\bar{\rd}_{n,m}=(n+1)t_0^nt_1^m \rd_1-m t_0^nt_1^m \rd_0,
\end{align}
which lies in $\mathcal{D}_{{\rm div}'}$. Notice that $\bar{\rd}_{-1,0}=0$.
It can be readily seen that the set
\begin{align}\label{basisB'}
\mathbb B':=\{t_0^{-1}\rd_0, t_0^{-1}\rd_1\}\cup\{  \bar{\rd}_{n,m}\mid (n,m)\in (\Z\times \Z)\backslash \{(-1,0)\}\}
\end{align}
is a $\C$-basis of  $\mathcal{D}_{{\rm div}'}$. By Lemma \ref{abdd}, the following relations hold in $\mathcal{T}(\fg,\mu)$:
\begin{align}\label{dbar-dbar}
&[\bar{\rd}_{m_0,m_1},\bar{\rd}_{n_0,n_1}]
= ((m_0+1)n_1-m_1(n_0+1))\bar{\rd}_{m_0+n_0,m_1+n_1}\\
&\quad +\mu (m_1n_0-n_1(m_0+1))(m_1(n_0+1)-m_0n_1)(m_0n_1-m_1n_0)\rk_{m_0+n_0,m_1+n_1}\nonumber\\
&\quad +\mu m_1^2\delta_{m_0+n_0,0}\delta_{m_1+n_1,0}(m_0\rk_0+m_1\rk_1)\nonumber
\end{align}
for $m_0,m_1,n_0,n_1\in \Z$.
In particular, the following relations in $\mathcal{D}_{{\rm div}'}$ (without central extension) hold:
\begin{align}
&[t_0^{-1}\rd_0,\bar{\rd}_{m_0,m_1}]=(m_0+1)\bar{\rd}_{m_0-1,m_1},\quad
[t_0^{-1}\rd_1,\bar{\rd}_{m_0,m_1}]=m_1\bar{\rd}_{m_0-1,m_1},\label{relations-divergence-va1}\\
&[\bar{\rd}_{m_0,m_1},\bar{\rd}_{n_0,n_1}]=((m_0+1)n_1-m_1(n_0+1))\bar{\rd}_{m_0+n_0,m_1+n_1}
\label{relations-divergence-va2}
\end{align}
for $m_0, m_1,n_0,n_1\in \Z$.

\begin{remt}
{\em Note that from  (\ref{relations-divergence-va1}) and (\ref{relations-divergence-va2}), we see that
$\mathcal{D}_{{\rm div}'}$ is indeed a subalgebra of $\mathcal{D}$. Furthermore, we have
\begin{align}
\mathcal{D}_{\rm div'}^{(1)}:=[\mathcal{D}_{\rm div'},\mathcal{D}_{\rm div'}]=\te{Span}\{ \bar{\rd}_{m,n}\mid m,n\in \Z\}
\end{align}
and $\mathcal{D}_{\rm div'}$ contains the following subalgebras:
\begin{align*}
\mathcal{D}_{\rm div'}^{(1)},\quad
\mathcal{D}_{\rm div'}^{(1)}\rtimes \C t_0^{-1}\rd_0, \quad \mathcal{D}_{\rm div'}^{(1)}\rtimes \C t_0^{-1}\rd_1, \quad
\mathcal{D}_{\rm div'}^{(1)}\rtimes (\C t_0^{-1}\rd_0+\C t_0^{-1}\rd_1)=\mathcal{D}_{\rm div'}.
\end{align*}}
\end{remt}

Now, using Lie subalgebras $\mathcal{D}_{{\rm div}'}$ and $\mathcal{D}_{\rm div'}^{(1)}\rtimes \C t_0^{-1}\rd_1$
of $\mathcal{D}$, we get two $\Z$-graded  subalgebras of the full toroidal Lie algebra $\mathcal{T}(\fg,\mu)$
(recall Remark \ref{full-Lie-subalgebras}).

\begin{dfnt}
{\em
Define
\begin{align*}
&\wt\ft(\fg,\mu)^{o}=\ft(\fg)+ \mathcal{D}_{{\rm div}'}=\mathcal{R}\ot\fg+\mathcal{K}
+ \C (t_0^{-1}\rd_0)+\C(t_0^{-1}\rd_1)+\sum_{(m,n)\in \Z\times \Z}\C\bar{\rd}_{m,n},\\
&\wh\ft(\fg,\mu)^{o}=\ft(\fg)+\mathcal{D}_{\rm div'}^{(1)}\rtimes \C t_0^{-1}\rd_1
=\mathcal{R}\ot \fg+\mathcal{K}+\C(t_0^{-1}\rd_1)+\sum_{(m,n)\in \Z\times \Z}\C\bar{\rd}_{m,n}.
\end{align*}}
\end{dfnt}

For $n,m\in \Z$, set
\begin{align}
\rd_{n,m}=\bar{\rd}_{n,m}+\mu m (n+1/2) t_0^nt_1^m\rk_0
=\bar{\rd}_{n,m}+\mu (n+1/2)m^2\rk_{n,m}\in \wh\ft(\fg,\mu)^{o}
\end{align}
(cf. \cite{B1}). In particular, for $n,m\in \Z$ we have
\begin{align}
\label{rdn0}\rd_{n,0}=\,&\bar{\rd}_{n,0}=(n+1)t_0^n\rd_1,\\
\rd_{0,m}=\,&\bar{\rd}_{0,m}+\frac{1}{2}\mu m t_1^m\rk_0
=t_1^m \rd_1-m t_1^m \rd_0+\frac{1}{2}\mu m t_1^m\rk_0.
\end{align}
Especially, we have
\begin{align}
\rd_{-1,0}=0,\quad \rd_{0,0}=\rd_1.
\end{align}

Form generating functions in $\wh\ft(\fg,\mu)^{o}[[z,z^{-1}]][\log z]$:
\begin{align}
&u(z)=\sum_{n\in \Z}(t_0^n u) z^{-n-1},\quad \rD_m(z)=\sum_{n\in \Z}\rd_{n,m}z^{-n-2},\\
&\quad \rK_m(z)=\rK_m[z]=\sum_{n\in \Z}\rk_{n,m}z^{-n}+\delta_{m,0}(\log z)\rk_1
\end{align}
for $u\in \widetilde\fg_1,\ m\in\Z$. (Recall that the square bracket notation such as $u[z]$ is used
for generating functions of $\wh\ft(\fg,\mu)$.)
It is clear that the coefficients of all generating functions linearly span $\wh\ft(\fg,\mu)^{o}$.
Noticing that
 $$\rk_1(z)=\sum_{n\in \Z}(t_0^n\rk_{1})z^{-n-1},\quad \rd_1(z)=\sum_{n\in \Z}(t_0^n\rd_{1})z^{-n-1}$$
(with $\rk_1,\rd_1\in \widetilde\fg_1$), we have
\begin{align}
\frac{d}{dz}\rK_0(z)=\rk_1(z),\quad \frac{d}{dz}\rd_1(z)=-\rD_0(z).
\end{align}

In addition to \eqref{ere1}, \eqref{ere3}, and \eqref{ere5},
the following relations hold in $\wh\ft(\fg,\mu)^{o}$:
\begin{align}
\label{bre2}&[\rd_{i,m},t_0^jt_1^n\ot u]=((i+1)n-mj)t_0^{i+j}t_1^{m+n}\ot u,\\
\label{bre3}&[\rd_{i,m},\rk_{j,n}]\\
& =((i+1)(m+n)-m(i+j))\rk_{i+j,m+n}
+\delta_{m+n,0}\delta_{i+j,0}((i+1)\rk_0+m\rk_1),\nonumber\\
\label{bre4}&[\rd_{i,m},t_0^j\rk_1]=m j(j-1)\rk_{i+j,m},\\
\label{bre5}&[\rd_{i,m},\rd_{j,n}]=((i+1)n-(j+1)m)\rd_{i+j,m+n}+2\mu m^3\delta_{i+j,0}\delta_{m+n,0}\rk_1\\
&\qquad\qquad\  +\mu\sum_{r=0}^3{3\choose r}n^{3-r}(i+1)^{(3-r)}(-m)^r(j+1)^{(r)}
\rk_{i+j,m+n},\notag\\
\label{bre6}&[t_0^i\rd_1,\rd_{j,n}]=n\rd_{i+j,n}+\mu n^3 i(i-1)\rk_{i+j,n}
\end{align}
for  $u\in \fg$, $i,j,m,n\in \Z$, where for $a\in \C,\ r\in \N$,
\begin{align}\label{eq:defmq}
a^{(r)}:=a(a-1)\cdots(a-r+1)=r! \binom{a}{r}.
\end{align}

All the relations except  \eqref{bre5}
follow easily from Lemmas \ref{relation-K},  \ref{abdk} and  \ref{abdd}.
In the following, we prove \eqref{bre5}. Using (\ref{dbar-dbar})  we have
\begin{align*}
&[\bar{\rd}_{i,m},\bar{\rd}_{j,n}]=[(i+1)t_0^it_1^m \rd_1-m t_0^it_1^m \rd_0, (j+1)t_0^jt_1^n \rd_1-n t_0^jt_1^n \rd_0]\\
=\ &((i+1)n-(j+1)m)\bar{\rd}_{i+j,m+n}\\
&+\mu((i+1)n-jm)(in-(j+1)m)\big((in-jm)\rk_{i+j,m+n}+\delta_{i+j,0}\delta_{m+n,0}
(i\rk_0+m\rk_1)\big)\\
=\ &((i+1)n-(j+1)m)\bar{\rd}_{i+j,m+n}\\
&+\mu ((i+1)n-jm)(in-(j+1)m)(in-jm)\rk_{i+j,m+n}+\mu m^2\delta_{i+j,0}\delta_{m+n,0}
(i\rk_0+m\rk_1)\\
=\ &((i+1)n-(j+1)m)\(\rd_{i+j,m+n}-\mu(m+n)^2(i+j+\frac 1 2)\rk_{i+j,m+n}\)\\
&+\mu((i+1)n-jm)(in-(j+1)m)(in-jm)\rk_{i+j,m+n}\\
&+\mu m^2\delta_{i+j,0}\delta_{m+n,0}(i\rk_0+m\rk_1)\\
=\ &((i+1)n-(j+1)m)\rd_{i+j,m+n} +\mu C^{i,j}_{m,n}\rk_{i+j,m+n}
+\mu m^2\delta_{i+j,0}\delta_{m+n,0}(i\rk_0+m\rk_1),
\end{align*}
where
\begin{align*}
C^{i,j}_{m,n}=& n^3(i+1)(i^2-i-j-\frac 1 2)
-n^2m\big(3(i+1)i(j+1)-(j+\frac 1 2)(j-1)\big)\\
& +m^2n\big(3(i+1)j(j+1)-(i+\frac 1 2)(i-1)\big)-m^3(j+1)(j^2-j-i-\frac 1 2).
\end{align*}
On the other hand, using \eqref{bre3} we get
\begin{align*}
&[\bar{\rd}_{i,m},\mu n^2(j+\frac 1 2)\rk_{j,n}]-[\bar{\rd}_{j,n}, \mu m^2 (i+\frac 1 2) \rk_{i,m}]\\
=&\mu\(((i+1)n-(j-1)m)n^2(j+\frac 1 2)-((j+1)m-(i-1)n)m^2(i+\frac 1 2)\)\rk_{i+j,m+n}\\
&+\mu \delta_{i+j,0}\delta_{m+n,0}\left( n^2(j+\frac 1 2)((i+1)\rk_0+m\rk_1)-m^2(i+\frac 1 2)((j+1)\rk_0+n\rk_1)\right)\\
=&\mu\(((i+1)n-(j-1)m)n^2(j+\frac 1 2)-((j+1)m-(i-1)n)m^2(i+\frac 1 2)\)\rk_{i+j,m+n}\\
&+\mu m^2 \delta_{i+j,0}\delta_{m+n,0}( -i\rk_0+m\rk_1).
\end{align*}
Combining the two relations above we obtain  \eqref{bre5}.


Rewriting these relations in terms of generating functions, we have:

\begin{prpt}\label{prop:hbcomm}
For $u,v\in \fg$, $m,n\in \Z$, we have
\begin{align*}
&\ \te{(1)}\quad [(t_1^m\ot u)(z),(t_1^n\ot v)(w)]\\
&\quad\quad =(t_1^{m+n}\ot [u,v])(w)z^{-1}\dwz+\<u,v\>m\(\pw\rK_{m+n}(w)\)z^{-1}\dwz\\
&\quad \quad \ \  +(m+n)\<u,v\>\rK_{m+n}(w)\pw z^{-1}\dwz
 +\delta_{m+n,0}\<u,v\> \pw z^{-1}\dwz \rk_0,\\
&\ \te{(2)}\quad  [\rK_m(z),(t_1^n\ot u)(w)]=0,\\
&\ \te{(3)}\quad[\rD_m(z),(t_1^n\ot u)(w)] =m\(\pw (t_1^{m+n}\ot u)(w)\)z^{-1}\dwz\\
&\quad\hspace{4.5cm} +(m+n) (t_1^{m+n}\ot u)(w)\pw z^{-1}\dwz,\\
&\ \te{(4)}\quad [\rd_1(z),(t_1^n\ot u)(w)]=n (t_1^n\ot u)(w) z^{-1}\dwz,\\
&\ \te{(5)}\quad [\rK_m(z),\rK_n(w)]=0,\\
&\ \te{(6)}\quad [\rD_m(z),\rK_n(w)] =m\(\pw \rK_{m+n}(w)\)z^{-1}\dwz\\
&\quad \hspace{3cm}\ +(m+n) \rK_{m+n}(w)\pw z^{-1}\dwz+\delta_{m+n,0}\pw z^{-1}\dwz \rk_0,\\
&\ \te{(8)}\quad [\rd_1(z),\rK_n(w)]=n\rK_n(w)z^{-1}\dwz+\delta_{n,0} z^{-1}\dwz \rk_0,\\
&\ \te{(9)}\quad [\rd_1(z),\rk_1(w)]=\pw z^{-1}\dwz \rk_0,\quad [\rd_1(z),\rd_1(w)]=0,\\
&\ \te{(10)}\quad [\rD_m(z),\rD_n(w)]\\
&\quad \quad\  \  =m\(\pw \rD_{m+n}(w)\)z^{-1}\dwz+(m+n) \rD_{m+n}(w)\pw z^{-1}\dwz\\
&\quad\quad \   \  +\mu\sum_{i=0}^3{3\choose i} \(\(m\pw\)^i \rK_{m+n}(w)\)\((m+n)\pw\)^{3-i} z^{-1}\dwz,\\
&\ \te{(11)}\quad [\rd_1(z),\rD_n(w)]=n\rD_n(w)z^{-1}\dwz+\mu n^3\rK_n(w)\(\pw\)^2 z^{-1}\dwz.
\end{align*}
\end{prpt}

\begin{proof} Relations (1)-(9) and (11) can be proved straightforwardly by using
the simple identity $an-bm=a(m+n)+m(a+b)$ for $a,b\in \Z$ and the fact
\begin{align*}
&\frac{d}{dw} \rK_0(w)=\sum_{l\in \Z}-l\rk_{l,0}w^{-l-1}+\rk_1w^{-1},\\
& \left(\frac{d}{dw}\right)^3\rK_0(w)=\sum_{l\in \Z}(-l)^{(3)}\rk_{l,0}w^{-l-3}+2\rk_1w^{-3}.
\end{align*}
As for relation (10), using the technical result in Lemma \ref{append-lemma1} we have
\begin{align*}
&\sum_{r=0}^3{3\choose r}n^{3-r}(i+1)^{(3-r)}(-m)^r(j+1)^{(r)}\\
=\ & \sum_{s=0}^3{3\choose s}(m+n)^{3-s}(i+1)^{(3-s)}m^s(-i-j)^{(s)}.\nonumber
\end{align*}
Then using this relation and \eqref{bre5}, we obtain
\begin{align*}
&[\rD_m(z),\rD_n(w)]\ \left(=\sum_{i,j\in \Z}[\rd_{i,m},\rd_{j,n}]z^{-i-2}w^{-j-2}\right)\\
=\ &\sum_{i,j\in \Z}\big((i+1)n-(j+1)m)\rd_{i+j,m+n}z^{-i-2}w^{-j-2}\\
&+\mu\sum_{i,j\in \Z}\sum_{r=0}^3{3\choose r}n^{3-r}(i+1)^{(3-r)}(-m)^r(j+1)^{(r)}
\rk_{i+j,m+n}z^{-i-2}w^{-j-2}\notag\\
&+2\mu m^3\delta_{m+n,0}\sum_{i,j\in \Z}\delta_{i+j,0}\rk_1z^{-i-2}w^{-j-2}\\
=\ &m\sum_{i,j\in \Z}\big((-i-j-2)\rd_{i+j,m+n}w^{-i-j-3}\big)z^{-i-2}w^{i+1}\\
&+(m+n)\sum_{i,j\in \Z}\rd_{i+j,m+n}w^{-i-j-2}\big((i+1)z^{-i-2}w^i\big)\\
&+\mu\sum_{i,j\in \Z}\sum_{s=0}^3{3\choose s}(m+n)^{3-s}(i+1)^{(3-s)}m^s(-i-j)^{(s)}
(\rk_{i+j,m+n}w^{-i-j-s})(z^{-i-2}w^{i-2+s})\notag\\
&+2\mu m^3\delta_{m+n,0}\sum_{i\in \Z}(\rk_1 w^{-3})(z^{-i-2}w^{i+1})\\
=\ &m\(\pw \rD_{m+n}(w)\)z^{-1}\dwz+(m+n) \rD_{m+n}(w)\(\pw z^{-1}\dwz\)\\
&+\mu\sum_{r=0}^3{3\choose r} \(\(m\pw\)^{r} \rK_{m+n}(w)\)\(\((m+n)\pw\)^{3-r} z^{-1}\dwz\),
\end{align*}
proving relation (10).
\end{proof}

The following is the technical result we used in the proof above:

\begin{lemt}\label{append-lemma1}
Let $p\in \N,\ a,b,\  \alpha,\beta\in \C$.  Then
\begin{align}\label{XYZ}
&\sum_{r=0}^p{p\choose r}\alpha^{p-r}a^{(p-r)}(-\beta)^rb^{(r)}\\
=\ &\sum_{t=0}^p{p\choose t}(\alpha+\beta)^{p-t}a^{(p-t)}\beta^{t}(-a-b-1+p)^{(t)}.\nonumber
\end{align}
\end{lemt}

\begin{proof} Note that we have the following version of the Newton identity
\begin{align}\label{newton}
(a+b)^{(q)}=\sum_{i=0}^{q}{q\choose i} a^{(i)}b^{(q-i)}.
\end{align}
Using this we get
\begin{align*}
&\sum_{r=0}^p{p\choose r}\alpha^{p-r}a^{(p-r)}(-\beta)^rb^{(r)}\\
=\ &\sum_{r=0}^p\sum_{s=0}^{p-r}{p-r\choose s}{p\choose r}(\alpha+\beta)^{p-r-s}a^{(p-r)}(-\beta)^{r+s}b^{(r)}\nonumber\\
=\ &\sum_{r=0}^p\sum_{s=0}^{p-r}{p\choose r+s}{r+s\choose r}(\alpha+\beta)^{p-r-s}a^{(p-r)}(-\beta)^{r+s}b^{(r)}\nonumber\\
=\ &\sum_{t=0}^p\sum_{r=0}^{t}{p\choose t}{t\choose r}(\alpha+\beta)^{p-t}a^{(p-r)}(-\beta)^{t}b^{(r)}\nonumber\\
=\ &\sum_{t=0}^p\sum_{r=0}^{t}{p\choose t}{t\choose r}(\alpha+\beta)^{p-t}a^{(p-t)}(a-p+t)^{(t-r)}(-\beta)^{t}b^{(r)}\nonumber\\
=\ &\sum_{t=0}^p{p\choose t}(\alpha+\beta)^{p-t}a^{(p-t)}(-\beta)^{t}(a+b-p+t)^{(t)}\nonumber\\
=\ &\sum_{t=0}^p{p\choose t}(\alpha+\beta)^{p-t}a^{(p-t)}\beta^{t}(-a-b-1+p)^{(t)},\nonumber
\end{align*}
noticing that $a^{(p-r)}=a^{(p-t)}(a-p+t)^{(t-r)}$ with $0\le r\le t\le p$.
\end{proof}

\subsection{Vertex Lie algebras and vertex algebras}
We here recall the notion of vertex Lie algebra from \cite{DLM2}
(see also \cite{B1, Kac, P}) and the vertex algebras associated to a vertex Lie algebra.

Let $\partial$ be an indeterminate. For a vector space $U$, denote by $\C[\partial]U$
the free $\C[\partial]$-module over $U$, i.e., $\C[\partial]U=\C[\partial]\otimes U$.
A {\em vertex Lie algebra} is a Lie algebra ${\mathcal{L}}$ equipped with vector spaces $\mathcal A$ and $\mathcal C$,
and a linear bijection
\begin{eqnarray}
\rho:\  {\mathcal{A}}\otimes \C[t,t^{-1}]\oplus {\mathcal{C}}&\rightarrow& {\mathcal{L}}\nonumber\\
a\otimes t^n+c&\mapsto& a(n)+c\quad(\text{for }a\in {\mathcal{A}},\ n\in \Z,\ c\in {\mathcal{C}}),
\end{eqnarray}
satisfying the condition that  ${\mathcal{C}}$ is central in ${\mathcal{L}}$ and
 for $a,b\in {\mathcal{A}}$, there exist finitely many elements $f_i(a,b)\in \C[\partial]{\mathcal{A}}+{\mathcal{C}}$
 for $i=0,1,\dots, r$ such that
\begin{eqnarray}\label{az-bw}
[a(z),b(w)]=\sum_{i=0}^{r} f_i(a,b)(w)\frac{1}{i!}\left(\frac{\partial}{\partial w}\right)^iz^{-1}\delta\left(\frac{w}{z}\right),
\end{eqnarray}
where $c(z)=c\in {\mathcal{L}}$ for $c\in {\mathcal{C}}$, and for $a\in {\mathcal{A}},\ i\in \N$,
\begin{align}
(\partial^ia)(z)=\left(\frac{d}{dz}\right)^i\sum_{n\in \Z}a(n)z^{-n-1}\in {\mathcal{L}}[[z,z^{-1}]].
\end{align}
Furthermore,  $\CL$ is called a {\em $\Z$-graded vertex Lie algebra} if
 $\CL=\oplus_{n\in \Z}\CL_{(n)}$ is a $\Z$-graded Lie algebra with
 $\mathcal A=\oplus_{n\in \Z}\CA_{(n)}$ a $\Z$-graded vector space such that
\begin{align*}
\deg \mathcal C=0\quad\te{and}\quad \deg a(n)=m-n-1
\quad \text{for  }a\in \mathcal A_{(m)},\  m,n\in \Z.
\end{align*}

Let $(\mathcal L, \mathcal A, \mathcal C, \rho)$ be a vertex Lie algebra.  Set
\begin{align*}
\mathcal L_-=\te{Span}\{a(n)\mid a\in \mathcal A,\  n<0\},\quad
\mathcal L_+=\te{Span}\{a(n)\mid a\in \mathcal A,\  n\ge 0\},
\end{align*}
which are subalgebras of $\CL$. Then we have a triangular decomposition of  $\mathcal L$:
\begin{align}\label{polardec}
\mathcal L=\mathcal L_+\oplus \mathcal C\oplus \mathcal L_-.
\end{align}
 Let $\gamma\in {\mathcal C}^*$.
 View  $\C$ as an $(\mathcal L_+ +\mathcal C)$-module with $\mathcal L_+$
acting trivially and with $c$ acting as scalar $\gamma(c)$ for $c\in \mathcal C$.
Then form an induced $\mathcal L$-module
\begin{align}
V_{\mathcal L}(\gamma):=\U(\mathcal L)\otimes_{\U(\mathcal L_+
+ \mathcal C)}\C.
\end{align}
Set $\mathbf{1}=1\ot 1\in V_{\mathcal L}(\gamma)$.
View $\mathcal A$ as a subspace of $V_{\mathcal L}(\gamma)$ through the linear map $a\mapsto a(-1)\mathbf{1}$.

The Lie algebra $\mathcal L$ admits a derivation $\bm{d}$ determined by
\begin{align}\label{defD}
\bm{d}(c)=0,\quad \bm{d}(a(n))=-n a(n-1)\quad  \te{for}\ c\in \mathcal C,\ a\in \mathcal A,\  n\in \Z.
\end{align}
View $\bm{d}$ as a derivation of $\U(\mathcal L)$.
As $\bm{d}$ preserves the subalgebra $\mathcal L_+ +{\mathcal{C}}$,
$\bm{d}$ acts on $V_{\mathcal L}(\gamma)$ with $\bm{d}\mathbf{1}=0$ and
$V_{\mathcal L}(\gamma)$ becomes  an $\mathcal L\rtimes \C\bm{d}$-module.

A {\em $\Z$-graded vertex algebra} is a vertex algebra $V$
with a  $\Z$-grading $V=\oplus_{n\in \Z}V_{(n)}$  such that ${\bf 1}\in V_{(0)}$ and
\begin{eqnarray}
u_{m}V_{(n)}\subset V_{(n+k-m-1)}\ \ \mbox{ for }u\in V_{(k)},\ m,n, k\in \Z.
\end{eqnarray}

The following result can be found in \cite{DLM2}:

 \begin{prpt}\label{prop:vla}
 Let $(\mathcal L, \mathcal A, \mathcal C, \rho)$
  be a vertex Lie algebra and let $\gamma$ be a linear functional on $\mathcal C$.
Then there exists a vertex algebra structure on $V_{\mathcal L}(\gamma)$, which is uniquely determined by
the condition that $\mathbf{1}$ is the vacuum vector and
\begin{align}\label{vlgaction}
Y(a,z)=a(z)\quad\te{for}\ a\in \mathcal A.\end{align}
Furthermore, if $\mathcal L$ is $\Z$-graded, then
$V_{\mathcal L}(\gamma)$ is a $\Z$-graded  vertex algebra with
\begin{align}
\deg\(a_1(n_1)\cdots a_s(n_s)\mathbf{1}\)=\deg a_1(n_1)+\cdots +\deg a_s(n_s)
\end{align}
for $a_i\in \mathcal A,\ n_i\in \Z$.
\end{prpt}

Let $\pi:
\  \C[\partial]{\mathcal{A}}\oplus {\mathcal{C}}\rightarrow V_{\mathcal L}(\gamma)$
be the linear map defined by
\begin{align}
\pi(c)=\gamma(c){\bf 1},\quad \pi (\partial^ia)=L(-1)^ia\quad \text{for }c\in {\mathcal{C}},\  i\in \N,\ a\in {\mathcal{A}},
\end{align}
where $L(-1)$ denotes the canonical derivation of the vertex algebra $V_{\mathcal L}(\gamma)$,
defined by $L(-1)v=v_{-2}\mathbf{1}$ for $v\in V_{\mathcal L}(\gamma)$. Then
\begin{align}\label{Y-extension-A}
Y(\pi (X),z)=X(z)\quad \text{ for }X\in \C[\partial]{\mathcal{A}}\oplus {\mathcal{C}}.
\end{align}

Note that $\bm{d}$ on $V_{\mathcal L}(\gamma)$ coincides with $L(-1)$.
In particular, we have
\begin{align}\label{Dva}
[\bm{d}, Y(v,z)]=Y(\bm{d}(v),z)=\frac{d}{dz}Y(v,z)\quad\te{for}\ v\in V_{\mathcal L}(\gamma).
\end{align}
From  \cite{LL}, we immediately have:

\begin{lemt}\label{idealofvla}
An ideal of the vertex algebra $V_{\mathcal L}(\gamma)$ is the same as
an $\mathcal L\rtimes \C\bm{d}$-submodule of $V_{\mathcal L}(\gamma)$.
 \end{lemt}

For a vector space $W$, set
\[\CE(W)=\mathrm{Hom}(W,W((z))).\]
 An $\CL$-module $W$ is said to be {\em restricted} if for any $a\in \CA$, $a(z)\in \CE(W)$,
and {\em of level $\gamma\in \mathcal C^*$} if $c\cdot w=\gamma(c)w$ for $c\in \mathcal C,\  w\in W$.
 Then we have (see \cite{DLM2}):

\begin{lemt}\label{lem:modofvla}
Let $W$ be a vector space and let $\gamma\in \mathcal C^*$.
Then a restricted $\CL$-module structure of level $\gamma$
on $W$ amounts to  a $V_{\mathcal L}(\gamma)$-module structure $Y_W(\cdot,z)$ on $W$ such that
$Y_W(a,z)=a(z)$ for $a\in {\mathcal{A}}$.
\end{lemt}

\subsection{Vertex algebras $V_{\widehat{\mathfrak{t}}(\fg,\mu)^o}(\ell)$ }
We here show that Lie algebra $\widehat{\mathfrak{t}}(\fg,\mu)^o$ is a vertex Lie algebra and
associate vertex algebras $V_{\widehat{\mathfrak{t}}(\fg,\mu)^o}(\ell)$ to $\widehat{\mathfrak{t}}(\fg,\mu)^o$.

Recall $\wt{\fg}_{1}=\C[t_1,t_1^{-1}]\ot \fg+\C\rk_1+\C\rd_1\subset \widehat{\mathfrak{t}}(\fg,\mu)^o$.
Set
\begin{align}
 \widehat{\mathfrak{t}}(\fg,\mu)^{o}_{+}&=\te{Span}\{t_0^n u, \rk_{n+1,m},
\rd_{n-1,m}\mid u\in \widetilde\fg_1,\  n\in \N,\ m\in \Z\},\\
 \widehat{\mathfrak{t}}(\fg,\mu)^{o}_{-}&=\te{Span}\{t_0^{-n-1} u, \rk_{-n,m},
\rd_{-n-2,m}\mid u\in \widetilde\fg_1,\ n\in \N,\ m\in \Z\}.
\end{align}
It is clear that $ \widehat{\mathfrak{t}}(\fg,\mu)^{o}_{\pm}$ are subalgebras of $ \widehat{\mathfrak{t}}(\fg,\mu)^o$
and we have the following triangular decomposition
\begin{align}
 \widehat{\mathfrak{t}}(\fg,\mu)^{o}
 = \widehat{\mathfrak{t}}(\fg,\mu)^{o}_{+}\oplus \C\rk_0\oplus  \widehat{\mathfrak{t}}(\fg,\mu)^{o}_{-}.
\end{align}

Let $\ell$ be a complex number. View $\C$ as a  $(\wh{\mathfrak{t}}(\fg,\mu)^{o}_{+}+\C\rk_0)$-module
with $\wh{\mathfrak{t}}(\fg,\mu)^{o}_{+}$ acting trivially and with $\rk_0$ acting as scalar $\ell$.
Form an induced $ \widehat{\mathfrak{t}}(\fg,\mu)^{o}$-module
\begin{align}\label{universal-va}
V_{\widehat{\mathfrak{t}}(\fg,\mu)^o}(\ell)
=\U( \widehat{\mathfrak{t}}(\fg,\mu)^{o})\otimes_{\U( \widehat{\mathfrak{t}}(\fg,\mu)^{o}_{+}+ \C\rk_0)}\C.
\end{align}
Set $\mathbf{1}=1\ot 1\in V_{\widehat{\mathfrak{t}}(\fg,\mu)^o}(\ell)$.
Notice that $\widehat{\mathfrak{t}}(\fg,\mu)^{o}_{+} + \C\rk_0$ is a graded subalgebra.
Define $\deg \C=0$, to make $\C$ a $\Z$-graded $(\widehat{\mathfrak{t}}(\fg,\mu)^{o}_{+}+ \C\rk_0)$-module.
Then $V_{\widehat{\mathfrak{t}}(\fg,\mu)^o}(\ell)$ is a $\Z$-graded $\widehat{\mathfrak{t}}(\fg,\mu)^{o}$-module
\begin{align}
V_{\widehat{\mathfrak{t}}(\fg,\mu)^o}(\ell)=\bigoplus_{m\in \Z}V_{\widehat{\mathfrak{t}}(\fg,\mu)^o}(\ell)_{(m)}
\end{align}
with $V_{\widehat{\mathfrak{t}}(\fg,\mu)^o}(\ell)_{(m)}=0$ for $m<0$.

Recall
$$\mathcal A_\fg
=\C[t_1,t_1^{-1}]\otimes \fg\oplus \C \rk_1\oplus \C \rd_1\oplus \sum_{n\in \Z^{\times}}(\C \rK_n\oplus \C \rD_n).$$
Identify $\mathcal A_\fg$ as a subspace of $V_{\widehat{\mathfrak{t}}(\fg,\mu)^o}(\ell)$
through the linear map
\begin{align}\label{embaginv}
u\mapsto (t_0^{-1}u)\mathbf{1},\ \rK_n\mapsto \rk_{0,n}\mathbf{1},\ \rD_n\mapsto \rd_{-2,n}\mathbf{1}\quad\te{for}
\ u\in \widetilde\fg_1,\, n\in \Z^\times.
\end{align}
Define a linear map
$$\rho:\   \(\C[t,t^{-1}]\ot \mathcal A_\fg\)\oplus \C\rk_0\rightarrow \widehat{\mathfrak{t}}(\fg,\mu)^o$$
 by
 \begin{align*}
 \rho(t^m\ot u)= t_0^mu,\quad \rho(t^m\ot \rK_n)=\rk_{m+1,n},\quad \rho(t^m\ot \rD_n)=
 \rd_{m-1,n},\quad \rho(\rk_0)= \rk_0
 \end{align*}
 for $u\in \widetilde\fg_1$, $m\in \Z,\  n\in \Z^\times$.

With this setting we have:

\begin{prpt} \label{prop:vas}
The quadruple $(\wh{\ft}(\fg,\mu)^o,\mathcal A_\fg,\C\rk_0,\rho)$ carries the structure of a vertex Lie algebra.
For any $\ell \in \C$, there is a vertex algebra structure on $V_{\widehat{\mathfrak{t}}(\fg,\mu)^o}(\ell)$,
which is uniquely determined by the condition that
$\mathbf{1}$ is the vacuum vector and
\begin{align*}
Y(a,z)=a(z)\quad\te{for}\ a\in \mathcal A_\fg.
\end{align*}
Furthermore, $V_{\widehat{\mathfrak{t}}(\fg,\mu)^o}(\ell)$ is a $\Z$-graded vertex algebra.
\end{prpt}

\begin{proof} The first assertion follows immediately from Proposition \ref{prop:hbcomm}.
Define
\begin{align*}
\deg(t_1^m\ot u)=1,\quad \deg(\rk_1)=\deg(\rd_1)=1,\quad
 \deg(\rK_n)=0,\quad \deg(\rD_n)=2
\end{align*}
for $u\in \fg,\  m\in \Z,\  n\in \Z^{\times}$, to make $\mathcal A_\fg$  a $\Z$-graded  vector space. Then
$\wh{\mathfrak{t}}(\fg,\mu)^{o}$ becomes a $\Z$-graded vertex Lie algebra.
Let $\gamma_\ell$ be the linear functional on $\C\rk_0$ defined by $\gamma_\ell(\rk_0)=\ell$.
We see that $V_{\widehat{\ft}(\fg,\mu)^{o}}(\ell)$ coincides with
the $\wh{\mathfrak{t}}(\fg,\mu)^{o}$-module $V_{\widehat{\ft}(\fg,\mu)^{o}}(\gamma_\ell)$
defined in Section 3.2.
Now, it follows from Proposition \ref{prop:vla} that $V_{\widehat{\ft}(\fg,\mu)^{o}}(\ell)$ is a $\Z$-graded vertex algebra.
\end{proof}

\begin{remt}\label{idealofv}
{\em Recall that $\wt{\mathfrak{t}}(\fg,\mu)^o=\wh{\mathfrak{t}}(\fg,\mu)^{o}\rtimes \C (t_0^{-1}\rd_0)$.
It can be readily seen that vertex algebra $V_{\wh{\mathfrak{t}}(\fg,\mu)^{o}}(\ell)$ is a
$\wt{\mathfrak{t}}(\fg,\mu)^o$-module with $-t_0^{-1}\rd_0$ acting as the canonical derivation $L(-1)$.
Then an ideal of vertex algebra $V_{\wh{\mathfrak{t}}(\fg,\mu)^{o}}(\ell)$
is the same as a $\wt{\mathfrak{t}}(\fg,\mu)^o$-submodule (see Lemma \ref{idealofvla}). }
\end{remt}


\subsection{Vertex algebras $V^{\rm int}_{\widehat{\mathfrak{t}}(\fg,\mu)^o}(\ell)$}

Now, we assume that $\fg$ is a finite-dimensional simple Lie algebra with a Cartan subalgebra $\fh$.
Denote by $\Delta$  the root system of $\fg$ with respect to $\fh$.
Fix a simple root base $\Pi$ of $\Delta$ and denote by $\theta$ the highest long root.
Let $\<\cdot,\cdot \>$ be the Killing form which is normalized so that the squared length of the long roots equals $2$.
Note that for $\al\in \Delta$, we have $\<\al,\al\>=2$, $1$, or $\frac{2}{3}$  (for $G_2$). Set
\begin{align}
\epsilon_\al=\frac{2}{\<\al,\al\>}\in \{1,2,3\}.
\end{align}

In the following, we shall define a vertex algebra
$V^{\rm int}_{\widehat{\mathfrak{t}}(\fg,\mu)^o}(\ell)$  for every nonnegative integer $\ell$ and associate
$V^{\rm int}_{\widehat{\mathfrak{t}}(\fg,\mu)^o}(\ell)$-modules to restricted and integrable
$\wh{\mathfrak{t}}(\fg,\mu)^o$-modules of level $\ell$.

For each $\al\in \Delta$, choose a Lie algebra embedding
$\rho_\al: \mathfrak{sl}_2\rightarrow \fg$ such that $\rho_\al(e)\in \fg_\al,\  \rho_\al(f)\in \fg_{-\al}$, and
set
\begin{align}
h_\al=\rho_\al(h)\in \fh\subset \fg,
\end{align}
where  $\{e,h,f\}$ is the standard basis of $\mathfrak{sl}_2$ with
\[[e,f]=h,\quad [h,e]=2e,\quad [h,f]=-2f.\]
Note that $\rho_\al(h)$ is independent of the choice of $\rho_\al$ and
\begin{align}
\<\rho_\al(e),\rho_\al(f)\>=\frac{1}{2}\<h_\al,h_\al\>=\frac{2}{\<\al,\al\>}=\epsilon_\al
\end{align}
(also independent of the choice of $\rho_\al$).

Let $m\in \Z,\ \al\in \Delta$.  Notice that from (\ref{fre1}) we have
$$[t_0^pt_1^m\rho_\al(e),t_0^qt_1^{-m}\rho_\al(f)]
=t_0^{p+q}h_\al+\<\rho_\al(e),\rho_\al(f)\>(mt_0^{p+q}\rk_1+p\delta_{p+q,0}\rk_0)$$
for $p,q\in \Z$.
Set
\begin{align}\label{sl2al}
\wh{\mathfrak{sl}}_2(m,\al)=\C[t_0^{\pm 1}]t_1^{m}\fg_\al+\C[t_0^{\pm 1}]t_1^{-m}\fg_{-\al}
+\C[t_0^{\pm 1}](h_\al+m\epsilon_\al\rk_1)+\C\rk_0,
\end{align}
which is a subalgebra of $\ft(\fg)$ and hence a subalgebra of both $\wh{\ft}(\fg,\mu)$ and $\wh{\ft}(\fg,\mu)^{o}$.
It is straightforward to see that the linear map
\begin{align}
\rho_{m,\al}: \  \wh{\mathfrak{sl}}_2\rightarrow \wh{\mathfrak{sl}}_2(m,\al),
\end{align}
defined by   $\rho_{m,\al}({\bf c})=\epsilon_\al\rk_0$ and
\begin{align*}
\rho_{m,\al}(t^ne)=t_0^nt_1^m\rho_\al(e),\   \   \rho_{m,\al}(t^nf)=t_0^nt_1^{-m}\rho_\al(f),\   \
\rho_{m,\al}(t^nh)=t_0^n(h_\al+m\epsilon_\al\rk_1)
\end{align*}
for $n\in \Z$, is a Lie algebra isomorphism.

\begin{dfnt}
{\em A $\wh{\ft}(\fg,\mu)^o$-module $W$ is said to be {\em restricted}  if
$a(z)\in \CE(W)$ for every $a\in \mathcal A_\fg$,   is said to be of {\em level} $\ell\in \C$ if $\rk_0$ acts as scalar $\ell$,
and  is said to be {\em integrable} if for any $m_0,m_1\in \Z,\  \al\in \Delta$,
$t_0^{m_0}t_1^{m_1}\fg_\al$ acts locally nilpotently on $W$.}
\end{dfnt}

Next, we present a characterization of integrable and restricted
$\wh{\ft}(\fg,\mu)^o$-modules as an analogue of a result for affine Kac-Moody algebras.

\begin{lemt}\label{lem:charinto}
Let $W$ be a restricted $\wh{\ft}(\fg,\mu)^o$-module  of level $\ell\in \C$. Then $W$ is integrable if and only if
$\ell$ is a nonnegative integer and
\begin{align*}
a(z)^{\epsilon_\al\ell+1}=0\quad \te{on}\ W\  \te{for any } a\in t_1^m\fg_{\al}\  \te{with } m\in \Z, \ \al\in \Delta.
\end{align*}
\end{lemt}

\begin{proof}  Note that for any $m\in \Z,\ \al\in \Delta$, via the homomorphism $\rho_{m,\al}$,
$W$ becomes a restricted module of level $\epsilon_\al\ell$ for the affine Lie algebra $\wh{\mathfrak{sl}}_2$.
We see that $\wh{\ft}(\fg,\mu)^o$-module $W$ is integrable if and only if  for any $m\in \Z,\ \al\in \Delta$,
$W$ is an integrable $\wh{\mathfrak{sl}}_2$-module via $\rho_{m,\al}$.
Recall from \cite{DLM1} that a restricted $\wh{\mathfrak{sl}}_2$-module $M$ of level $\ell'\in \C$ is integrable
if and only if $\ell'$ is a nonnegative integer and
$$e(z)^{\ell'+1}=0=f(z)^{\ell'+1}=0\  \  \te{on}\  M.$$
Then it follows immediately.
\end{proof}

\begin{dfnt}
{\em Let $\ell$ be a nonnegative integer. Denote by $J(\ell)$
the $\widehat{\ft}(\fg,\mu)^{o}$-submodule  of $V_{\widehat{\ft}(\fg,\mu)^{o}}(\ell)$, generated by the vectors
\begin{align}\label{defjbl}
(t_0^{-1}a)^{\epsilon_\al\ell+1} \mathbf{1}\quad \mbox{ for }  a\in t_1^m\fg_{\al} \mbox{ with }m\in \Z,\ \al\in \Delta.
\end{align}}
\end{dfnt}

\begin{lemt}\label{idealJ}
Let $\ell$ be a nonnegative integer. Then the $\widehat{\ft}(\fg,\mu)^{o}$-submodule $J(\ell)$
is a graded  ideal of the vertex algebra $V_{\wh{\ft}(\fg,\mu)^{o}}(\ell)$.
\end{lemt}

\begin{proof} As the generators of $J(\ell)$ are homogeneous,
$J(\ell)$ is a graded submodule of $V_{\wh{\ft}(\fg,\mu)^{o}}(\ell)$.
To prove that $J(\ell)$ is a graded ideal of $V_{\wh{\ft}(\fg,\mu)^{o}}(\ell)$,
by Remark \ref{idealofv} it suffices to show that $J(\ell)$ is $t_0^{-1}\rd_0$-invariant.
Let $a\in t_1^r\fg_{\al}$ with $\al\in \Delta,\ r\in \Z$.
Note that the quotient $\wh{\ft}(\fg,\mu)^{o}$-module $V_{\wh{\ft}(\fg,\mu)^{o}}(\ell)/J(\ell)$ is integrable.
Then it follows from Lemma \ref{lem:charinto} that
\begin{align}\label{desJ}
a(z)^{\epsilon_\al\ell+1}\bm{1}\in J(\ell)[[z,z^{-1}]].
\end{align}
Note that $(t_0^{-1}\rd_0){\bf 1}=0$, $[t_0^{m}a,t_0^{n}a]=0$ for $m,n\in \Z$, and $(t_0^{p}a){\bf 1}=0$ for $p\ge 0$.
Then using  \eqref{desJ} we get
\[(t_0^{-1}\rd_0) (t_0^{-1} a)^{\epsilon_\al\ell+1}\mathbf{1}=-(\epsilon_\al\ell+1)(t_0^{-2}a)
(t_0^{-1}a)^{\epsilon_\al\ell}\mathbf{1}
=-\te{Res}_z z^{-2}a(z)^{\epsilon_\al\ell+1}\mathbf{1}\in J(\ell).
\]
It follows that $(t_0^{-1}\rd_0)J(\ell)\subset J(\ell)$.
Therefore, $J(\ell)$ is a graded ideal of $V_{\wh{\ft}(\fg,\mu)^{o}}(\ell)$.
 \end{proof}

Now, we introduce another vertex algebra among the main objects of this paper.

\begin{dfnt}
{\em Let $\ell$ be a nonnegative integer. Set
\begin{align}
V^{\rm int}_{\wh{\ft}(\fg,\mu)^{o}}(\ell)=V_{\wh{\ft}(\fg,\mu)^{o}}(\ell)/J(\ell),
\end{align}
which is a $\Z$-graded vertex algebra.}
\end{dfnt}

We have:

\begin{prpt}\label{prop:vavsmod}
Let $\ell\in \C$ and let $W$ be a vector space. Then a level $\ell$ restricted  $\wh{\ft}(\fg,\mu)^o$-module structure
on  $W$ amounts to a $V_{\wh{\ft}(\fg,\mu)^{o}}(\ell)$-module structure $Y_W(\cdot,z)$ such that
\begin{align}
Y_W(a,z)=a(z)\quad \text{ for }a\in \mathcal A_\fg.
\end{align}
Furthermore, restricted and integrable $\wh{\ft}(\fg,\mu)^o$-modules of level $\ell$ (which is necessarily
a nonnegative integer)  correspond exactly to  $V^{\rm int}_{\wh{\ft}(\fg,\mu)^{o}}(\ell)$-modules.
\end{prpt}

\begin{proof} The first assertion follows immediately from Lemma \ref{lem:modofvla}.
For the second assertion, let $W$ be a restricted and integrable $\wh{\ft}(\fg,\mu)^o$-module of level $\ell$, and let
 $a\in t_1^{m}\fg_{\al}$ with $m\in \Z,\ \al\in \Delta$.
In view of Lemma \ref{lem:charinto},  we have
\begin{align*}
Y_W(a,z)^{\epsilon_\al\ell+1}=a(z)^{\epsilon_\al\ell+1}=0\quad\te{on}\ W.
\end{align*}
By a result of \cite{DL}, we get $Y_W((a_{-1})^{\epsilon_\al\ell+1}\bm{1},z)=0$.
Then it follows that  $W$ is naturally a $V^{\rm int}_{\wh{\ft}(\fg,\mu)^{o}}(\ell)$-module.
On the other hand, let $W$ be a
$V^{\rm int}_{\wh{\ft}(\fg,\mu)^{o}}(\ell)$-module.
Again by \cite{DL}, we have
\begin{align*}
Y_W(a,z)^{\epsilon_\al\ell+1}=Y_W((a_{-1})^{\epsilon_\al\ell+1}\bm{1},z)=0\quad\te{on}\ W.
\end{align*}
Viewing $W$ as a $\wh{\ft}(\fg,\mu)^o$-module we have
\begin{align*}
a(z)^{\epsilon_\al\ell+1}=0\quad\te{on}\ W.
\end{align*}
Then  $W$ is integrable by Lemma \ref{lem:charinto}.
\end{proof}

\section{Restricted $\wh{\ft}(\fg,\mu)$-modules and $\phi$-coordinated
$V_{\wh{\ft}(\fg,\mu)^{o}}(\ell)$-modules}

In this section,  we assume that $\fg$ is a finite-dimensional simple Lie algebra
with a Cartan subalgebra $\fh$ as in Section 3.4.
As the main results, we give a canonical connection between restricted (resp. integrable and restricted)
  $\wh{\ft}(\fg,\mu)$-modules of level $\ell$ and $\phi$-coordinated
modules for $V_{\wh{\ft}(\fg,\mu)^{o}}(\ell)$ (resp. $V^{\rm int}_{\wh{\ft}(\fg,\mu)^{o}}(\ell)$).

\subsection{$\phi$-coordinated modules for vertex algebras}

We here recall   from \cite{L} the notion of $\phi$-coordinated module for a vertex algebra and
some basic results.

\begin{dfnt}\label{defcoor}
{\em Let $V$ be a vertex algebra. A {\em $\phi$-coordinated $V$-module} is a vector space $W$
 equipped with a linear map
\begin{eqnarray*}
Y_{W}(\cdot,z):&& V\rightarrow \mathrm{Hom}(W,W((z)))\subset (\mathrm{End}
W)[[z,z^{-1}]]\\
&&v\mapsto Y_{W}(v,z),
\end{eqnarray*}
satisfying the conditions that $Y_W({\bf 1},z)=1_W$ and that for $u,v\in V$, there exists $k\in \N$ such that
\begin{align}\label{Lcommutator}
&(z_1-z_2)^k Y_W(u,z_1) Y_W(v,z_2)\in \mathrm{Hom}(W, W((z_1,z_2))),\\
&(z_2e^{z_0}-z_2)^k Y_W(Y(u,z_0)v,z_2)=\((z_1-z_2)^kY_W(u,z_1)Y_W(v,z_2)\)|_{z_1=z_2e^{z_0}}.
\end{align}}
\end{dfnt}

\begin{remt}
{\em The parameter $\phi$ in Definition \ref{defcoor} refers to the formal series $\phi(x,z)=
x e^z$, which is a particular associate, as defined in \cite{L}, of the one-dimensional additive formal group
(law) $F(x,y)=x+y$.
Taking $\phi(x,z)=x+z$ (the formal group law itself) in Definition \ref{defcoor}, we get an equivalent definition
 of a module  for $V$ (cf. \cite{LTW}).}
\end{remt}

\begin{remt}\label{rem:cancel}{\em
Let $(W,Y_W)$ be a $\phi$-coordinated module for a vertex algebra $V$.
It was proved in \cite[Lemma 3.6]{L} that for $u,v\in V$,
\begin{align*}
(z_2e^{z_0}-z_2)^{k_0} Y_W(Y(u,z_0)v,z_2)=\((z_1-z_2)^{k_0}Y_W(u,z_1)Y_W(v,z_2)\)|_{z_1=z_2e^{z_0}}
\end{align*}
holds for any $k_0\in \N$ such that
\begin{align*}
(z_1-z_2)^{k_0} Y_W(u,z_1) Y_W(v,z_2)\in \mathrm{Hom}(W, W((z_1,z_2))).
\end{align*}}
\end{remt}

\begin{lemt}\label{lem:keryw}
Let $V$ be a vertex algebra and let $(W,Y_W)$ be a $\phi$-coordinated $V$-module. Then
$\ker Y_W$ is an ideal of $V$, where $\ker Y_W=\{ v\in V\, |\,  Y_W(v,z)=0\}$.
\end{lemt}

\begin{proof} Let $u,v\in V$ such that either $Y_W(u,z)=0$ or $Y_W(v,z)=0$.
Then by Remark \ref{rem:cancel} with $k_0=0$ we get
\begin{align*}
Y_W(Y(u,z_0)v,z_2)=Y_W(u,z_1)Y_W(v,z_2)|_{z_1=z_2e^{z_0}}=0,
\end{align*}
which implies that $u_n v\in \ker Y_W$ for all $n\in \Z$. Thus $\ker Y_W$ is an ideal of $V$.
\end{proof}

We have (see \cite[Lemma 3.7]{L}, \cite{L2017}):

\begin{lemt}
Let $V$ be a vertex algebra and let $(W,Y_W)$ be any $\phi$-coordinated $V$-module. Then
\begin{align}\label{Dphimod}
Y_W(L(-1) v,z)=z\frac{d}{d z}Y_W(v,z)\quad \text{ for }v\in V,
\end{align}
where $L(-1)$ is the linear operator on $V$ defined by $L(-1)v=v_{-2}\mathbf{1}$ for $v\in V$.
\end{lemt}

The following was proved in \cite[Proposition 5.9]{L} and \cite[Theorem 3.19, Lemma 3.29]{BLP}:

\begin{prpt}\label{prop:Bordcomm}
Let $V$ be a vertex algebra and let $(W,Y_W)$ be a $\phi$-coordinated $V$-module. Then
for $u,v\in V$,
\begin{align}\label{Bordcomm1}
[Y_W(u,z),Y_W(v,w)]=\sum_{j\in \N} Y_W(u_jv, w) \frac{1}{j!}\(\wpw\)^j \dwz.
\end{align}
On the other hand, if $W$ is faithful and if
\begin{align}\label{Bordcomm2}
[Y_W(u,z),Y_W(v,w)]=\sum_{j\in \N} Y_W(A_j, w) \frac{1}{j!}\(\wpw\)^j \dwz
\end{align}
with $A_j\in V$, then $A_j=u_jv$ for $j\in \N$.
\end{prpt}

Let $(W,Y_W)$ be a $\phi$-coordinated $V$-module. For $v\in V$, write
\begin{align}
Y_{W}(v,z)=\sum_{n\in \Z}v[n]z^{-n}.
\end{align}
Then
\begin{align}
[u[m],v[n]]=\sum_{j\ge 0}m^j (u_jv)[m+n]
\end{align}
on $W$ for $u,v\in V,\ m,n\in \Z$.

\begin{dfnt}
{\em  Let $V$ be a vertex algebra. A  {\em $\Z$-graded $\phi$-coordinated $V$-module}
 is a $\phi$-coordinated $V$-module  $W$ equipped  with a $\Z$-grading
$W=\oplus_{n\in \Z} W(n)$ such that
\begin{eqnarray}
u[m]W(n)\subset W(m+n)\ \ \mbox{ for }u\in V,\ m,n\in \Z.
\end{eqnarray}}
\end{dfnt}

Let $W$ be a vector space. Formal series
$a(x), b(x)\in \CE(W)$ are said to be {\em local} if there exists a nonnegative integer $k$ such that
\[(z-w)^k[a(z),b(w)]=0.\]
A subset  $U$ of $\CE(W)$ is said to be {\em local} if any $a(z), b(z)\in U$ are local.
For a  local pair $(a(x), b(x))$ as above,
define $a(x)_n^\phi b(x)\in \CE(W)$ for $n\in \Z$ in terms of generating function
\[\CY_\CE^\phi (a(x),z)b(x):=\sum_{n\in \Z}\(a(x)_n^\phi b(x)\)z^{-n-1}\in \CE(W)[[z,z^{-1}]]\]
by
\[\CY_\CE^\phi (a(x),z)b(x)=(x e^z-x)^{-k}\((x_1-x)^k a(x_1) b(x)\)|_{x_1=xe^z}.\]
A local subspace $U$ of $\CE(W)$ is said to be {\em $\CY_\CE^\phi$-closed}  if  $a(x)_n^\phi b(x)\in U$
for all $a(x),b(x)\in U,\  n\in \Z$.

The following result was obtained in  (\cite[Theorem 5.4]{L}, \cite[Theorem 2.9]{BLP}):

\begin{prpt} \label{locasetva}
Let $U$ be a local subset of $\CE(W)$.
Then there exists a $\CY_\CE^\phi$-closed local subspace of $\CE(W)$ which contains $1_W$ and $U$.
Denote by $\<U\>_\phi$ the smallest such local subspace.
Then $(\<U\>_\phi,\CY_\CE^\phi,1_W)$ is a vertex algebra with $U$
as a generating subset and  $(W,Y_W)$ is a faithful  $\phi$-coordinated module with
 \begin{align}\label{phiact1}
Y_W(a(x),z)=a(z)\quad\te{for}\  a(x)\in \<U\>_\phi.
\end{align}
\end{prpt}

\begin{remt}\label{rem:locasetva}
{\em Let $U$ be a local subset of $\CE(W)$. For $a(z)\in \<U\>_{\phi}$,
by \eqref{Dphimod} we have
\begin{align*}
L(-1)a(z)=Y_W(L(-1)a(x),z)=z\frac{d}{dz}Y_W(a(x),z)=z\frac{d}{dz}a(z).
\end{align*}
This in particular implies $\(z\frac{d}{d z}\)a(z)\in \<U\>_\phi$.}
\end{remt}

\subsection{$\phi$-coordinated modules for vertex Lie algebras}
Let $(\mathcal L,\mathcal A,\mathcal C,\rho)$ be a vertex Lie algebra
and let $\gamma$ be a linear functional on $\mathcal C$.
In this section, we introduce a notion of $\phi$-coordinated $\CL$-module of level $\gamma$
and prove that $\phi$-coordinated $\CL$-modules exactly amount to $\phi$-coordinated $V_\CL(\gamma)$-modules.

\begin{dfnt}
{\em A {\em $\phi$-coordinated $\CL$-module of level $\gamma$}
 is a vector space $W$ equipped with a linear map
 \begin{align*}
 \psi_W(\cdot,z):\   \mathcal A\oplus {\mathcal C} \rightarrow \text{Hom }(W,W((z)));\quad u\mapsto \psi_W(u,z)
 \end{align*}
 such that $ \psi_W(c,z)=\gamma(c)$ for $c\in \mathcal{C}$ and for $a,b\in \mathcal A$,
\begin{align}\label{Wcommutator}
[\psi_W(a,z),\psi_W(b,w)]=&\sum_{i=0}^r \psi_W(f_i(a,b),w)\frac{1}{i!}\(\wpw\)^{i} \dwz,
\end{align}
where $f_i(a,b)\in \C[\partial]\mathcal A\oplus \mathcal C$ are the same as in
(\ref{az-bw}) and where the linear map $\psi_W(\cdot,z)$ is extended to $\C[\partial]{\mathcal{A}}\oplus {\mathcal{C}}$
linearly by
\begin{align}\label{partial-extension}
\psi_W(\partial^ia,z)=\left(z\frac{d}{dz}\right)^i\psi_W(a,z)\quad \text{for }i\in \N,\ a\in {\mathcal{A}}.
\end{align}}
\end{dfnt}

Recall that  $\pi:
\  \C[\partial]{\mathcal{A}}\oplus {\mathcal{C}}\rightarrow V_{\mathcal L}(\gamma)$
is the linear map defined by
\begin{align}
\pi(c)=\gamma(c){\bf 1},\quad \pi (\partial^ia)=L(-1)^ia\quad \text{for }c\in {\mathcal{C}},\  i\in \N,\ a\in {\mathcal{A}}
\end{align}
and that
\begin{align*}
Y(\pi (X),z)=X(z)\quad \text{ for }X\in \C[\partial]{\mathcal{A}}\oplus {\mathcal{C}}.
\end{align*}

As the main result of this section, we have:

\begin{prpt}\label{prop:cdvla}
Let $(W,Y_W)$ be a $\phi$-coordinated module for the vertex algebra $V_{\mathcal L}(\gamma)$. Then
$W$ is a $\phi$-coordinated module of level $\gamma$ for the vertex Lie algebra $\CL$ with
$\psi_W(c,z)=\gamma(c)$ for $c\in {\mathcal{C}}$ and
\begin{align}\label{phi=Y}
\psi_W(a,z)=Y_W(a,z)\quad \te{for}\ a\in \mathcal A.
\end{align}
On the other hand,  for any $\phi$-coordinated $\CL$-module $(W,\psi_W)$ of level $\gamma$,
there exists a $\phi$-coordinated $V_\CL(\gamma)$-module structure $Y_W(\cdot,z)$ on $W$,
which is uniquely determined by
\begin{align}\label{Y=phi}
Y_W(a,z)=\psi_W(a,z)\quad \te{for}\ a\in \mathcal A.
\end{align}
\end{prpt}

\begin{proof} Let $a,b\in \mathcal A$ and assume that
$[a(z),b(w)]$ has the expression as in \eqref{Lcommutator}.
Note that for any $c\in \mathcal C$, $c$ acts on $V_{\mathcal L}(\gamma)$ as scalar $\gamma(c)$.
 Then using (\ref{Y-extension-A}) we get
\begin{align*}
&[Y(a,z),Y(b,w)]=[a(z),b(w)]\\
=&\sum_{i=0}^r f_i(a,b)(w)\frac{1}{i!}\(\pw\)^{i}z^{-1}\dwz\\
=&\sum_{i=0}^r Y(\bar{f}_i(a,b),w)\frac{1}{i!}\(\pw\)^{i}z^{-1}\dwz,
\end{align*}
where $\bar{f}_i(a,b)=\pi (f_i(a,b))\in V_{\mathcal L}(\gamma)$.
In view of Proposition \ref{prop:Bordcomm}, we have
\begin{align*}
[Y_W(a,z),Y_W(b,w)]=\sum_{i=0}^r Y_W(\bar{f}_i(a,b),w)\frac{1}{i!}\(w\pw\)^{i}\dwz.
\end{align*}
Therefore, $W$ is a $\phi$-coordinated $\CL$-module of level $\gamma$ with
$\psi_W(a,z)=Y_W(a,z)$ for $a\in \mathcal A$.
This proves the first assertion.

For the second assertion, the uniqueness is clear as $\mathcal A$ generates $V_{\mathcal L}(\gamma)$
as a vertex algebra, so it remains to establish the existence.
Set
\[U=\{\psi_W(a,x)\mid a\in \mathcal A\}.\]
 From \eqref{Wcommutator}, $U$ is a local subspace of $\CE(W)$. Then by Proposition \ref{locasetva},
$U$ generates a vertex algebra $\<U\>_\phi$ and $W$ is a faithful $\phi$-coordinated module for $\<U\>_\phi$
such that
\begin{align}\label{YWphiW}
Y_W(\psi_W(a,x),z)=\psi_W(a,z)\quad\te{for}\ a\in \mathcal A.
\end{align}
For $a,b\in {\mathcal{A}}$, by the second assertion of Proposition \ref{prop:Bordcomm}
we have
 \begin{align*}\label{Wcommutator}
&[\CY_\CE^\phi(\psi_W(a,x),z_1),\CY_\CE^\phi(\psi_W(b,x),z_2)]\\
=\ &\sum_{i=0}^r\CY_\CE^\phi\(\psi_W(f_i(a,b),x),z_2\)\frac{1}{i!}\(\frac{\partial}{\partial z_2}\)^{i} z_1^{-1}\delta\(\frac{z_2}{z_1}\).
\end{align*}
From Remark \ref{rem:locasetva}, it follows that for $a\in \mathcal A,\ i\in \N$, $\(x\frac{\partial}{\partial x}\)^i\psi_W(a,x)\in
 \<U\>_\phi$ and
 \begin{align*}
 \CY_\CE^\phi\(\(x\frac{\partial}{\partial x}\)^i\psi_W(a,x),z\)=
 \CY_\CE^\phi\(L(-1)^i(\psi_W(a,x)),z\)=\left(\frac{\partial}{\partial z}\right)^i\CY_\CE^\phi(\psi_W(a,x),z).
 \end{align*}
 Then it follows that  $\<U\>_\phi$ is an $\mathcal L$-module  with
\[a(z)= \CY_\CE^\phi(\psi_W(a,x),z),\quad
c= \gamma(c)\quad \te{for}\ a\in \mathcal A,\  c\in \mathcal C.\]

 Since $U$ generates the vertex algebra $\<U\>_\phi$, it follows that $\<U\>_\phi$ is generated by $1_W$ as
 an $\mathcal L$-module. Moreover, the creation property of the vertex algebra $\<U\>_\phi$ implies that
 $\mathcal L_{+}1_W=0$.
 Then there exists an $\mathcal L$-module homomorphism
 \begin{align}
 \varphi: V_{\mathcal L}(\gamma)\rightarrow \<U\>_\phi
 \end{align}
 with $\varphi(\mathbf{1})=1_W$.
 For $a\in \mathcal A$, we have
 \[\varphi(Y(a,z)\mathbf{1})=\varphi(a(z)\mathbf{1})
 =\CY_\CE^\phi(\psi_W(a,x),z)1_W,\]
 which by setting $z=0$ gives
 \begin{align}\label{phiaphiW}
 \varphi(a)=\psi_W(a,x).
 \end{align}
Thus,
 \begin{align*}
 \varphi(Y(a,z)v)=\varphi(a(z)v)=\CY_\CE^\phi(\psi_W(a,x),z)\varphi(v)
 =\CY_\CE^\phi(\varphi(a),z)\varphi(v)
 \end{align*}
 for $a\in \mathcal A,\  v\in V_{\mathcal L}(\gamma)$.
As $\mathcal A$ generates $V_{\mathcal L}(\gamma)$ as a vertex algebra, by \cite[Proposition 5.7.9]{LL}
 $\varphi$ is a vertex algebra homomorphism.
Then $W$ becomes a $\phi$-coordinated $V_{\mathcal L}(\gamma)$-module
 via $\varphi$. Moreover, by \eqref{YWphiW} and \eqref{phiaphiW}, we have
 \begin{align*}
 Y_W(a,z)=Y_W(\varphi(a),z)=Y_W(\psi_W(a,x),z)=\psi_W(a,z)\quad\te{for}\ a\in \mathcal A.
 \end{align*}
 This proves the existence and hence concludes the proof.
\end{proof}


\subsection{$\wh{\ft}(\fg,\mu)$-modules and $\phi$-coordinated modules for $V_{\wh{\ft}(\fg,\mu)^{o}}(\ell)$ }

 Here, we give a canonical association of restricted $\wh{\ft}(\fg,\mu)$-modules of level $\ell$
 with $\phi$-coordinated modules for the vertex algebra $V_{\wh{\ft}(\fg,\mu)^{o}}(\ell)$.

Just as for Lie algebra $\wh{\ft}(\fg,\mu)^o$, we formulate the following notions:

\begin{dfnt}
{\em A $\wh{\ft}(\fg,\mu)$-module $W$ is said to be {\em restricted}  if
$a[z]\in \CE(W)$ for every $a\in \mathcal A_\fg$,   is said to be of {\em level} $\ell\in \C$
if $\rk_0$ acts as scalar $\ell$,     and is said to be {\em integrable} if
$(t_0^{m_0}t_1^{m_1}\fg_\al)$ acts locally nilpotently on $W$ for any $m_0,m_1\in \Z,\  \al\in \Delta$.}
\end{dfnt}

By the same arguments in the proof of Lemma \ref{lem:charinto}, we have:

\begin{lemt}\label{lem:charint}
Let $W$ be a restricted $\wh{\ft}(\fg,\mu)$-module  of level $\ell\in \C$. Then $W$ is integrable if and only if
$\ell$ is a nonnegative integer and
\begin{align*}
a[z]^{\epsilon_\al\ell+1}=0\quad \te{on}\ W\  \te{for all } a\in t_1^m\fg_{\al}\  \te{with } m\in \Z, \ \al\in \Delta,
\end{align*}
where $\epsilon_\al=\frac{2}{\<\al,\al\>}$.
\end{lemt}

The following is an analogue of a result of Dong and Lepowsky (see \cite{DL}):

\begin{lemt} \label{lem:nilponent}
Let $V$ be a vertex algebra and let $(W,Y_W)$ be a $\phi$-coordinated $V$-module.
Assume that  $k$ is a nonnegative integer and  $a\in V$ such that $a_na=0$ for $n\ge 0$.
If $(a_{-1})^{k+1}{\mathbf{1}}=0$ in $V$, then
\begin{eqnarray}\label{ePnilpotent}
Y_W(a,z)^{k+1}=0.
\end{eqnarray}
On the other hand, the converse is also true if $(W,Y_W)$ is faithful.
\end{lemt}

\begin{proof} Note that for any $u,v\in V$, if $u_nv=0$ for $n\ge 0$,
then by Proposition \ref{prop:Bordcomm} we have
 $[Y_W(u,z_1),Y_W(v,z_2)]=0$.
 Since $a_ia=0$ for all $i\ge 0$, we have $[a_m,a_n]=0$ on $V$ for all $m,n\in \Z$.
Then for $b=(a_{-1})^{r}{\bf 1}$ with $r\in \N$ we have $a_ib=0$ for $i\ge 0$, and hence
$[Y_W(a,z_1),Y_W(b,z_2)]=0$.
By Remark \ref{rem:cancel}, we have
 \begin{align*}
 Y_W(Y(a,z_0)b,z_2)=Y_W(a,z_1)Y_W(b,z_2)|_{z_1=z_2e^{z_0}}.
 \end{align*}
 Setting $z_0= 0$ we get
\[Y_W(a_{-1}b,z_2)=Y_W(a,z_2)Y_W(b,z_2).\]
It then follows that
\[Y_W((a_{-1})^{k+1}{\mathbf{1}},z_2)=Y_W(a,z_2)^{k+1}.\]
Therefore, if $(a_{-1})^{k+1}{\mathbf{1}}=0$, then $Y_W(a,z_2)^{k+1}=0$.
Conversely, if $Y_W(a,z_2)^{\ell+1}=0$ and if $W$ is faithful, then we have $(a_{-1})^{\ell+1}{\mathbf{1}}=0$.
\end{proof}

Recall the vertex Lie algebra $(\wh{\ft}(\fg,\mu)^{o},\mathcal A_\fg, \C\rk_0,\rho)$
 from  Proposition \ref{prop:vas}.
 By Propositions \ref{prop:hgcomm} and \ref{prop:hbcomm}, we immediately have:

\begin{lemt}\label{modfgfb}
Let $\ell\in \C$ and let $W$ be a vector space.
Then the structure of a restricted $\wh{\ft}(\fg,\mu)$-module of level $\ell$  on $W$
amounts to the structure $\psi_W(\cdot,z)$ of a
$\phi$-coordinated $\wh{\ft}(\fg,\mu)^{o}$-module of level $\ell$ on $W$ with
\begin{align*}
a[z]=\psi_W(a,z)\quad\te{for}\ a\in \mathcal A_\fg.
\end{align*}
\end{lemt}

Now, we are in a position to present our first main result of this paper.

\begin{thm}\label{thm:main1}
Let $\ell$ be any complex number. For any restricted $\wh{\ft}(\fg,\mu)$-module $W$ of level $\ell$, there
exists a $\phi$-coordinated $V_{\wh{\ft}(\fg,\mu)^{o}}(\ell)$-module structure $Y_W(\cdot,x)$ on $W$, which is uniquely determined by
\begin{align}\label{mainaction1}
Y_W(a,z)=a[z]\quad \te{for}\ a\in \mathcal A_\fg.
\end{align}
On the other hand, for any $\phi$-coordinated $V_{\wh{\ft}(\fg,\mu)^{o}}(\ell)$-module $(W,Y_W)$, $W$ is a restricted
$\wh{\ft}(\fg,\mu)$-module of level $\ell$ with
\begin{align}\label{mainaction2}
a[z]=Y_W(a,z)\quad \te{for}\ a\in \mathcal A_\fg.
\end{align}
Furthermore,  restricted and integrable $\wh{\ft}(\fg,\mu)$-modules of level $\ell$ (which is necessarily
a nonnegative integer) exactly correspond to $\phi$-coordinated $V^{\rm int}_{\wh{\ft}(\fg,\mu)^{o}}(\ell)$-modules.
\end{thm}

\begin{proof} Recall the vertex Lie algebra $(\wh{\ft}(\fg,\mu)^{o},\mathcal A_\fg, \C\rk_0,\rho)$
and the linear functional $\gamma_\ell$ defined in the proof of Proposition \ref{prop:vas}. We have
$V_{\wh{\ft}(\fg,\mu)^{o}}(\gamma_\ell)=V_{\wh{\ft}(\fg,\mu)^{o}}(\ell)$.
Let $W$ be a restricted $\wh{\ft}(\fg,\mu)$-module of level $\ell$.
From Lemma \ref{modfgfb}, $W$ is a $\phi$-coordinated module of level $\ell$ for the vertex Lie algebra
$\wh{\ft}(\fg,\mu)^{o}$ with
\begin{align*}
\psi_W(a,z)=a[z]\quad\te{for}\ a\in \mathcal A_\fg.
\end{align*}
Then by Proposition \ref{prop:cdvla}
there exists a $\phi$-coordinated module structure $Y_W(\cdot,z)$ on $W$ for the vertex algebra
$V_{\wh{\ft}(\fg,\mu)^{o}}(\ell)\ (=V_{\wh{\ft}(\fg,\mu)^{o}}(\gamma_\ell))$ such that
\begin{align*}
Y_W(a,z)=\psi_W(a,z)=a[z]\quad\te{for}\ a\in \mathcal A_\fg.
\end{align*}
On the other hand, let $(W,Y_W)$ be a $\phi$-coordinated $V_{\wh{\ft}(\fg,\mu)^{o}}(\ell)$-module.
Then it follows from the first assertion of Proposition \ref{prop:cdvla} that $W$ is a
$\phi$-coordinated $\wh{\ft}(\fg,\mu)^{o}$-module of level $\ell$, and hence by Lemma \ref{modfgfb}
$W$ is a restricted $\wh{\ft}(\fg,\mu)$-module of level $\ell$ with
\begin{align*}
a[z]=\psi_W(a,z)=Y_W(a,z)\quad\te{for}\ a\in \mathcal A_\fg.
\end{align*}

Next, we prove the furthermore assertion.
Suppose $W$ is an integrable and restricted $\wh{\ft}(\fg,\mu)$-module of level $\ell$.
 By the first part of this theorem, $W$
is a $\phi$-coordinated module for $V_{\wh{\ft}(\fg,\mu)^{o}}(\ell)$.
Then $W$ is naturally a faithful $\phi$-coordinated module for the quotient vertex algebra
$V_{\wh{\ft}(\fg,\mu)^{o}}(\ell)/\ker Y_W$.
Let $a\in t_1^{m}\fg_{\al}$ with $m\in \Z,\ \al\in \Delta$.  Notice that
$[a(z),a(w)]=0$ in $\wh{\ft}(\fg,\mu)^{o}[[z^{\pm 1},w^{\pm 1}]]$, which implies
\begin{align}\label{xnx}
a_n a=0\quad \te{in }V_{\wh{\ft}(\fg,\mu)^{o}}(\ell)\   \te{for all}\ n\ge 0.
\end{align}
As $W$ is integrable, by Lemma \ref{lem:charint} we have $a[z]^{\epsilon_\al\ell+1}=0$ on $W$, so that
\begin{align}\label{ywxzl}
Y_W(a,z)^{\epsilon_\al\ell+1}=a[z]^{\epsilon_\al\ell+1}=0\quad\te{on }W.
\end{align}
Then by Lemma \ref{lem:nilponent} we have
\begin{align*}
(a_{-1})^{\epsilon_\al\ell+1}{\bf 1}=0\quad\te{in}\ V_{\wh{\ft}(\fg,\mu)^{o}}(\ell)/\ker Y_W.
\end{align*}
It follows that $J(\ell)\subset \ker Y_W$.
 Consequently,  $W$ is a $\phi$-coordinated $V^{\rm int}_{\wh{\ft}(\fg,\mu)^{o}}(\ell)$-module.

On the other hand, let $W$ be a $\phi$-coordinated $V^{\rm int}_{\wh{\ft}(\fg,\mu)^{o}}(\ell)$-module.
From the first part,  $W$ is a restricted $\wh{\ft}(\fg,\mu)$-module of level $\ell$ with
\begin{align}\label{xz=ywxz}
u[z]=Y_W(u,z)\quad \te{ for }u\in \mathcal A_\fg.
\end{align}
Let $a\in t_1^{m}\fg_{\al}$ with $m\in \Z,\ \al\in \Delta$.
Denote the image of $a$ in $V^{\rm int}_{\wh{\ft}(\fg,\mu)^{o}}(\ell)$ still by $a$.
From the definition of $V^{\rm int}_{\wh{\ft}(\fg,\mu)^{o}}(\ell)$ we have
 $(a_{-1})^{\epsilon_\al\ell+1}{\bf 1}=0$ in $V^{\rm int}_{\wh{\ft}(\fg,\mu)^{o}}(\ell)$.
Then by Lemma \ref{lem:nilponent} we have $Y_W(a,z)^{\epsilon_\al\ell+1}=0$ on $W$.
Thus  $a[z]^{\epsilon_\al\ell+1}=0$ on $W$. Therefore,
 $W$ is an integrable $\wh{\ft}(\fg,\mu)$-module by Lemma \ref{lem:charint}.
\end{proof}

\section{Irreducible $V_{\wh{\ft}(\fg,\mu)^{o}}(\ell)$-modules}
In this section, we classify irreducible bounded $\N$-graded modules
(with finite-dimensional homogenous spaces) for Lie algebra $\wh{\ft}(\fg,\mu)^{o}$ and
vertex algebras $V_{\wh{\ft}(\fg,\mu)^{o}}(\ell)$ and $V^{\rm int}_{\wh{\ft}(\fg,\mu)^{o}}(\ell)$.

\subsection{Realization of irreducible $\wh{\ft}(\fg,\mu)^{o}$-modules}
Here, we give an explicit realization of a class of irreducible $\wh{\ft}(\fg,\mu)^{o}$-modules of nonzero levels.
Recall that both  $\wt{\mathfrak{t}}(\fg,\mu)^o$ and  $\wh{\mathfrak{t}}(\fg,\mu)^o$
are $\Z$-graded subalgebras of ${\mathcal{T}}(\fg,\mu)$ with the grading given by the adjoint action of $-\rd_0$.
The degree-zero subalgebra of $\wh{\mathfrak{t}}(\fg,\mu)^o$ is
\begin{align}
\mathcal L_{(0)}:=\wh{\mathfrak{t}}(\fg,\mu)^o_{(0)}
=\,&\wt\fg_1\oplus \sum_{n\in \Z^\times}(\C \rd_{0,n}+\C \rk_{0,n})\oplus \C\rk_0\\
=\,&\C[t_1,t_1^{-1}]\ot \fg + \sum_{n\in \Z}(\C\rd_{0,n}+\C t_1^n\rk_0)+
\C\rk_1\nonumber
\end{align}
(recall that $\rd_{0,0}=\rd_1$, $\rk_{0,0}=0$ and $n\rk_{0,n}=t_1^n\rk_0$ for $n\in \Z^{\times}$), where  $\rk_1$ is central and
\begin{align}
&[t_1^m\rk_0,t_1^n\rk_0]=0,\    \   [t_1^m\rk_0,t_1^n\ot u]=0, \\
&[\rd_{0,m},\rd_{0,n}]=(n-m)\rd_{0,m+n}+2\mu m^3\delta_{m+n,0}\rk_1,\\
\label{rd0mt1nrk0}&[\rd_{0,m},t_1^n\rk_0]=nt_1^{m+n}\rk_0+m\delta_{m+n,0}\rk_1,\\
&[\rd_{0,m},t_1^n\ot u]=n(t_1^{m+n}\ot u),\\
&[t_1^m\ot u, t_1^n\ot v]=t_1^{m+n}\ot [u,v]+m\delta_{m+n,0}\<u,v\>\rk_1
\end{align}
for $m,n\in \Z,\  u,v\in \fg$. Note that $\rk_0\ (=t_1^0\rk_0)$ is also central.

Set
\begin{align*}
P_+=\{\lambda\in \fh^\ast\mid \lambda(\al^{\vee})\in \N\ \te{for}\ \al\in \Pi\},
\end{align*}
the set of dominant integral weights  of $\fg$.
Let $\lambda\in P_+$ and let $\ell,\al,\beta\in \C$ with $\ell\ne 0$.
Denote by $L_{\fg}(\lambda)$ the (finite-dimensional) irreducible highest weight $\fg$-module
with highest weight $\lambda$.
 We then define an $\mathcal L_{(0)}$-module $T_{\ell,\lambda,\al,\beta}$, where
 \begin{align}\label{def-T-ell-lambda-al-be}
 T_{\ell,\lambda,\al,\beta}=
 \C[q,q^{-1}]\ot L_{\fg}(\lambda)
 \end{align}
 as a vector space (with $q$ an indeterminate) and where $\rk_1$ acts trivially,
\begin{align}
&(t_1^ma) (q^n\ot w)=q^{m+n}\ot a w,\quad (t_1^m\rk_0)(q^n\ot w)=\ell q^{m+n}\ot w,\label{e5.8}\\
&\hspace{2cm}\rd_{0,m} (q^n\ot w)
=(n+\al+\beta m)q^{m+n}\ot w\label{e5.9}
\end{align}
for $m,n\in \Z,\ a\in \fg$, $w\in L_{\fg}(\lambda)$. In particular, $\rk_0$ acts as scalar $\ell$.
(It is known that this indeed gives an $\mathcal L_{(0)}$-module.)
Also, it can be readily seen that $T_{\ell,\lambda,\al,\beta}$ is irreducible.
We are particularly interested in $T_{\ell,\lambda,\al,\beta}$ with $\lambda=0$ and $\al=\beta=0$.

\begin{dfnt}\label{def-T-ell}
{\em  For $\ell\in \C^{\times}$, define an irreducible $\mathcal L_{(0)}$-module $T_\ell$,
where $T_\ell=\C[q,q^{-1}]$ as a vector space and
where $\C[t_1,t_1^{-1}]\ot \fg=0$, $\rk_1=0$, $\rk_0=\ell$, and
\begin{align}
 (t_1^m\rk_0)\cdot q^n=\ell q^{m+n},\  \  \rd_{0,m} \cdot q^n
=nq^{m+n}\quad \text{for }m,n\in \Z.
\end{align}}
\end{dfnt}

Set $\wh{\mathfrak{t}}(\fg,\mu)^o_{(\pm)}=\sum_{n>0}\wh{\mathfrak{t}}(\fg,\mu)^o_{(\pm n)}$.
Let $\lambda,\ell,\al,\beta$ be given as before.
We make $T_{\ell,\lambda,\al,\beta}$ a $(\wh{\mathfrak{t}}(\fg,\mu)^o_{(-)}+ \mathcal L_{(0)})$-module with
$\wh{\mathfrak{t}}(\fg,\mu)^o_{(-)}$ acting trivially.
Then form an induced $\wh{\mathfrak{t}}(\fg,\mu)^o$-module
\begin{align}
V_{\wh{\mathfrak{t}}(\fg,\mu)^o}(T_{\ell,\lambda,\al,\beta}):=
\U(\wh{\mathfrak{t}}(\fg,\mu)^o)\ot_{\U(\wh{\mathfrak{t}}(\fg,\mu)^o_{(-)}+ \mathcal L_{(0)})}T_{\ell,\lambda,\al,\beta}.
\end{align}
Note that $V_{\wh{\mathfrak{t}}(\fg,\mu)^o}(T_{\ell,\lambda,\al,\beta})$
is naturally an $\N$-graded $\wh{\mathfrak{t}}(\fg,\mu)^o$-module with $\deg T_{\ell,\lambda,\al,\beta}=0$.
Denote by $L_{\wh{\mathfrak{t}}(\fg,\mu)^o}(T_{\ell,\lambda,\al,\beta})$
the graded irreducible quotient   of
$V_{\wh{\mathfrak{t}}(\fg,\mu)^o}(T_{\ell,\lambda,\al,\beta})$.
In fact, $L_{\wh{\mathfrak{t}}(\fg,\mu)^o}(T_{\ell,\lambda,\al,\beta})$ is an irreducible $\wh{\mathfrak{t}}(\fg,\mu)^o$-module.

Recall that  $\wt{\mathfrak{t}}(\fg,\mu)^o$ and $\wh{\mathfrak{t}}(\fg,\mu)^o$
share the same degree-zero subalgebra $\mathcal L_{(0)}$.
Similarly, from the irreducible $\mathcal L_{(0)}$-module $T_{\ell,\lambda,\al,\beta}$,
we have an induced $\wt{\mathfrak{t}}(\fg,\mu)^o$-module
$V_{\wt{\mathfrak{t}}(\fg,\mu)^o}(T_{\ell,\lambda,\al,\beta})$ and its irreducible quotient $L_{\wt{\mathfrak{t}}(\fg,\mu)^o}(T_{\ell,\lambda,\al,\beta})$.
Note that as $\rd_1=\rd_{0,0}$ is semisimple on $T_{\ell,\lambda,\al,\beta}$, it is
semisimple on $V_{\wh{\mathfrak{t}}(\fg,\mu)^o}(T_{\ell,\lambda,\al,\beta})$ and
$L_{\wh{\mathfrak{t}}(\fg,\mu)^o}(T_{\ell,\lambda,\al,\beta})$.

In the following, we are going to use a result of Billig to give an explicit realization of these irreducible
$\wh{\mathfrak{t}}(\fg,\mu)^o$-modules $L_{\wh{\mathfrak{t}}(\fg,\mu)^o}(T_{\ell,\lambda,\al,\beta})$.
First, follow \cite{B1} to form a reductive Lie algebra
\begin{align}
\ff=\fg\oplus \C I
\end{align}
with $1$-dimensional center $\C I$. Consider the following (universal) central extension of
$(\mathrm{Der}\, \C[t^{\pm 1}])\ltimes \(\C[t^{\pm 1}]\ot \ff\)$, with underlying space
 \begin{align}
\bar{\ff}=(\mathrm{Der}\ \C[t^{\pm 1}])\ltimes \(\C[t^{\pm 1}]\ot \ff\) \oplus \C \rk\oplus
\C \rk_{I}\oplus \C \rk_{VI}\oplus \C \rk_{Vir}
\end{align}
and with the following Lie bracket relations:
\begin{align*}
[L(m),L(n)]&=(m-n)L(m+n)+\frac{m^3-m}{12}\delta_{m+n,0}\rk_{Vir},\\
[(u+aI)(m),(v+bI)(n)]&=[u,v](m+n)+m\delta_{m+n,0}(\<u,v\>\rk+ab\rk_{I}),\\
[L(m),(u+aI)(n)]&=-n(u+aI)(m+n)-(m^2+m)a\delta_{m+n,0}\rk_{VI}
\end{align*}
for $u,v\in \fg$, $m,n\in \Z,\  a,b\in \C$, where
$$L(n)=-t^{n+1}\frac{\rd}{\rd t}, \quad u(n)=t^n\ot u\quad \text{ for }n\in \Z,\  u\in \ff.$$
This Lie algebra is often referred to as the {\em twisted Virasoro-affine algebra}.

\begin{remt} {\em
Note that Lie algebra $\mathcal L_{(0)}$ is a $1$-dimensional central extension of
$(\mathrm{Der}\, \C[t^{\pm 1}])\ltimes \(\C[t^{\pm 1}]\ot \ff\)$.
It then follows that $\mathcal L_{(0)}$ is isomorphic to a quotient of (the universal central extension) $\bar\ff$.}
\end{remt}

Lie algebra $\bar\ff$ contains as subalgebras the following affine Lie algebra,
Heisenberg algebra, and the Virasoro algebra:
\begin{align*}
&\wh\fg=\C[t,t^{-1}]\ot \fg+ \C\rk,\\
&\textsf{Hei}=\C[t,t^{-1}]\ot \C I+ \C \rk_{I},\\
&\textsf{Vir}=(\mathrm{Der}\, \C[t,t^{-1}])+ \C \rk_{Vir}.
\end{align*}
Furthermore, $\bar{\ff}$ contains as subalgebras the following {\em Virasoro-affine algebra} and
{\em twisted Virasoro-Heisenberg algebra}:
\begin{align}
&\wh\fg\rtimes \textsf{Vir}=(\mathrm{Der}\, \C[t,t^{-1}])\ltimes  \(\C[t,t^{-1}]\ot \fg\) + \C \rk+ \C\rk_{Vir},\\
&\textsf{HVir}=(\mathrm{Der}\, \C[t,t^{-1}])\ltimes  \(\C[t,t^{-1}]\ot \C I\) + \C \rk_{I}
+ \C\rk_{VI}+ \C\rk_{Vir}.
\end{align}

Set
\begin{eqnarray}\label{def-C}
\mathsf{C}=\C \rk+ \C \rk_{I}+ \C \rk_{VI}+ \C \rk_{Vir},
\end{eqnarray}
 the center of $\bar{\ff}$, and set
\begin{align*}
\bar\ff_{+}&=\text{Span}\{u(m),L(n)\mid u\in \ff,\  m\ge 0,\  n\ge -1\},\\
\bar\ff_{-}&=\text{Span}\{u(m),L(n)\mid u\in \ff,\  m< 0,\  n< -1\}.
\end{align*}
From \cite{B1}, $\bar{\ff}$ is  a vertex Lie algebra with triangular  decomposition
\begin{align}
\bar\ff=\bar\ff_+\oplus \mathsf C\oplus \bar\ff_-.
\end{align}
For any $\gamma\in \mathsf{C}^{*}$, we have a vertex algebra
$V_{\bar\ff}(\gamma)$.
In fact, $V_{\bar\ff}(\gamma)$ is a vertex operator algebra with conformal vector
$$\omega^{\ff}:=L(-2)\bm{1},$$
 where
\begin{align}
V_{\bar\ff}(\gamma)_{(n)}=0\quad\text{ for }n<0,\quad V_{\bar\ff}(\gamma)_{(0)}={\mathbb C}{\bf 1},\quad  V_{\bar\ff}(\gamma)_{(1)}=\ff.
\end{align}

Set $H=\C{\bf k}+\C{\bf d}$, an (independent) vector space with basis $\{ {\bf k}, {\bf d}\}$,
equipped with the symmetric bilinear form $\<\cdot,\cdot\>$ given by
\begin{align}
\<\bfk,\bfk\>=0=\<\bfd,\bfd\>,\   \   \   \<\bfk,\bfd\>=1.
\end{align}
Viewing $H$ as an abelian Lie algebra, we have a (general) affine algebra $\wh{H}$ and then
a simple vertex operator algebra $V_{\wh{H}}(1,0)$,
commonly known as the Heisenberg vertex operator algebra of rank $2$.
Set
\begin{align}
L=\Z{\bfk}\subset H,
\end{align}
 a free abelian group. Associated to the pair $(H,L)$,
we have a simple conformal vertex algebra (see \cite{LX}, \cite{BBS}, \cite{BDT})
\begin{align}
V_{(H,L)}=V_{\wh{H}}(1,0)\ot \C[L],
\end{align}
which contains $V_{\wh{H}}(1,0)$ as a vertex subalgebra, where $\C[L]=\oplus_{m\in \Z}\C e^{m\bfk}$ is the group algebra of $L$.
Let $\omega^H$ denote the standard conformal vector of $V_{\wh{H}}(1,0)$, which is also
the conformal vector of $V_{(H,L)}$. We have
\begin{eqnarray}
{\rm wt}\, \bfk =1={\rm wt}\, \bfd\  \  \text{and }\   {\rm wt}\,e^{m\bfk}=0\  \  \text{for }m\in \Z.
\end{eqnarray}

On the other hand, for every $\al\in \C$, we have an irreducible $V_{(H,L)}$-module
\begin{align}
V_{(H,L)}(\al):=V_{\wh{H}}(1,0)\ot e^{\al\bfk}\C[L],
\end{align}
 where for $m\in \Z$,
$$Y(e^{m\bfk},z)=E^{-}(-m\bfk,z)E^{+}(-m\bfk,z)e^{m\bfk}z^{m\bfk}.$$
Noticing that $z^{m\bfk}=1$ on $V_{(H,L)}(\al)$ as $\<\bfk,\bfk\>=0$, we have
\begin{align}\label{module-action-VHL}
Y(e^{m\bfk},z)e^{(\al+n)\bfk}=E^{-}(-m\bfk,z)e^{(\al+m+n)\bfk}\quad\text{for }m,n\in \Z.
\end{align}
The $V_{(H,L)}$-module $V_{(H,L)}(\al)$ is also $\Z$-graded by the conformal weights with
\begin{eqnarray}
V_{(H,L)}(\al)_{(n)}=0\  \ \text{ for }n<0\  \ \text{and }\ V_{(H,L)}(\al)_{(0)}=e^{\alpha\bfk}\C[L].
\end{eqnarray}
(Notice that ${\rm wt}\,e^{(\alpha+m)\bfk}=\frac{1}{2}\langle (\alpha+m)\bfk,(\alpha+m)\bfk\rangle=0$ for $m\in \Z$.)

Next, consider the tensor product $V_{\bar\ff}(\gamma)\ot V_{(H,L)}$, which  is also a conformal vertex algebra
with conformal vector
\begin{align}\label{conformal-vector-omega}
\omega=\omega^{\ff}\ot \bm{1}+\bm{1}\ot \omega^{H}.
\end{align}
Identify $\ff$ and $H$ as subspaces of $V_{\bar\ff}(\gamma)\ot V_{(H,L)}$ naturally.
The following result is due to Billig (see \cite{B1}):

\begin{thm}\label{wtTmod}
Let $\ell\in \C^{\times}$. Define  a linear functional  $\gamma_\ell$ on $\mathsf{C}$ by
\begin{align*}
\gamma_\ell(\rk)=\ell,\ \gamma_\ell(\rk_{I})=1-\mu\ell,\
\gamma_\ell(\rk_{VI})=\frac{1}{2},\ \gamma_\ell(\rk_{Vir})=12\mu\ell-2.
\end{align*}
Then $V_{\bar\ff}(\gamma_\ell)\ot V_{(H,L)}$ is a ${\mathcal T}(\fg,\mu)$-module
with $\rk_0=\ell$ and with
\begin{align}
\label{vafact1}
&\sum_{j\in \Z}(t_0^jt_1^n\rk_0)z^{-j-1}=\ell Y(e^{n\bfk},z)\quad \text{ for } n\in \Z^\times,\\
\label{vafact2}&\sum_{j\in \Z}(t_0^j\rk_1) z^{-j-1}= \ell Y(\bfk,z),\\
\label{vafact3}&\sum_{j\in \Z}(t_0^jt_1^m\ot a) z^{-j-1}= Y(a,z)Y(e^{m\bfk},z)\quad \text{ for } a\in \fg,\  m\in \Z,\\
\label{vafact4}&\sum_{j\in \Z}(t_0^jt_1^m\rd_1)z^{-j-1}= : Y(\bfd+mI,z)Y(e^{m\bfk},z):,\\
\label{vafact5}&\sum_{j\in \Z}(-t_0^jt_1^m\rd_0+\mu(j+1/2)t_0^jt_1^m\rk_0)z^{-j-2}= :Y(\omega,z)Y(e^{m\bfk},z):\\
&\  \quad   +Y(I,z)Y(m\bfk,z)Y(e^{m\bfk},z)
+(\ell\mu-1)\(\frac{d}{dz}Y(m\bfk,z)\)Y(e^{m\bfk},z), \notag
\end{align}
where $\omega$, defined in (\ref{conformal-vector-omega}), is the conformal vector of
$V_{\bar\ff}(\gamma_\ell)\ot V_{(H,L)}$.
\end{thm}

In the following, we are going to use this result of Billig to obtain an explicit realization of
certain irreducible $\wt{\mathfrak{t}}(\fg,\mu)^{o}$-modules.
For convenience, we formulate the following general result:

\begin{lemt}\label{vla-new}
Let $V$ be a vertex algebra and let $({\mathcal{L}},\mathcal A, \mathcal C,\rho)$ be a vertex Lie algebra.
Suppose that $\psi: \ \C[\partial]\mathcal A\oplus \mathcal C\rightarrow V$ is a linear map  such that
$$\psi(\mathcal C)\subset \C {\bf 1},\quad \psi(\partial^i a)=L(-1)^i\psi(a)\quad \text{  for }i\in \N,\ a\in \mathcal A$$
and such that $V$ is an ${\mathcal{L}}$-module
with $u(z)=Y(\psi(u),z)$ for $u\in \mathcal A+\mathcal C$.
Then  every $V$-module $W$ is an ${\mathcal{L}}$-module with
$$u(z)=Y(\psi(u),z)\quad \text{ for }u\in \mathcal A+\mathcal C.$$
Furthermore, $W$ is an $\mathcal{L}\rtimes \C\bm{d}$-module with $\bm{d}=-L(-1)$.
\end{lemt}

\begin{proof}  Let $a,b\in \mathcal A$. By definition, there exist $f_i(a,b)\in \C[\partial]\mathcal A+\mathcal C$
 for $i=0,1,\dots, r$ such that (\ref{az-bw}) holds.
As $V$ is an ${\mathcal{L}}$-module with $u(z)=Y(\psi(u),z)$ for $u\in \mathcal A+\mathcal C$,
the following commutator relation holds on $V$:
\begin{eqnarray}\label{Ypsi-ab}
\mbox{}\  \quad \quad  [Y(\psi(a),z),Y(\psi(b),w)]=\sum_{i=0}^{r} Y(\psi(f_i(a,b)),w)
\frac{1}{i!}\left(\frac{\partial}{\partial w}\right)^iz^{-1}\delta\left(\frac{w}{z}\right).
\end{eqnarray}
By \cite[Lemma 2.3.5]{Li-local} (cf. \cite{LL}, Proposition 5.6.7), this relation holds on every $V$-module $W$.
Thus, $W$ is an ${\mathcal{L}}$-module with $u(z)=Y(\psi(u),z)$ for $u\in \mathcal A+\mathcal C$.
\end{proof}

\begin{cort}
Let $M$ be any $V_{\bar{\ff}}(\gamma_\ell)\ot V_{(H,L)}$-module. Then $M$ is a ${\mathcal T}(\fg,\mu)$-module
with the action given as in Theorem \ref{wtTmod}. In particular, for any $V_{\bar{\ff}}(\gamma_\ell)$-module $W$,
any $\alpha\in {\mathbb C}$, $W\otimes V_{(H,L)}(\al)$ with the aforementioned action is a ${\mathcal T}(\fg,\mu)$-module.
\end{cort}

\begin{proof} It follows from the argument of Billig in the proof of Theorem \ref{wtTmod}.
We here give a slightly different proof by using Theorem \ref{wtTmod} instead of its proof.
 Let $B$ be a basis of $\fg$. Then the following elements form a basis of ${\mathcal T}(\fg,\mu)$:
 $$t_0^rt_1^m\otimes u,\quad K_{m,n}, \quad t_0^m\rk_1,\quad t_0^rt_1^m\rd_1,
 \quad (-t_0^rt_1^m\rd_0+\mu(r+1/2)t_0^rt_1^m\rk_0), \quad \rk_0$$
 for $u\in B,\ n\in \Z^{\times},\ m,r\in \Z$.
 Introduce a vector space
 \begin{align}
U_{\fg}:=\C[t_1,t_1^{-1}]\otimes \fg+\C \rk_1+\sum_{n\in \Z^{\times}}\C K^{(n)}+\sum_{m\in \Z}(\C D^{(0,m)}+\C D^{(1,m)})
\end{align}
with basis
 $$\{ \rk_1\}\cup \{K^{(n)}\ |\ n\in \Z^{\times}\}\cup \{ t_1^mu,\ D^{(0,m)},\ D^{(1,m)}\ |\ m\in \Z,\ u\in B\}.$$
Define a linear isomorphism
 $\rho:\  \C[t,t^{-1}]\otimes U_{\fg}\oplus \C\rk_0\rightarrow {\mathcal T}(\fg,\mu)$ by
 $$\rho(\rk_0)=\rk_0,\quad \rho(t^r\otimes K^{(n)})=K_{r,n},\quad \rho(t^m\otimes \rk_1)=t_0^m\rk_1,\quad
 \rho(t^r\otimes t_1^mu)=t_0^rt_1^m\otimes u,$$
 $$\rho(t^r\otimes D^{(1,m)})=t_0^rt_1^m\rd_1,\quad \rho(t^r\otimes D^{(0,m)})=-t_0^rt_1^m\rd_0+\mu(r+1/2)t_0^rt_1^m\rk_0$$
 for $n\in \Z^{\times},\ m,r\in \Z,\ u\in \fg$.
 From  \cite{B1}, ${\mathcal T}(\fg,\mu)$ is a vertex Lie algebra.
Define a linear map $\psi: U_{\fg}\oplus \C\rk_0\rightarrow V_{\bar{\ff}}(\gamma_\ell)\ot V_{(H,L)}$ by
$$\psi(\rk_0)=\ell {\bf 1},\quad  \psi(t_1^m\otimes u)=u\otimes e^{m\bfk},\quad \psi(\rk_1)= {\bf 1}\otimes \ell \bfk,\quad
\psi(K^{(n)})=1\otimes \frac{1}{n}e^{n\bfk},$$
$$\psi(D^{(0,m)})=L(-2)({\bf 1}\otimes e^{m\bfk})
+  I\otimes (m\bfk)_{-1}e^{m\bfk}+(\mu\ell-1)({\bf 1}\otimes (m\bfk)_{-2}e^{m\bfk}),  $$
$$\quad  \psi(D^{(1,m)})={\bf 1}\otimes \bfd_{-1}e^{m\bfk}+I\otimes me^{m\bfk}$$
for $u\in \fg,\ m\in \Z,\ n\in \Z^{\times}$.
Theorem \ref{wtTmod} states that $V_{\bar{\ff}}(\gamma_\ell)\ot V_{(H,L)}$ is a ${\mathcal T}(\fg,\mu)$-module
with $u(z)=Y(\psi(u),z)$ for $u\in U_\fg+\C\rk_0$.
Then it follows immediately from Lemma \ref{vla-new}.
\end{proof}

Next, we discuss $V_{\bar{\ff}}(\gamma_\ell)$-modules.
We first recall the following result (see \cite{B1}):

\begin{lemt}\label{eta-embedding}
The linear map
\begin{align}\label{embedingrho2}
\eta: \  \wh\fg\rtimes \mathsf{Vir}\rightarrow \bar{\ff}=\wh\fg\rtimes \mathsf{HVir},
\end{align}
 defined by
\begin{align}\label{embedingrho1}
&\eta|_{\wh{\fg}}=1,\quad
\eta(\rk_{Vir})= \rk_{Vir}+24 \rk_{VI}-12\rk_{I},\\
&\eta(L(m))=L(m)+(m+1)I(m)\quad \text{ for }m\in \Z,
\end{align}
 is a Lie algebra embedding.
\end{lemt}

For any pair $(\ell, c)$ of complex numbers, we have a vertex algebra
$V_{\wh\fg\rtimes {\mathsf{Vir}}}(\ell,c)$ associated to $\wh\fg\rtimes \mathsf{Vir}$,
where $\ell$ is the level of $\wh\fg$ and $c$ is the central charge
of $\mathsf{Vir}$.
The embedding \eqref{embedingrho2} naturally induces  an embedding of vertex algebras
\begin{align*}
\eta^{va}:\  V_{\wh\fg\rtimes \mathsf{Vir}}(\ell,24\mu\ell-2)\hookrightarrow V_{\bar{\ff}}(\gamma_\ell),
\end{align*}
where $\gamma_\ell$ is the linear functional on  $\mathsf{C}$ defined in Proposition \ref{wtTmod}.
Furthermore, we get a vertex algebra embedding
\begin{align}\label{vaembedding}
\eta^{va}\otimes 1:\  V_{\wh\fg\rtimes \mathsf{Vir}}(\ell,24\mu\ell-2)\ot V_{(H,L)}\rightarrow V_{\bar{\ff}}(\gamma_\ell)
\ot V_{(H,L)}.
\end{align}

As the first main result of this section, we have:

\begin{prpt}\label{prop:voav}
Let $\ell\in \C^{\times}$. Then the conformal vertex algebra
\[V:=V_{\wh\fg\rtimes \mathsf{Vir}}(\ell,24\mu\ell-2)\ot V_{(H,L)}\] is a
$\wt{\mathfrak{t}}(\fg,\mu)^{o}$-module of level $\ell$, where the actions of the fields $\rK_n(z),
\rk_1(z), (t_1^m\ot u)(z)$ for $n\in \Z^\times, m\in \Z, u\in \fg$ are given by \eqref{vafact1}---\eqref{vafact3}, and
in addition
\begin{align}\label{actrd1z}
\rd_1(z)=Y(\bfd,z),\quad t_0^{-1}\rd_0=-\omega^V_0\ (=-L(-1)),
\end{align}
\begin{align}
\label{actdnz} \rD_n(z)=
nY(\omega^V_{-1}e^{n\bfk},z)-
\frac{d}{dz} Y(\bfd_{-1}e^{n\bfk},z)
+n(\ell\mu-1)Y((n\bfk)_{-2}e^{n\bfk},z)
\end{align}
for $n\in \Z^\times$, where  $\omega^V$ denotes the conformal vector of $V$.
\end{prpt}

\begin{proof} With $\wt{\mathfrak{t}}(\fg,\mu)^{o}$ a subalgebra of ${\mathcal T}(\fg,\mu)$,
$V_{\bar{\ff}}(\gamma_\ell)\ot V_{(H,L)}$ is
 naturally  a $\wt{\mathfrak{t}}(\fg,\mu)^{o}$-module by Theorem \ref{wtTmod}.
View $V$ as a subspace of $V_{\bar{\ff}}(\gamma_\ell)\ot V_{(H,L)}$ via the embedding $\eta^{va}\ot 1$.
Then it suffices to prove that $V$ is a $\wt{\mathfrak{t}}(\fg,\mu)^{o}$-submodule of
$V_{\bar{\ff}}(\gamma_\ell)\ot V_{(H,L)}$.

Note that under the embedding \eqref{vaembedding} we have (see Lemma \ref{eta-embedding})
$$\eta_V(\omega^V)=\omega-L(-1)I.$$
Thus
\begin{align}\label{omegare}
Y(\omega^V,z)=Y(\omega,z)-\frac{d}{dz}Y(I,z).
\end{align}
In particular, we have $\omega^V_0=\omega_0$, namely $L^V(-1)=L(-1)$.
It can be readily seen that
 the fields $\rK_n(z),
\rk_1(z), (t_1^m\ot u)(z)$, $\rd_1(z)$, and the operator $t_0^{-1}\rd_0$ preserve $V$.
Thus it remains to show that  the fields $\rD_n(z)$, $n\in \Z^\times$ also preserve $V$.
Recall that
\begin{align*}
\rD_n(z)=\sum_{j\in \Z}\left(
(j+1)t_0^jt_1^n \rd_1-n t_0^jt_1^n \rd_0+\mu n^2 (j+1/2) \rk_{j,n}\right)z^{-j-2}.
\end{align*}
With \eqref{vafact4} and \eqref{vafact5}, $\rD_n(z)$ acts on $V_{\bar{\ff}}(\gamma_\ell)\ot V_{(H,L)}$ as 
\begin{align*}
&-\frac{d}{dz}(:(Y(\bfd,z)+Y(nI,z))Y(e^{n\bfk},z):)+n:Y(\omega,z)Y(e^{n\bfk},z):\\
&+nY(I,z)Y(n\bfk,z)Y(e^{n\bfk},z):
+n(\ell\mu-1)\left(\frac{d}{dz}Y(n\bfk,z)\right)Y(e^{n\bfk},z)\\
=\ &n:(Y(\omega,z)-\frac{d}{dz}Y(I,z))Y(e^{n\bfk},z):-
\frac{d}{dz}(:Y(\bfd,z)Y(e^{n\bfk},z):)\\
&+n(\ell\mu-1)\left(\frac{d}{dz}Y(n\bfk,z)\right)Y(e^{n\bfk},z)\\
=\ &n:Y(\omega^V,z)Y(e^{n\bfk},z):-
\frac{d}{dz}\left(:Y(\bfd,z)Y(e^{n\bfk},z):\right)\\
&+n(\ell\mu-1)\left(\frac{d}{dz}Y(n\bfk,z)\right)Y(e^{n\bfk},z)\\
=\ &n Y(\omega^V_{-1}e^{n\bfk},z)-
\frac{d}{dz}Y(\bfd_{-1}e^{n\bfk},z)+n(\ell\mu-1)Y((n\bfk)_{-2}e^{n\bfk},z),
\end{align*}
where we are using the fact that $\frac{d}{dz}Y(e^{n\bfk},z)=Y(n\bfk,z)Y(e^{n\bfk},z)$ for vertex algebra $V_{(H,L)}$.
With  \eqref{omegare}, this implies that the fields $\rD_n(z)$ preserve  $V$ (under the action  \eqref{actdnz}).
This completes the proof.
\end{proof}

Equip $\wh\fg\rtimes \mathsf{Vir}$ with the $\Z$-grading defined by
$$\deg \rk=0=\deg \rk_{\rm Vir},\   \deg (a\otimes t^n)=-n,\  \deg L(n)=-n\   \text{ for }a\in \fg,\ n\in \Z.$$
Notice that the Lie algebra embedding $\eta$ in Lemma \ref{eta-embedding} preserves the $\Z$-gradings.
The degree-zero subalgebra $(\wh\fg\rtimes \mathsf{Vir})_{(0)}$
 equals $\fg\oplus (\C L(0)+\C\rk+\C\rk_{\rm Vir})$, a direct product of $\fg$ with
 the abelian Lie algebra $\C L(0)+\C\rk+\C\rk_{Vir}$.
Set
$$(\wh\fg\rtimes \mathsf{Vir})_{(-)}=\sum_{n<0}(\wh\fg\rtimes \mathsf{Vir})_{(n)}
={\rm Span}\{ L(n),\ a\otimes t^n\mid n>0,\ a\in \fg\}.$$
 Let $\lambda\in P_+,\  \beta,\ell,c\in \C$.
Make the $\fg$-module $L_{\fg}(\lambda)$ a $(\wh\fg\rtimes \mathsf{Vir})_{(-)}+(\wh\fg\rtimes \mathsf{Vir})_{(0)}$-module
by letting $L(0), \rk$ and $\rk_{Vir}$ acting as scalars $\beta$, $\ell$ and $c$, respectively,
and letting $(\wh\fg\rtimes \mathsf{Vir})_{(-)}$ acting trivially.
Then form an induced  $\wh\fg\rtimes \mathsf{Vir}$-module
\begin{align}
V_{\wh\fg\rtimes \mathsf{Vir}}(\ell,c,\lambda,\beta):=\U(\wh\fg\rtimes \mathsf{Vir})
\ot_{\U((\wh\fg\rtimes \mathsf{Vir})_{(-)}+(\wh\fg\rtimes \mathsf{Vir})_{(0)})}L_{\fg}(\lambda).
\end{align}
Notice that $V_{\wh\fg\rtimes \mathsf{Vir}}(\ell,c,0,0)=V_{\wh\fg\rtimes \mathsf{Vir}}(\ell,c)$ (the associated vertex algebra).
Denote by $L_{\wh\fg\rtimes \mathsf{Vir}}(\ell,c,\lambda,\beta)$ the unique
irreducible quotient of $V_{\wh\fg\rtimes \mathsf{Vir}}(\ell,c,\lambda,\beta)$.

Recall that $V=V_{\wh\fg\rtimes \mathsf{Vir}}(\ell,24\mu\ell-2)\ot V_{(H,L)}$ is a conformal vertex algebra, where
\begin{eqnarray}
\mbox{}\  \  {\rm wt}\,\fg=1,\  \  {\rm wt}\,\omega=2,\  \
{\rm wt}\, \bfk =1={\rm wt}\, \bfd,\  \   {\rm wt}\,e^{m\bfk}=0\  \  (m\in \Z).
\end{eqnarray}

As the main result of this section, we have (cf. \cite[Theorem 5.5]{B1}):

\begin{thm}\label{prop:irrlomod}
Let $\lambda\in P_+$, $\al,\beta,\ell\in \C$ with $\ell\ne 0$. Then
\[L_{\wh\fg\rtimes \mathsf{Vir}}(\ell,24\mu\ell-2,\lambda,\beta)\ot V_{(H,L)}(\al)\]
 is a $\wt{\mathfrak{t}}(\fg,\mu)^{o}$-module with the action given by \eqref{vafact1},
 \eqref{vafact2}, \eqref{vafact3}, \eqref{actrd1z}, and \eqref{actdnz}.
Furthermore, it is an irreducible module for $\wt{\mathfrak{t}}(\fg,\mu)^{o}$ and $\wh{\mathfrak{t}}(\fg,\mu)^{o}$,
which is isomorphic to
$L_{\wh{\mathfrak{t}}(\fg,\mu)^{o}}(T_{\ell,\lambda,\al,\beta})$ as a $\wh{\mathfrak{t}}(\fg,\mu)^{o}$-module,
 and isomorphic to $L_{\wt{\mathfrak{t}}(\fg,\mu)^{o}}(T_{\ell,\lambda,\al,\beta})$ as a $\wt{\mathfrak{t}}(\fg,\mu)^{o}$-module.
\end{thm}

\begin{proof} For convenience, set
$$W=L_{\wh\fg\rtimes \mathsf{Vir}}(\ell,24\mu\ell-2,\lambda,\beta)\ot V_{(H,L)}(\al).$$
Recall that $\wt{\mathfrak{t}}(\fg,\mu)^{o}=\wh{\mathfrak{t}}(\fg,\mu)^{o}\rtimes\C t_0^{-1}\rd_0$ and
$\wh{\mathfrak{t}}(\fg,\mu)^{o}$ is a vertex Lie algebra with
$\mathcal{A}=\mathcal{A}_{\fg}$ and $\mathcal{C}=\C\rk_0$, where
$$\mathcal{A}_{\fg}=\C[t_1,t_1^{-1}]\otimes \fg+\C \rk_1+\C\rd_1+\sum_{n\in \Z^{\times}}(\C \rK_n+\C \rD_n).$$
Define a linear map $\psi: \mathcal{A}_{\fg}\oplus \C\rk_0\rightarrow V$ (with $V$ defined in Proposition \ref{prop:voav}) by
\begin{align}\label{defpsi}
\psi(\rk_0)&=\ell {\bf 1},\ \  \psi(\rd_1)=\bfd, \  \ \psi(\rk_1)={\bf 1}\otimes \ell\bfk,\\
\notag \ \  \psi(\rK_n)&={\bf 1}\otimes \frac{1}{n}e^{n\bfk},\  \  \psi(t_1^m\otimes a)=a\otimes e^{m\bfk},\\
\notag \psi(\rD_n)&=nL(-2)({\bf 1}\otimes e^{n\bfk})-L(-1)({\bf 1}\otimes \bfd_{-1}e^{n\bfk})\\
\notag&\qquad+n(\ell\mu-1)({\bf 1}\otimes (n\bfk)_{-2}e^{n\bfk})\end{align}
for $a\in \fg,\ m\in \Z,\ n\in \Z^{\times}$.
 Proposition \ref{prop:voav} states that $V$ is a $\wh{\mathfrak{t}}(\fg,\mu)^{o}$-module
 with  $u(z)=Y(\psi(u),z)$ for $u\in \mathcal{A}_{\fg}+\C\rk_0$.
In view of Lemma \ref{vla-new},  $W$ is a $\wt{\mathfrak{t}}(\fg,\mu)^{o}$-module (with $t_0^{-1}\rd_0=-L(-1)$).
Note that the vertex algebra $V$ is generated by $\fg, \bfk, \bfd, e^{\bfk},$ and $\omega$.
It follows that  every $V$-submodule of $W$ is a $\wh{\mathfrak{t}}(\fg,\mu)^{o}$-submodule of $W$.
Thus $W$ is an irreducible $\wh{\mathfrak{t}}(\fg,\mu)^{o}$-module and
an  irreducible $\wt{\mathfrak{t}}(\fg,\mu)^{o}$-module.

Recall that $\wt{\mathfrak{t}}(\fg,\mu)^{o}$ and $\wh{\mathfrak{t}}(\fg,\mu)^{o}$ are $\Z$-graded Lie algebras with
${\mathcal{L}}_{(0)}=\wh{\mathfrak{t}}(\fg,\mu)^{o}_{(0)}=\wt{\mathfrak{t}}(\fg,\mu)^{o}_{(0)}$
 and recall ${\mathcal{L}}_{(0)}$-module $T_{\ell,\lambda,\al,\beta}=\C[q,q^{-1}]\otimes L_{\fg}(\lambda)$ as a
 vector space. On the other hand,  for the $V$-module $W$, we have
$$L_{\fg}(\lambda)\subset L_{\wh\fg\rtimes \mathsf{Vir}}(\ell,24\mu\ell-2,\lambda,\beta),\quad
e^{\alpha\bfk}\C[L]\subset V_{(H,L)}(\al).$$
Recall that  the $V_{(H,L)}$-module $V_{(H,L)}(\al)$ is $\Z$-graded by the conformal weights with
\begin{eqnarray*}
V_{(H,L)}(\al)_{(n)}=0\  \ \text{ for }n<0\  \ \text{and }\ V_{(H,L)}(\al)_{(0)}=e^{\alpha\bfk}\C[L].
\end{eqnarray*}
We see that $W=\oplus_{n\in \N}W_{(n+\beta)}$ is $\C$-graded by conformal weights with
$$W_{(\beta)}=L_{\fg}(\lambda)\otimes e^{\alpha\bfk}\C[L].$$
Note that
$${\rm wt}\, \psi(\rd_1)=1={\rm wt}\,\psi(\rk_1),\   \   {\rm wt}\,\psi(t_1^m\otimes a)=1,\  \   {\rm wt}\,\psi(\rK_n)=0,\  \  \psi(\rD_n)=2 $$
for $u\in \fg,\ m\in \Z,\ n\in \Z^{\times}$ and that
$${\rm wt}\, v_{m}={\rm wt}\,v-m-1\  \   \text{ for any homogeneous vector }v\in V,\ m\in \Z.$$
 It then follows from the action of $\wt{\mathfrak{t}}(\fg,\mu)^{o}$ on $W$ that
 the $\Z$-grading of $\wt{\mathfrak{t}}(\fg,\mu)^{o}$ agrees with the conformal grading. Thus
$W$ with the conformal grading is a $\C$-graded $\wt{\mathfrak{t}}(\fg,\mu)^{o}$-module.
Consequently, we have
$$\wt{\mathfrak{t}}(\fg,\mu)^{o}_{(-)}\left(L_{\fg}(\lambda)\otimes e^{\alpha\bfk}\C[L]\right)=0$$
and $L_{\fg}(\lambda)\otimes e^{\alpha\bfk}\C[L]$ is an irreducible ${\mathcal{L}}_{(0)}$-module.

We now show that $L_{\fg}(\lambda)\otimes e^{\alpha\bfk}\C[L] \simeq T_{\ell,\lambda,\al,\beta}$
as an ${\mathcal{L}}_{(0)}$-module. Recall \eqref{vafact1}---\eqref{vafact3}.
Let $u\in L_{\fg}(\lambda),\ r\in \Z$. We see that $\rk_1$ acts as $\ell \bfk_{0}=0$ on $V_{(H,L)}(\alpha)$,
$\rk_0$ acts as scalar $\ell$ on $W$,
$\rd_{0,0}=\rd_1$ acts as $\bfd_0$ and
$$\bfd_0(u\ot e^{(\alpha+r)\bfk})=u\ot \langle \bfd,(\alpha+r)\bfk\rangle e^{(\alpha+r)\bfk}=(\alpha+r)(u\ot e^{(\alpha+r)\bfk}).$$
By (\ref{module-action-VHL}) we have
$$Y(e^{m\bfk},z)(u\ot e^{(\alpha+r)\bfk})=u\ot E^{-}(-m\bfk,z)e^{(\alpha+m+r)\bfk},$$
$$Y(a,z)Y(e^{m\bfk},z)(u\ot e^{(\alpha+r)\bfk})=Y(a,z)u\ot E^{-}(-m\bfk,z)e^{(\alpha+m+r)\bfk}$$
for $m\in \Z,\  a\in \fg$.
Then
$$(t_1^n\rk_0)(u\ot e^{(\alpha+r)\bfk})=\ell (e^{n\bfk})_{-1}(u\ot e^{(\alpha+r)\bfk})=\ell (u\ot e^{(\alpha+n+r)\bfk}),$$
$$(t_1^m\otimes a)(u\ot e^{(\alpha+r)\bfk})=\psi(t_1^m\otimes a)_{0}(u\ot e^{(\alpha+r)\bfk})=au\ot e^{(\alpha+m+r)\bfk} $$
for $n\in \Z^{\times}$.

 Note that $\rd_{0,n}={\rm Res}_zz\rD_n(z)$ (as $\rD_n(z)=\sum_{m\in \Z}\rd_{m,n}z^{-m-2}$).
 Recall the action of $\rD_n(z)$ from \eqref{actdnz}. We have
\begin{align*}
&{\rm Res}_zz\left(\frac{d}{dz}Y(n\bfk,z)\right)Y(e^{n\bfk},z)(u\ot e^{(\alpha+r)\bfk})\\
=\  &{\rm Res}_zu\ot z\left(\frac{d}{dz}Y(n\bfk,z)\right)E^{-}(-n\bfk,z)e^{(\alpha+n+r)\bfk}\\
=\  &{\rm Res}_zu\ot z E^{-}(-n\bfk,z)\left(\frac{d}{dz}Y(n\bfk,z)\right)e^{(\alpha+n+r)\bfk}\\
=\  &0,
\end{align*}
noticing that $(n\bfk)_{j}e^{(\alpha+n+r)\bfk}=0$ for $j\ge 0$. We also have
\begin{align*}
&:Y(\bfd,z)Y(e^{n\bfk},z):(u\ot e^{(\alpha+r)\bfk})\\
=\  &\sum_{j\ge 1}u\ot \bfd_{-j}z^{j-1}E^{-}(-n\bfk,z)e^{(\alpha+n+r)\bfk}
+\sum_{j\ge 0}u\ot Y(e^{n\bfk},z)\bfd_jz^{-j-1}e^{(\alpha+r)\bfk}\\
=\ &\sum_{j\ge 1}u\ot \bfd_{-j}z^{j-1}E^{-}(-n\bfk,z)e^{(\alpha+n+r)\bfk}
+u\ot (\alpha+r)z^{-1}E^{-}(-n\bfk,z)e^{(\alpha+n+r)\bfk},
\end{align*}
from which we get
\begin{align*}
&-{\rm Res}_zz \frac{d}{dz}:Y(\bfd,z)Y(e^{n\bfk},z):(u\ot e^{(\alpha+r)\bfk})\\
=\ &{\rm Res}_z:Y(\bfd,z)Y(e^{n\bfk},z):(u\ot e^{(\alpha+r)\bfk})\\
=\ &(\alpha+r)(u\ot e^{(\alpha+n+r)\bfk}).
\end{align*}
On the other hand, we have
\begin{align*}
&{\rm Res}_zz  :Y(\omega,z)Y(e^{n\bfk},z):(u\ot e^{(\alpha+r)\bfk})\\
=\ &{\rm Res}_zz\sum_{j\ge 1}z^{j-1}L(-j-1)Y(e^{n\bfk},z)(u\ot e^{(\alpha+r)\bfk})\\
&\  \ +{\rm Res}_zz\sum_{j\ge 0}Y(e^{n\bfk},z)z^{-j-2}L(j)(u\ot e^{(\alpha+r)\bfk})\\
=\  &{\rm Res}_zz\sum_{j\ge 1}z^{j-1}L(-j-1)\left(u\ot E^{-}(-n\bfk,z)e^{(\alpha+n+r)\bfk}\right)\\
&\  \   +{\rm Res}_zz\sum_{j\ge 0}z^{-j-2}L(j)u\ot E^{-}(-n\bfk,z)e^{(\alpha+n+r)\bfk}\\
&\  \   +{\rm Res}_zz\sum_{j\ge 0}z^{-j-2}Y(e^{n\bfk},z)(u\ot L(j)e^{(\alpha+r)\bfk})\\
=\ &L(0)u\ot e^{(\alpha+n+r)\bfk}\\
=\ &\beta (u\ot  e^{(\alpha+n+r)\bfk}),
\end{align*}
noticing that $L(j)e^{(\alpha+r)\bfk}=0$ for $j\ge 0$.
Now, comparing these relations with (\ref{e5.8}) and (\ref{e5.9}),  we conclude that
$\C[q,q^{-1}]\otimes L_{\fg}(\lambda)\simeq L_{\fg}(\lambda)\otimes e^{\alpha\bfk}\C[L]$
as an ${\mathcal{L}}_{(0)}$-module.
Then it follows that  $W\cong L_{\wt{\mathfrak{t}}(\fg,\mu)^{o}}(T_{\ell,\lambda,\al,\beta})$.
\end{proof}

As a consequence of Proposition \ref{prop:voav} we have:

\begin{cort}\label{cor:vahom}
Let $\ell\in \C^\times$. Then there exists a vertex algebra epimorphism $$\Theta:\
 V_{\wh{\mathfrak{t}}(\fg,\mu)^o}(\ell)\rightarrow
 V_{\wh\fg\rtimes \mathsf{Vir}}(\ell,24\mu\ell-2)\ot V_{(H,L)},$$
which is uniquely determined by
\begin{align}\label{theta-1}
&\Theta(t_1^m\ot u)=u\ot e^{m\bfk},\quad \Theta(\rk_1)=\ell\bfk,\quad \Theta(\rd_1)=\bfd,\quad \Theta(\rK_n)
=\frac{\ell}{n}e^{n\bfk},\\
 &\Theta(\rD_n)=
nL(-2)e^{n\bfk}-L(-1)(\bfd_{-1}e^{n\bfk})+n(\mu\ell-1)(n\bfk_{-2}e^{n\bfk})\label{theta-2}
\end{align}
for $u\in \fg,\ m\in \Z,\  n\in \Z^\times$.
\end{cort}

\begin{proof}As in the proof of Theorem \ref{prop:irrlomod}, from Proposition \ref{prop:voav} ,
$V_{\wh\fg\rtimes \mathsf{Vir}}(\ell,24\mu\ell-2)\ot V_{(H,L)}$
is a $\wh{\mathfrak{t}}(\fg,\mu)^o$-module
 with $u(z)=Y(\psi(u),z)$ for $u\in \mathcal{A}_{\fg}+\C\rk_0$ (see \eqref{defpsi}).
Note that
$$\rk_0 ({\bf 1}\ot {\bf 1})=\ell ({\bf 1}\ot {\bf 1})\  \text{ and }\  \wh{\mathfrak{t}}(\fg,\mu)^o_{(-)}({\bf 1}\ot {\bf 1})=0.$$
It follows that there is a $\wh{\mathfrak{t}}(\fg,\mu)^o$-module epimorphism
$$\Theta: \ V_{\wh{\mathfrak{t}}(\fg,\mu)^o}(\ell) \rightarrow V_{\wh\fg\rtimes \mathsf{Vir}}(\ell,24\mu\ell-2)\ot V_{(H,L)}$$
such that $\pi({\bf 1})={\bf 1}\ot {\bf 1}$. Furthermore, (\ref{theta-1}) and (\ref{theta-2}) hold.
Since  $\mathcal{A}_{\fg}$  generates $V_{\wh{\mathfrak{t}}(\fg,\mu)^o}(\ell)$
as a vertex algebra,  $\Theta$ is a vertex algebra homomorphism.
\end{proof}

On the other hand, we have:

\begin{cort}\label{cor-Psi}
Let $\ell\in \C^\times$. Then there exists a $\wt{\mathfrak{t}}(\fg,\mu)^o$-module epimorphism
\begin{align}
\Psi: \ V_{\wh{\mathfrak{t}}(\fg,\mu)^o}(\ell) \rightarrow L_{\wh{\mathfrak{t}}(\fg,\mu)^o}(T_\ell),
\end{align}
which is uniquely determined by $\Psi({\bf 1})=1\ot 1$, and
 $\ker \Psi$ is an ideal of the vertex algebra $V_{\wh{\mathfrak{t}}(\fg,\mu)^o}(\ell)$.
On the other hand, there exists a $\wt{\mathfrak{t}}(\fg,\mu)^o$-module isomorphism
\begin{align}
\pi:\  L_{\wh{\mathfrak{t}}(\fg,\mu)^o}(T_\ell)\rightarrow  L_{\wh\fg\rtimes \mathsf{Vir}}(\ell,24\mu\ell-2)\ot V_{(H,L)},
\end{align}
which is uniquely determined by $\pi(1\otimes 1)={\bf 1}\ot {\bf 1}$. Furthermore, $\pi$ is a vertex algebra isomorphism
with $L_{\wh{\mathfrak{t}}(\fg,\mu)^o}(T_\ell)$ equipped with the quotient vertex algebra structure such that
$\Theta=\pi\circ \Psi$.
\end{cort}

\begin{proof} From the second part of Theorem \ref{prop:irrlomod},
there exists a $\wt{\mathfrak{t}}(\fg,\mu)^o$-module epimorphism
$\Psi: \ V_{\wh{\mathfrak{t}}(\fg,\mu)^o}(\ell) \rightarrow L_{\wh{\mathfrak{t}}(\fg,\mu)^o}(T_\ell)$
such that $\Psi({\bf 1})=1\ot 1$.
As a $\wt{\mathfrak{t}}(\fg,\mu)^o$-submodule of $V_{\wh{\mathfrak{t}}(\fg,\mu)^o}(\ell)$,
 $\ker \Psi$ is an ideal of the vertex algebra $V_{\wh{\mathfrak{t}}(\fg,\mu)^o}(\ell)$ from Remark \ref{idealofv}.
Then $L_{\wh{\mathfrak{t}}(\fg,\mu)^o}(T_\ell)$ is naturally a vertex algebra
with $1\ot 1$ as the vacuum vector, while $\Psi$ becomes a vertex algebra epimorphism.
Furthermore, from Corollary \ref{cor:vahom},
 $L_{\wh{\mathfrak{t}}(\fg,\mu)^o}(T_\ell)$  is a simple vertex algebra isomorphic to
$L_{\wh\fg\rtimes \mathsf{Vir}}(\ell,24\mu\ell-2)\ot V_{(H,L)}$.
\end{proof}

\begin{remt}\label{re:decaffvir}
{\em Assume $\ell\ne -h^{\vee}$, where $h^{\vee}$ denotes the dual Coxeter number of $\fg$.
For the affine vertex operator algebra $V_{\wh{\fg}}(\ell,0)$, denote the conformal vector by $\omega^{\fg}$
and  the central charge by $c_{\ell}$.
Note that $V_{\wh\fg\rtimes \mathsf{Vir}}(\ell,c)$ contains $V_{\wh{\fg}}(\ell,0)$ naturally as a vertex subalgebra.
Set $$\omega'=\omega-\omega^{\fg}\in V_{\wh\fg\rtimes \mathsf{Vir}}(\ell,c).$$
From \cite{FZ}, the vertex subalgebra $\<\omega'\>$ generated by $\omega'$
is a vertex operator algebra with $\omega'$ as its conformal vector of central charge $c-c_{\ell}$.
 It follows from the P-B-W theorem that
\begin{align}
V_{\wh\fg\rtimes \mathsf{Vir}}(\ell,c)=V_{\wh\fg}(\ell,0)\ot \<\omega'\>
\cong V_{\wh\fg}(\ell,0)\ot V_{\mathsf{Vir}}(c-c_\ell,0),
\end{align}
an isomorphism of vertex operator algebras.
Furthermore, from \cite{FHL} we have
\begin{align}
L_{\wh\fg\rtimes \mathsf{Vir}}(\ell,c)
\cong L_{\wh\fg}(\ell,0)\ot L_{\mathsf{Vir}}(c-c_\ell,0).
\end{align}
On the other hand, we have
\begin{align}\label{need-extra}
 L_{\wh\fg\rtimes \mathsf{Vir}}(\ell,c,\lambda,\beta)
= L_{\wh\fg}(\ell,\lambda)\ot L_{\mathsf{Vir}}(c-c_\ell,\beta),
\end{align}
where $L_{\wh\fg}(\ell,\lambda)$ is the irreducible highest weight $\wh\fg$-module
 of level $\ell$ and $L_{\mathsf{Vir}}(c,\beta)$  is the irreducible highest weight $\mathsf{Vir}$-module
 of central charge $c$.}
 \end{remt}

Recall that $\theta$ is the highest long root of $\fg$ with $h_\theta$ denoting the coroot. We have:

\begin{lemt}\label{lem:ltlint}
Let $\lambda\in P_+$,  $\ell, \al,\beta\in \C$ with $\ell\ne 0$.
Then  the $\wh{\mathfrak t}(\fg,\mu)^o$-module $L_{\wh{\mathfrak t}(\fg,\mu)^o}(T_{\ell,\lambda,\al,\beta})$
is integrable if and only if $\ell\in \Z_{+}$ and $\lambda(h_\theta)\le \ell$.
Furthermore, if $\ell\in \Z_{+}$ and $\lambda(h_\theta)\le \ell$,
then
\begin{align*}
(t_1^ma)^{\epsilon_\gamma\ell+1}\cdot (1\ot 1)=0\  \mbox{ in }L_{\wh{\mathfrak t}(\fg,\mu)^o}(T_{\ell,\lambda,\al,\beta})
\end{align*}
for $m\in \Z,\  a\in \fg_\gamma,\   \gamma\in \Delta$.
\end{lemt}

\begin{proof} From \cite{Kac}, the $\wh\fg$-module $L_{\wh\fg}(\ell,\lambda)$ is integrable
if and only if $\ell\in \Z_{+}$ and $\lambda(\theta^{\vee})\le \ell$.
Then the first assertion follows from  Proposition \ref{prop:irrlomod} and (\ref{need-extra}).
The second assertion follows immediately from  Lemma \ref{lem:charinto}.
\end{proof}

Using Lemma \ref{lem:ltlint}, we immediately have:

\begin{prpt}
Let $\ell$ be a positive integer.
Then the vertex algebra epimorphism $\Psi:\ V_{\wh{\mathfrak{t}}(\fg,\mu)^o}(\ell)\rightarrow
L_{\wh{\mathfrak{t}}(\fg,\mu)^o}(T_\ell)$ factors through $V^{\rm int}_{\wh{\mathfrak{t}}(\fg,\mu)^o}(\ell)$:
 \begin{align}
 V_{\wh{\mathfrak{t}}(\fg,\mu)^o}(\ell)\twoheadrightarrow
 V^{\rm int}_{\wh{\mathfrak{t}}(\fg,\mu)^o}(\ell)\twoheadrightarrow
 L_{\wh{\mathfrak{t}}(\fg,\mu)^o}(T_\ell).
 \end{align}
In other words, $\Psi$ reduces to a vertex algebra epimorphism
 \begin{align}
 \Psi^{\rm int}:\  V^{\rm int}_{\wh{\mathfrak{t}}(\fg,\mu)^o}(\ell)\rightarrow
 L_{\wh{\mathfrak{t}}(\fg,\mu)^o}(T_\ell).
 \end{align}
\end{prpt}

\subsection{Classification of bounded $\N$-graded modules for $V_{\wh{\ft}(\fg,\mu)^{o}}(\ell)$ and
$V^{\rm int}_{\wh{\ft}(\fg,\mu)^{o}}(\ell)$}
Throughout this subsection, we assume that $\ell$ is a nonzero complex number.

Recall that $\rd_1\in \wh{\ft}(\fg,\mu)^o\subset  \wt{\ft}(\fg,\mu)^o$ and
 $\rd_0\notin \wt{\ft}(\fg,\mu)^o$.
 Set
\begin{align}
\wh{\ft}^{\star}(\fg,\mu)^o=\wh{\ft}(\fg,\mu)^{o}\rtimes \C \rd_0,
\end{align}
a subalgebra of $\mathcal{T}(\fg,\mu)$.
Set
\begin{align}
{\bf D}=\C \rd_0+\C\rd_1\subset \wh{\ft}^{\star}(\fg,\mu)^o.
\end{align}
Following Billig (see \cite{B1,B2}) we formulate the following notion:

\begin{dfnt}\label{defgradedphimod}
{\em A  $\wh{\ft}^{\star}(\fg,\mu)^o$-module $W$ is called a {\em ${\bf D}$-weight module} if
 ${\bf D}$ acts semisimply, and
 a ${\bf D}$-weight $\wh{\ft}^{\star}(\fg,\mu)^o$-module $W$ is called a {\em bounded ${\bf D}$-weight module} if
 the real parts of the $\rd_0$-eigenvalues are bounded from above and
if all ${\bf D}$-weight subspaces of $W$ are finite-dimensional.}
\end{dfnt}

\begin{remt}\label{star-algebra}
{\em Note that each ${\bf D}$-weight $\wh{\ft}^{\star}(\fg,\mu)^o$-module with the grading
given by the eigenvalues of $-\rd_0$ is naturally a $\C$-graded $\wh{\ft}(\fg,\mu)^{o}$-module.
On the other hand, let $W=\oplus_{\nu\in \C}W_{(\nu)}$ be a $\C$-graded $\wh{\ft}(\fg,\mu)^{o}$-module.
 Then $W$ becomes a $\wh{\ft}^{\star}(\fg,\mu)^o$-module with the action of $\rd_0$ given by
 $\rd_0|_{W_{(\nu)}}=-\nu$ for $\nu\in \C$.
 Assume that $W$ is an irreducible $\Z$-graded $\wh{\ft}(\fg,\mu)^{o}$-module.
As we work on $\C$, it is straightforward to show that  $\rd_1$ is semisimple on $W$, so that
$W$ becomes a ${\bf D}$-weight $\wh{\ft}^{\star}(\fg,\mu)^o$-module.
We then define a {\em bounded $\N$-graded $\wh{\ft}(\fg,\mu)^{o}$-module}
to be an $\N$-graded $\wh{\ft}(\fg,\mu)^{o}$-module which viewed as a
$\wh{\ft}^{\star}(\fg,\mu)^o$-module is a bounded ${\bf D}$-weight module.}
\end{remt}

The following is the main result of this subsection:

\begin{thm}\label{thm:irrVmod}
Let  $\ell\in \C^{\times}$. Then for any  $\lambda\in P_+$,
$\al,\beta\in \C$, $L_{\wh{\mathfrak{t}}(\fg,\mu)^o}(T_{\ell,\lambda,\al,\beta})$ is an irreducible bounded
$\wh{\ft}^{\star}(\fg,\mu)^{o}$-module of level $\ell$.  Furthermore, every irreducible bounded
$\wh{\ft}^{\star}(\fg,\mu)^{o}$-module of level $\ell$ is of this form.
\end{thm}

Before we present the proof of this theorem, we give a consequence.
Combining Theorem \ref{prop:irrlomod} and Lemma \ref{lem:ltlint}
with Theorem \ref{thm:irrVmod} we immediately have:

\begin{cort}\label{thm:irrLmod}
Let $\ell$ be a positive integer.
Then all irreducible bounded $\N$-graded  $V^{\rm int}_{\wh{\ft}(\fg,\mu)^{o}}(\ell)$-modules
up to equivalence are exactly the irreducible $\wh{\ft}(\fg,\mu)^{o}$-modules
$L_{\wh{\mathfrak{t}}(\fg,\mu)^o}(T_{\ell,\lambda,\al,\beta})$
for $\al,\beta\in \C,\  \lambda\in P_+$ with $\lambda(h_\theta)\le \ell$.
\end{cort}

Now we proceed to prove Theorem \ref{thm:irrVmod}.
The first part follows immediately from the explicit realization in Proposition \ref{prop:irrlomod}. Now,
assume that $W$  is an irreducible bounded $\wh{\mathfrak{t}}^{\star}(\fg,\mu)^o$-module
of level $\ell$. Then there is an eigenvalue $\beta$ of $-\rd_0$ on $W$ such that
all the eigenvalues of  $-\rd_0$ on $W$ lie in $\beta+\N$. For $n\in \Z$, set
$$W_{(n)}=\{ u\in W\, |\, \rd_0(u)=-(n+\beta)u\}.$$
Then $W=\oplus_{m\in \N}W_{(m)}$ is
an irreducible bounded $\N$-graded $\wh{\mathfrak{t}}(\fg,\mu)^o$-module
where $W_{(0)}\ne 0$ and it is an irreducible $\mathcal L_{(0)}$-module.
 For $m\in \Z,\  \beta\in \C$, set
\begin{align*}
W_{(m,\beta)}=\{w\in W_{(m)}\mid \rd_1 w=\beta w\}.
\end{align*}
Let $\alpha\in \C$ such that $W_{(0,\alpha)}\ne 0$.
Since $W$ irreducible,  the eigenvalues of $\rd_1$ on $W$ belong to the single $\Z$-coset $\al+\Z$ of $\C$.

First, as an analogue of a result of Billig  (see \cite[Lemma 2.1]{B1}) we have:

\begin{lemt}\label{lem:k1=0}
The central element $\rk_1$ acts trivially on $W$.
\end{lemt}

\begin{proof} As $W$ is irreducible, $\rk_1$ acts as a scalar, say $c\in \C$.
Assume $c\ne 0$.
Take a nonzero vector $v\in W_{(0,\alpha)}$.
Note that  $\rd_{-1,-n}\rd_{-1,n}v\in W_{(2,\alpha)}$ for $n\in \Z_+$ and $\dim W_{(2,\alpha)}<\infty$.
We claim that $\rd_{-1,-n}\rd_{-1,n}v$ for $n=1,2,\dots$ are linearly independent,
which is a contradiction. Assume
$\sum_{n\ge 1} a_n \rd_{-1,-n}\rd_{-1,n}v=0$ with $a_n\in \C$. Notice that for $r,s\in \Z$,
from  the Lie bracket relations \eqref{bre3} we have
\begin{align*}
\rk_{1,r}v\ (\in W_{(-1)})=0,\quad [\rd_{-1,r},\rd_{-1,s}]=0\quad \te{and}\quad
[\rk_{1,r},\rd_{-1,s}]=r\delta_{r+s,0}\rk_1.
\end{align*}
Using these facts,  for $m\ge 1$, we get
\begin{align*}
0=\rk_{1,m}\rk_{1,-m}\left(\sum_{n\ge 1} a_n \rd_{-1,-n}\rd_{-1,n}v\right)=a_mm^2c^2v,
\end{align*}
which implies $a_m=0$.
This proves the claim. Therefore, we must have $c=0$, concluding the proof.
\end{proof}

\begin{lemt}\label{lem:nohwm}
There does not exist a positive integer $N$ such that either $W_{(0,\alpha+n)}=0$ for $n>N$
or $W_{(0,\alpha+n)}=0$ for $n<-N$.
\end{lemt}

\begin{proof} We here prove that there does not exist $N\in \Z_+$ such that $W_{(0,\alpha+n)}=0$ for $n>N$,
while the proof for the other case is similar.
Assume by contradiction  that there exists $N\in \Z_+$ such that $W_{(0,\alpha+N)}\ne 0$
and $W_{(0,\alpha+n)}=0$ for $n>N$.
Take a nonzero vector $v\in W_{(0,\alpha+N)}$. Note that
$\rk_{-1,n}\rk_{0,-n}v\in W_{(1,\alpha+N)}$ for $n\in \Z_+$, where $\dim W_{(1,\alpha+N)}<\infty$.
We now show that these vectors are linearly independent, to get a contradiction.
Assume that $\sum_{n\ge 1}a_n\rk_{-1,n}\rk_{0,-n}v=0$ with $a_n\in \C$.
Notice that  $\rd_{1,-m}\rk_{0,-n}v\in W_{(-1)}=0$. On the other hand, from the Lie bracket relations \eqref{bre3},
using Lemma \ref{lem:k1=0}  we have
$$[\rd_{1,r},\rk_{-1,s}]=2(r+s)\rk_{0,r+s}+2\ell \delta_{r+s,0}\quad \text{ for }r,s\in \Z.$$
Then we get
\begin{align}\label{d(1,-m)-sum}
\rd_{1,-m}\left(\sum_{n\ge 1}a_n\rk_{-1,n}\rk_{0,-n}v\right)
=\sum_{n\ge 1}2(n-m)a_n\rk_{0,n-m}\rk_{0,-n}+2\ell a_m\rk_{0,-m}v
\end{align}
for $m\in \Z_+$.
Notice that for $m, n\in \Z_+$ with $m\ne n$, we have
$$\rk_{0,n}v\in W_{(0,\alpha+N+n)}=0, \  \  \rd_{0,m}v=0,\ \   \rk_{0,n-m}\rk_{0,m-n}v\ (=\rk_{0,m-n} \rk_{0,n-m}v)=0. $$
Using this and \eqref{bre3} we get
\begin{align*}
&\rd_{0,m}\rk_{0,n-m}\rk_{0,-n}v
=(n\rk_{0,n}+\rk_{0,n-m}\rd_{0,m})\rk_{0,-n}v\\
&\  =n\rk_{0,-n}\rk_{0,n}v+
\rk_{0,n-m}((m-n)\rk_{0,m-n}+\rk_{0,-n}\rd_{0,m})v=0
\end{align*}
for $m,n\in \Z_+$ with $m\ne n$.
Then by (\ref{d(1,-m)-sum}) we obtain
\begin{align*}
0=\rd_{0,m}\rd_{1,-m}\left(\sum_{n\ge 1}a_n\rk_{-1,n}\rk_{0,-n}v\right)=\rd_{0,m}(2\ell a_m\rk_{0,-m}v)
=a_m2\ell^2 v,
\end{align*}
which implies $a_m=0$. This proves the linear independence, as desired.
\end{proof}

Fix a basis $\{v_1,\dots,v_r\}$ of $W_{(0,\alpha)}$. For $1\le i\le r$, $m\in \Z$, set
$v_i(m)=\frac{1}{\ell}(t_1^m\rk_0) v_i\in W_{(0,\alpha+m)}$.
With Lemmas \ref{lem:k1=0} and \ref{lem:nohwm}, it follows from (the proof of) \cite[Theorem 3.1]{JM} that
 $\{v_i(m)\mid i=1,\dots,r, m\in \Z\}$ is a basis of $W_{(0)}$ and
\begin{align*}
\frac{1}{\ell}(t_1^{m}\rk_0)v_i(n)=v_i(m+n)\quad \text{ for }1\le i\le r,\ m,n\in \Z.
\end{align*}
From this we see that $W_{(0)}$ is an irreducible jet module for
$\mathcal L_{(0)}/\C\rk_1\cong \mathrm{Der}\ \C[t^{\pm 1}]\ltimes \(\C[t^{\pm 1}]\ot \ff\)$ in the sense of \cite{B}.
 Then by \cite[Theorem 4]{B} the $\mathcal L_{(0)}$-module $W_{(0)}$ is isomorphic to
 $T_{\ell,\lambda,\al,\beta}$ for some $\lambda\in P_+,\  \al,\beta\in \C$.
Consequently, $W$ as an (irreducible) $\wh{\mathfrak t}(\fg,\mu)^o$-module is isomorphic to
$L_{\wh{\mathfrak t}(\fg,\mu)^o}(T_{\ell,\lambda,\al,\beta})$.
This completes the proof of Theorem \ref{thm:irrVmod}.

\section{Realization of a class of irreducible $\wh{\mathfrak t}(\fg,\mu)$-modules}

In this section, we give a realization of a class of irreducible $\wh{\mathfrak t}(\fg,\mu)$-modules
similar to the realization of irreducible $\wh{\mathfrak t}(\fg,\mu)^{o}$-modules in Section 5, which recovers a
result of \cite{B2}.
To achieve this goal, by using a result of Zhu (and Huang) and a result of Lepowsky,
we show that for a general conformal vertex algebra $V$,
there is a canonical $\phi$-coordinated $V$-module structure on any $V$-module $W$.

We begin by defining a notion of a conformal vertex algebra.
Roughly speaking, conformal vertex algebras are slight generalizations of vertex operator algebras in the sense of
\cite{FLM} and \cite{FHL}. A {\em conformal vertex algebra} is a vertex algebra $V$ equipped with a vector $\omega$,
called the {\em conformal vector,} such that
$$[L(m),L(n)]=(m-n)L(m+n)+\frac{1}{12}(m^3-m)\delta_{m+n,0}c$$
for $m,n\in \Z$, where $Y(\omega,z)=\sum_{n\in \Z}L(n)z^{-n-2}$ and $c\in \C$, called the {\em central charge,} and such that
$$V=\oplus_{n\in \Z}V_{(n)},\ \  \text{where }V_{(n)}=\{ v\in V\, |\, L(0)v=nv\},$$
$$Y(L(-1)v,z)=\frac{d}{dz}Y(v,z)\quad\text{ for }v\in V,$$
and for every $v\in V$,  there exists a positive integer $N$ such that
$$L(n_1)\cdots L(n_r)v=0$$
for any positive integers $n_1,\dots,n_r$ with $r>N$.

\begin{prpt}\label{voa-module-phi-module}
Suppose that $V$ is a conformal vertex algebra with conformal vector $\omega$ of central charge $c$.
Let $(W,Y_W)$ be any $V$-module.
Then there exists a $\phi$-coordinated $V$-module structure $Y_{W}^{\phi}(\cdot,z)$ on $W$ such that
\begin{eqnarray}
&&Y_{W}^{\phi}(a,z)=Y_{W}(z^{L(0)}a,z)\quad \text{ for }a\in P(V),\\
&&Y_{W}^{\phi}(\omega,z)=z^2Y_{W}(\omega,z)-\frac{1}{24}c,
\end{eqnarray}
where $P(V)=\{ v\in V\, |\, L(n)v=0\  \text{ for all }n\ge 1\}$, the space of primary vectors.
\end{prpt}

\begin{proof} Set $\wt{\omega}=\omega-\frac{1}{24}c{\bf 1}$, and for $v\in V$ define
$$Y[v,z]=Y(e^{zL(0)}v,e^{z}-1).$$
It was proved by Zhu (see \cite{Zhu}, Theorem 4.2.1; cf. \cite{Lep1}) that $(V,Y[\cdot,z],{\bf 1},\wt{\omega})$
 is a conformal vertex algebra.
Furthermore, it was proved by Zhu and by Huang (see \cite{Zhu}, Theorem 4.2.2, \cite{Huang})
that the conformal vertex algebra $(V,Y[\cdot,z],{\bf 1},\wt{\omega})$ is
isomorphic to $(V,Y,{\bf 1},\omega)$, where an isomorphism $T$ from $(V,Y,{\bf 1},\omega)$ to
$(V,Y[\cdot,z],{\bf 1},\wt{\omega})$ was constructed explicitly.
From \cite{Zhu}, we have
\begin{eqnarray*}
T({\bf 1})={\bf 1}, \quad T(\omega)=\wt{\omega}=\omega-\frac{1}{24}c{\bf 1},\  \text{and }\
T(a)=a  \  \text{ for }a\in P(V).
\end{eqnarray*}

Let $(W,Y_{W})$ be a $V$-module. For $v\in V$, set
$$X_{W}(v,z)=Y_W(z^{L(0)}v,z).$$
It was essentially proved by Lepowsky  (see \cite{Lep2, Lep3};  \cite{li-F})
that $(W,X_W)$ is a $\phi$-coordinated module for the vertex algebra $(V,Y[\cdot,z],{\bf 1})$.
With the aforementioned vertex algebra isomorphism $T$, it follows that
there is a $\phi$-coordinated $V$-module structure $Y_{W}^{\phi}(\cdot,z)$ on $W$, where
\begin{align}\label{module-phi-module}
Y_{W}^{\phi}(v,z)=X_W(T(v),z)=Y_W(z^{L(0)}T(v),z)\quad \text{ for }v\in V.
\end{align}
For $a\in P(V)$, we have
$$Y_W^{\phi}(a,z)=Y_W(z^{L(0)}T(a),z)=Y_W(z^{L(0)}a,z).$$
On the other hand, we have
$$Y_{W}^{\phi}(\omega,z)=Y_W(z^{L(0)}\wt{\omega},z)
=Y_{W}\left(z^2\omega-\frac{1}{24}c{\bf 1},z\right)=z^2Y_{W}(\omega,z)-\frac{1}{24}c.$$
This completes the proof.
\end{proof}

Let $(V,Y,{\bf 1},\omega)$ be a conformal vertex algebra of central charge $c$.
Recall the conformal vertex  algebra  $(V,Y[\cdot,z],{\bf 1},\wt{\omega})$
defined in the proof of Proposition \ref{voa-module-phi-module}.
For $a\in V$, write
\begin{align}
Y[a,z]=\sum_{m\in \Z}a[m]z^{-m-1}.
\end{align}
When $a$ is homogeneous, we have (see \cite{Zhu})
\begin{align}\label{a[m]}
a[m]={\rm Res}_zY(a,z)(\log (1+z))^m(1+z)^{{\rm wt}a-1},
\end{align}
where
\begin{align*}
&\log (1+z)=\sum_{n\ge 1}(-1)^{n-1}\frac{1}{n}z^n=z-\frac{1}{2}z^2+\frac{1}{3}z^3-\frac{1}{4}z^4+\cdots,\\
&(1+z)^{{\rm wt}a-1}=\sum_{i\ge 0}\binom{{\rm wt}a-1}{i}z^i,
\end{align*}
which are elements of $\C[[z]]\subset \C((z))$, and
 $(\log (1+z))^{-1}$ denotes the inverse of $\log (1+z)$ in the field $\C((z))$. We have
\begin{align}
(\log (1+z))^{-1}
=z^{-1}\left(1+\frac{1}{2}z-\frac{1}{12}z^2+\frac{1}{24}z^3+O(z^4)\right),
\end{align}
\begin{align}
(\log (1+z))^{-2}=z^{-2}\left(1+z+\frac{1}{12}z^2+O(z^4)\right),
\end{align}
\begin{align}
(\log (1+z))^{-1}(1+z)^{-1}=z^{-1}\left(1-\frac{1}{2}z+\frac{5}{12}z^2-\frac{3}{8}z^3+O(z^4)\right),
\end{align}
and
\begin{align}
(\log (1+z))^{-2}(1+z)^{-1}=z^{-2}\left(1+\frac{1}{12}z^2-\frac{1}{12}z^3+O(z^4)\right).
\end{align}

\begin{remt}\label{three-cases}
{\em  Here, we list a few cases we need:
Case 1 with ${\rm wt} a=0$:
\begin{align}
&a[-1]=a_{-1}-\frac{1}{2}a_{0}+\frac{5}{12}a_{1}-\frac{3}{8}a_{2}+\cdots,\\
&a[-2]=a_{-2}+\frac{1}{12}a_{0}-\frac{1}{12}a_{1}+\cdots.
\end{align}

Case 2 with ${\rm wt} a=1$:
\begin{align}
a[-1]=a_{-1}+\frac{1}{2}a_0-\frac{1}{12}a_1+\frac{1}{24}a_2+\cdots.
\end{align}

Case 3  with ${\rm wt} a=2$:
\begin{align}
a[-1]=a_{-1}+\frac{3}{2}a_0+\frac{5}{12}a_1-\frac{1}{24}a_2+\cdots.
\end{align}}
\end{remt}

Note that in the construction of the irreducible module
$L_{\wh\fg\rtimes \mathsf{Vir}}(\ell,24\mu\ell-2,\lambda,\beta)$ for $\wh\fg\rtimes \mathsf{Vir}$,
with $L_{\fg}(\lambda)$ replaced by a general irreducible $\fg$-module $U$
we still get an irreducible $\wh\fg\rtimes \mathsf{Vir}$-module.
Denote this module by $L_{\wh\fg\rtimes \mathsf{Vir}}(\ell,24\mu\ell-2,U,\beta)$.
We see that for any $\alpha\in \C$, the tensor product space
$L_{\wh\fg\rtimes \mathsf{Vir}}(\ell,24\mu\ell-2,U,\beta)\ot V_{(H,L)}(\al)$
 is naturally a module for the conformal vertex algebra $V_{\wh\fg\rtimes \mathsf{Vir}}(\ell,24\mu\ell-2)\ot V_{(H,L)}$.
We have (cf. \cite{B2}):

\begin{thm}
Let $\ell,\alpha, \beta\in \C$ with $\ell\ne 0$ and let $U$ be an irreducible $\fg$-module.
Then the $V_{\wh\fg\rtimes \mathsf{Vir}}(\ell,24\mu\ell-2)\ot V_{(H,L)}$-module
$$L_{\wh\fg\rtimes \mathsf{Vir}}(\ell,24\mu\ell-2,U,\beta)\ot V_{(H,L)}(\al)$$
is an irreducible $\wh{\mathfrak t}(\fg,\mu)$-module with
\begin{align*}
&(t_1^m\ot u)[z]=zY_W(u,z)Y_W(e^{m\bfk},z),\quad
\rk_1[z]=\ell zY_W(\bfk,z),\quad \rd_1[z]=zY_W(\bfd,z),\\
& \rK_n[z] =\frac{\ell}{n}Y_W(e^{n\bfk},z),
\end{align*}
\begin{align*}
\rD_n[z]&=n z^2:Y_W(\omega,z)Y_{W}(e^{n\bfk},z):+nz\frac{d}{dz}Y_{W}(e^{n\bfk},z)-\frac{1}{24}ncY_{W}(e^{n\bfk},z)\\
&-z\frac{d}{dz}z:Y_W(\bfd,z) Y_{W}(e^{n\bfk},z):
+n(\ell\mu-1)\left(z\frac{d}{dz}Y_W(e^{n\bfk},z)\right)Y_{W}(e^{n\bfk},z)
\end{align*}
for $m\in \Z,\ u\in \fg,\  n\in \Z^{\times}$.
\end{thm}

\begin{proof}  Set
$$V=V_{\wh\fg\rtimes \mathsf{Vir}}(\ell,24\mu\ell-2)\ot V_{(H,L)},\quad
W=L_{\wh\fg\rtimes \mathsf{Vir}}(\ell,24\mu\ell-2,U,\beta)\ot V_{(H,L)}(\al).$$
With $W$ a $V$-module, by Proposition \ref{voa-module-phi-module},
$(W,Y_W^{\phi})$ is a $\phi$-coordinated $V$-module.
Recall from Corollary \ref{cor:vahom} the vertex algebra homomorphism
$\Theta:\  V_{\wh{\mathfrak t}(\fg,\mu)^{o}}(\ell)\rightarrow V$, where
 \begin{align*}
&\Theta(t_1^m\ot u)=u\ot e^{m\bfk},\quad \Theta(\rk_1)=\ell\bfk,\quad \Theta(\rd_1)=\bfd,\quad \Theta(\rK_n)
=\frac{\ell}{n}e^{n\bfk},\\
 &\Theta(\rD_n)=
nL(-2)e^{n\bfk}-L(-1)(\bfd_{-1}e^{n\bfk})+n(\mu\ell-1)(n\bfk_{-2}e^{n\bfk})
\end{align*}
for $u\in \fg,\ m\in \Z,\  n\in \Z^\times$.  Then
$W$ becomes a $\phi$-coordinated $V_{\wh{\mathfrak t}(\fg,\mu)^{o}}(\ell)$-module via $\Theta$.
Furthermore, by Theorem \ref{thm:main1} $W$ is a $\wh{\mathfrak t}(\fg,\mu)$-module of level $\ell$
with
$$v[z]=Y_W^{\phi}(\Theta(v),z)=Y_W(z^{L(0)}T\Theta(v),z)\quad \text{  for }
v\in A_{\fg}\subset V_{\wh{\mathfrak t}(\fg,\mu)^{o}}(\ell),$$
where $T$ is an isomorphism from $(V,Y,{\bf 1},\omega)$ to
$(V,Y[\cdot,z],{\bf 1},\wt{\omega})$, described in the proof of Proposition  \ref{voa-module-phi-module}.
Note that
$$ u\ot e^{m\bfk},\  \bfk,\   \bfd,\  e^{n\bfk}\in P(V)\  \text{ with }{\rm wt}\, (u\ot e^{m\bfk})={\rm wt}\, \bfk={\rm wt}\, \bfd =1,\
{\rm wt}\, e^{n\bfk}=0$$
for $u\in \fg,\ m\in \Z,\  n\in \Z^\times$ as above, so we have
$$T(u\ot e^{m\bfk})=u\ot e^{m\bfk},\  T(\bfk)=\bfk,\   T(\bfd)=\bfd,\  T(e^{n\bfk})=e^{n\bfk}$$
and
\begin{align*}
&(t_1^m\ot u)[z]=Y_{W}(z^{L(0)}(u\ot e^{m\bfk}),z)=zY_W(u,z)Y_W(e^{m\bfk},z),\\
&\rk_1[z]=\ell zY_W(\bfk,z),\quad
 \rd_1[z]=zY_W(\bfd,z),\quad  \rK_n[z] =\frac{\ell}{n}Y_W(e^{n\bfk},z).
\end{align*}

As $e^{n\bfk}\in P(V)$ and ${\rm wt}\, e^{n\bfk}=0$, by Remark \ref{three-cases} Case 3, we have
\begin{eqnarray*}
\omega[-1]e^{n\bfk}=\omega_{-1}e^{n\bfk}+\frac{3}{2}\omega_{0}e^{n\bfk}+\frac{5}{12}\omega_1e^{n\bfk}+\cdots
=L(-2)e^{n\bfk}+\frac{3}{2}L(-1)e^{n\bfk},
\end{eqnarray*}
so
\begin{eqnarray*}
&&T(\omega_{-1}e^{n\bfk})=T(\omega)[-1]T(e^{n\bfk})=(\omega-\frac{1}{24}c{\bf 1})[-1]e^{n\bfk}
=\omega[-1]e^{n\bfk}-\frac{1}{24}ce^{n\bfk}\\
&&\quad =L(-2)e^{n\bfk}+\frac{3}{2}L(-1)e^{n\bfk}-\frac{1}{24}ce^{n\bfk}.
\end{eqnarray*}
Then
\begin{eqnarray}
Y_{W}^{\phi}(\omega_{-1}e^{n\bfk},z)&=&Y_{W}(z^{L(0)}T(\omega_{-1}e^{n\bfk}),z)\nonumber\\
&=&Y_{W}\left(z^{2}L(-2)e^{n\bfk}+\frac{3}{2}zL(-1)e^{n\bfk}-\frac{1}{24}ce^{n\bfk},z\right)\nonumber\\
&=&z^2Y_{W}(\omega_{-1}e^{n\bfk},z)+\frac{3}{2}z\frac{d}{dz}Y_{W}(e^{n\bfk},z)-\frac{1}{24}cY_{W}(e^{n\bfk},z).
\end{eqnarray}

As $\bfd,e^{n\bfk}\in P(V)$, ${\rm wt}\bfd=1$,  ${\rm wt}e^{n\bfk}=0$, and
$\bfd_{i}e^{n\bfk}=n\delta_{i,0}e^{n\bfk}$ for $i\ge 0$, we have
\begin{eqnarray*}
\bfd[-1]b=\bfd_{-1}b+\frac{1}{2}\bfd_{0}b-\frac{1}{12}\bfd_1b+\frac{1}{24}\bfd_2b+\cdots=\bfd_{-1}b+\frac{1}{2}nb.
\end{eqnarray*}
Furthermore, we have
\begin{eqnarray*}
&&Y_{W}^{\phi}(\bfd_{-1}e^{n\bfk},z)=Y_{W}(z^{L(0)}T(\bfd_{-1}e^{n\bfk}),z)=Y_{W}(z^{L(0)}\bfd[-1]e^{n\bfk},z)\\
&&=zY_{W}(\bfd_{-1}e^{n\bfk},z)+\frac{1}{2}nY_W(e^{n\bfk},z),
\end{eqnarray*}
so
\begin{eqnarray}
Y_{W}^{\phi}(L(-1)\bfd_{-1}e^{n\bfk},z)&=&z\frac{d}{dz}Y_{W}^{\phi}(\bfd_{-1}e^{n\bfk},z)\\
&=&z\frac{d}{dz}zY_{W}(\bfd_{-1}e^{n\bfk},z)+\frac{1}{2}nz\frac{d}{dz}Y_W(e^{n\bfk},z).\nonumber
\end{eqnarray}
Set $a,b=e^{n\bfk}\in P(V)$, where ${\rm wt}\, a={\rm wt}\, b=0$ and $a_ib=0$ for $i\ge 0$. Then
$$a[-2]b=a_{-2}b+\frac{1}{12}a_0b-\frac{1}{12}a_1b+\cdots =a_{-2}b,$$
and
\begin{eqnarray}
\mbox{}\quad\quad  &&Y_{W}^{\phi}(a_{-2}b,z)=Y_{W}(z^{L(0)}T(a_{-2}b),z)=Y_{W}(z^{L(0)}a[-2]b,z)\\
&&=zY_{W}(a_{-2}b,z)=\left(z\frac{d}{dz}Y_W(a,z)\right)Y_{W}(b,z).\nonumber
\end{eqnarray}
Using all the facts above,   for $n\in \Z^{\times}$ we obtain
\begin{align*}
&D_n[z]\\
=\  &nY_{W}^{\phi}(\omega_{-1}e^{n\bfk},z)-Y_{W}^{\phi}(L(-1)\bfd_{-1}e^{n\bfk},z)+n(\ell\mu-1)Y_{W}^{\phi}((n\bfk)_{-2}e^{n\bfk},z)\\
=\ &n z^2Y_{W}(\omega_{-1}e^{n\bfk},z)+\frac{3}{2}nz\frac{d}{dz}Y_{W}(e^{n\bfk},z)-\frac{1}{24}ncY_{W}(e^{n\bfk},z)\\
&-z\frac{d}{dz}zY_{W}(\bfd_{-1}e^{n\bfk},z)-\frac{1}{2}nz\frac{d}{dz}Y_W(e^{n\bfk},z)\\
&+n(\ell\mu-1)\left(z\frac{d}{dz}Y_W(e^{n\bfk},z)\right)Y_{W}(e^{n\bfk},z)\\
=\ &n z^2:Y_W(\omega,z)Y_{W}(e^{n\bfk},z):+nz\frac{d}{dz}Y_{W}(e^{n\bfk},z)-\frac{1}{24}ncY_{W}(e^{n\bfk},z)\\
&-z\frac{d}{dz}z:Y_W(\bfd,z) Y_{W}(e^{n\bfk},z):
+n(\ell\mu-1)\left(z\frac{d}{dz}Y_W(e^{n\bfk},z)\right)Y_{W}(e^{n\bfk},z).
\end{align*}
This completes the proof.
\end{proof}

%

\end{document}